\documentclass[a4paper,11pt]{article}

\usepackage[T1]{fontenc}
\usepackage{lmodern}

\usepackage{dsfont}

\usepackage[latin1]{inputenc}
\usepackage{amsmath}
\usepackage{amsthm}
\usepackage{amssymb}
\usepackage{mathrsfs}
\usepackage{graphicx}
\usepackage[all]{xy}
\usepackage{hyperref}

\usepackage{tikz}

\usepackage{makeidx}

\usepackage{stmaryrd}
\usepackage{caption}

\usepackage{abstract} 

\newtheorem{thm}{Theorem}[section]
\newtheorem{cor}[thm]{Corollary}
\newtheorem{claim}[thm]{Claim}
\newtheorem{fact}[thm]{Fact}

\newtheorem{lemma}[thm]{Lemma}
\newtheorem{prop}[thm]{Proposition}

\theoremstyle{definition}
\newtheorem{definition}[thm]{Definition}
\newtheorem{ex}[thm]{Example}
\newtheorem{remark}[thm]{Remark}
\newtheorem{question}[thm]{Question}

\newtheorem{conj}[thm]{Conjecture}

\newtheorem*{theorem*}{Theorem}

\newcommand{\cut}{\textup{cut}^{\delta}}
\newcommand{\eps}{\epsilon}
\newcommand{\dl}{\textup{DL}}
\newcommand{\ori}{\partial_{0}}
\newcommand{\ter}{\partial_{1}}

\usepackage{scalerel}[2014/03/10]
\usepackage[usestackEOL]{stackengine}
\def\dashint{\,\ThisStyle{\ensurestackMath{%
  \stackinset{c}{.2\LMpt}{c}{.5\LMpt}{\SavedStyle-}{\SavedStyle\phantom{\int}}}%
  \setbox0=\hbox{$\SavedStyle\int\,$}\kern-\wd0}\int}

\def\rquotient#1#2{%
	\makeatletter
	\raise.3ex\hbox{$#1$}/\lower.3ex\hbox{$#2$}%
	\makeatother
}	

\makeatletter
\newcommand{\subjclass}[2][2010]{%
	\let\@oldtitle\@title%
	\gdef\@title{\@oldtitle\footnotetext{#1 \emph{Mathematics subject classification.} #2}}%
}
\newcommand{\keywords}[1]{%
	\let\@@oldtitle\@title%
	\gdef\@title{\@@oldtitle\footnotetext{\emph{Key words and phrases.} #1.}}%
}
\makeatother

\newcommand{\Address}{{
		\bigskip
		\small
		
		\textsc{Max Planck Institute for Mathematics\\ 
Vivatsgasse 7\\
53111 Bonn (Germany)}\par\nopagebreak
		\textit{E-mail address}: \texttt{bensaid@mpim-bonn.mpg.de}
  
  	\bigskip
		\small
		
		\textsc{University of Montpellier\\ 
Institut Math\'ematiques Alexander Grothendieck\\
Place Eug\`ene Bataillon\\
34090 Montpellier (France)}\par\nopagebreak
		\textit{E-mail address}: \texttt{anthony.genevois@umontpellier.fr}
  
    \bigskip
		\small
		
		\textsc{University of Paris-Cit\'e\\ 
Institut de Math\'ematiques de Jussieu-Paris Rive Gauche\\
Place Aur\'elie Nemours\\
75013 Paris (France)}\par\nopagebreak
		\textit{E-mail address}: \texttt{romain.tessera@imj-prg.fr}

}}

\makeindex

\title{Coarse separation and large-scale geometry of wreath products}
\date{\today}
\author{Oussama Bensaid, Anthony Genevois, and Romain Tessera}


\begin{document}

\maketitle

\begin{abstract}
In this article, we introduce and study a natural notion of coarse separation for metric spaces, with an emphasis on coarse separation by subspaces of polynomial or subexponential growth. For instance, we show that symmetric spaces of non-compact type different from $\mathbb{H}_\mathbb{R}^2$ and thick Euclidean buildings of rank $\geq 2$ cannot be coarsely separated by subspaces of subexponential growth; and that a connected nilpotent Lie group of growth degree $D \geq 2$ cannot be coarsely separated by a subspace of polynomial degree $\leq D-2$. We apply these results to the large-scale geometry of amalgamated free products and wreath products. The latter application is based on an Embedding Theorem that generalises previous work of the last two authors, and which is of independent interest. We also discuss some further applications to the distortion of coarse embeddings between certain metric spaces.
\end{abstract}

\tableofcontents

\section{Introduction}

We introduce a natural notion of coarse separation for metric spaces, and we apply it to study the large-scale geometry of wreath products (see \S \ref{secIntro:wreath}). We also discuss some further applications to amalgamated free products and to the distortion of coarse embeddings between certain metric spaces.

\subsection{Coarse separation}

\noindent
Recall that a connected graph is one-ended if the complement of any bounded subset of vertices admits  exactly one unbounded connected 
component. In other words, such a graph is one-ended if it cannot be separated in some suitable sense by a bounded subset. The first goal of this article is to generalise this property, or rather a \emph{uniform} version of it, to a notion of coarse separation  by certain (infinite) subsets. 

\medskip \noindent
For the sake of simplicity, we shall restrict in this introduction to graphs with bounded degree, we refer the reader to our preliminary section \S \ref{sec:Prelim} for definitions in a more general setting. The following is a simplified definition, we refer to Definition~\ref{weak separation} for a precise one. Given a subset $A\subset X$ in a metric space $X$, we shall denote $A^{+R}$ the $R$-neighbourhood of $A$, i.e.\ the set of points at distance at most $R$ from $A$.

\begin{definition}\label{defIntro:coarseSep} 
Let $X$ be a graph with bounded degree and let $\mathcal{S}$ be a family of subsets of $X$. We say that $\mathcal{S}$ \emph{coarsely separates $X$} if there exists $L \geq 0$ such that, for all $D \geq 0$, there is $S\in \mathcal{S}$ such that $X \backslash S^{+L}$ contains at least two connected components with points at distance $\geq D$ from $S$. 
\end{definition}

\noindent
Note that a vertex-transitive graph $X$ is one-ended if and only if it is not coarsely separated by the family of its vertices (or equivalently by a single vertex). In the absence of some form of homogeneity, the situation is different as seen by the following example: a half-line is coarsely separated by the family of its vertices, although it is one-ended. In other words, the property of being not coarsely separated by the family of its vertices can be seen as a \emph{uniform} version of one-endedness. 

\medskip \noindent
In this paper, we shall focus on a special case of families of subsets: those with subexponential growth. Given a graph $X$, we denote by $B_X(x,R)$ the ball of radius $R$ centred at the vertex $x.$

\begin{definition}
    Given a bounded degree graph $X$ and a family of subsets $\mathcal{S}$, the volume growth of $\mathcal{S}$ is defined as 
    \[V_{\mathcal{S}}(R)=\sup_{s\in S,\ S\in \mathcal S}|B_X(s,R)\cap S|.\]
\end{definition}

\noindent
We shall denote $\mathfrak{M}_{exp}$ the class of bounded degree graphs such that any coarsely separating family of subsets must have exponential growth.
The class $\mathfrak{M}_{exp}$ is clearly invariant under quasi-isometry. This provides a cheap way to extend it to arbitrary metric spaces by requiring that such a metric space belongs to $\mathfrak{M}_{exp}$ if it is quasi-isometric to a graph in $\mathfrak{M}_{exp}$.

\medskip \noindent
Our first main contribution is to provide a wide family of examples of spaces in $\mathfrak{M}_{exp}$.

\begin{thm}\label{thmIntro:Symmetric/buildings}
The following spaces belong to $\mathfrak{M}_{exp}$, as well as any direct product with an arbitrary bounded degree graph:
\begin{itemize}
    \item symmetric spaces of non-compact type different from  $\mathbb{H}_{\mathbb{R}}^2$;
    \item thick Euclidean buildings with cocompact affine Weyl group, of rank $\geq 2$;
    \item Bourdon's hyperbolic buildings $I_{p,q}$, with $p \geq 5$ and $q\geq 3$.    
\end{itemize}
\end{thm}

\noindent
Note that the condition of cocompact affine Weyl group is satisfied for the Bruhat-Tits buildings. The restriction of higher rank for Euclidean buildings and the exclusion of $\mathbb{H}_{\mathbb{R}}^2$ are necessary: the hyperbolic plane is separated by any bi-infinite geodesic, while an infinite tree of valence $\geq 2$ is obviously not one-ended. 

\medskip \noindent
Note that it is not clear a priori that belonging to $\mathfrak{M}_{exp}$ is stable under direct product with any bounded degree graph. The fact that it does in our case is a consequence of our method of proof (see Theorem \ref{invariance by taking product}). 

\medskip \noindent
To the best of our knowledge, the mere fact that any of these spaces cannot be coarsely separated by a family of subsets of polynomial growth  is new. An interesting example of euclidean building of rank $2$ to which this theorem applies is the direct product of two $3$-regular trees.

\begin{thm}\label{thmIntro:horocyclic products}
    For $i=1,2$, let $Y_i$ be either a $(k+1)$-regular tree with $k \geq 2$ or a rank one symmetric space of non-compact type. The horocyclic product $Y_1 \bowtie Y_2$ belongs to~$\mathfrak{M}_{exp}$, as well as any direct product with any bounded degree graph.
\end{thm}

\noindent
This theorem provides examples of solvable groups that belong to $\mathfrak{M}_{exp}$: indeed, lamplighter groups, solvable Baumslag-Solitar groups, and the group $\mathrm{\mathrm{SOL}}$ (and its lattices) are respectively quasi-isometric to a horocyclic product of two trees, of a tree with  $\mathbb{H}_{\mathbb{R}}^2$, and of two copies of  $\mathbb{H}_{\mathbb{R}}^2$.

\medskip \noindent
We now turn to wider classes: for each $d>0$, we let $\mathfrak{M}_d$ the class of metric spaces which are quasi-isometric to a bounded degree graph such that any coarsely separating family of subsets must have growth at least $r^d$. Note that, for all $d\leq d'$, we have $\mathfrak{M}_{exp} \subset \mathfrak{M}_{d'}\subset \mathfrak{M}_d$.

\medskip \noindent
Our main result regarding the classes $\mathfrak{M}_d$ is the following.

\begin{thm}\label{thmIntro:nilpot}
Let $G$ be a connected nilpotent Lie group of growth degree $D\geq 2$. Then $G\in \mathfrak{M}_{D-1}$.  More generally, $G\times X\in \mathfrak{M}_{D-1}$ for any bounded degree connected graph~$X$.
\end{thm}

\noindent
Note that the case $D=2$ is essentially trivial, given that the growth of an infinite coarsely connected subset is at least linear. 
Although not really surprising, the theorem seems to be new even in the case of $\mathbb{R}^3$.

\medskip \noindent
Let us end this paragraph with a comparison with closely related existing notions of coarse separation.
For instance the one defined in \cite[\S 4.2]{Farb1996} coincides with ours when $\mathcal{S}$ is reduced to a single subspace. Definition \ref{defIntro:coarseSep} is also  related to one implicitly used by Papasoglu in \cite[Proposition 1.4.1.]{MR2153400}, where he defines separation with respect to an auxiliary space $Z$: his definition of ``$Z$ coarsely separates $X$'' being equivalent to the fact that, for all increasing proper functions $(\rho_1,\rho_2)$, $X$ is coarsely separated by the set of images of $(\rho_1,\rho_2)$-coarse embeddings from $Z$ to $X$. The point of his definition is that it is well suited to his study of separation of $X$ by quasi-lines. Our point of view is quite different, as we are interested in separating families that have in common a bound on their volume growth, but which are not bound to satisfy any coarse topological property.

\subsection{Large-scale geometry of wreath products of groups}\label{secIntro:wreath}

This paper originally arises from an effort to understand the large scale geometry of wreath products $K\wr H$ of finitely generated groups. These groups are generally called \emph{lamplighter groups}, and $K$ is called the \emph{lamp group} while $H$ is called the \emph{base group}. The case where the lamp group is finite has been the object of a previous article by the two last authors. To explain and motivate our results, let us recall (a simplified version of) the central result that they had obtained under this assumption. 

\begin{thm}[\cite{GTbig}]\label{thm:AnthRom}
    Let $G$ be a finitely presented one-ended group and let $\rho:G\to K\wr H$ be a coarse embedding, where $K$ is a finite group and $H$ is an arbitrary finitely generated group. Then $\rho(G)$ lies in a bounded neighbourhood of a coset of $H$.
\end{thm}

\noindent
This result leads to a complete classification up to quasi-isometry of lamplighter groups over finitely presented one-ended group with finite lamp groups (see \cite{GTbig}). In the above theorem, the assumption that $G$ is one-ended is directly linked to the assumption that $K$ is finite. Indeed, roughly speaking, the proof of this theorem consists in showing that, if the conclusion is not satisfied, then $\rho(G)$ is coarsely separated by its intersection with a coset of a subgroup of the form $\bigoplus_S K$, where $S$ is a finite subset of $H$. Such a subset is bounded, hence this contradicts the one-endedness of $G$. 

\medskip \noindent
In this article, we prove the following far reaching abstract generalisation of Theorem~\ref{thm:AnthRom}, in which lamp groups are no longer assumed to be finite (see Theorem~\ref{thm:FullEmbeddingThm} for a more precise statement).

\begin{thm}\label{thmIntro:FullEmbeddingThm}
Let $A,B$ be two finitely generated groups. There exist families $\mathcal{B}_1, \mathcal{B}_2, \ldots$ of uniformly quasi-isometrically embedded subspaces of powers of $A$ in $A \wr B$ such that the following holds. For all finitely presented group $G$ and coarse embedding $\rho : G \to A \wr B$, if $\rho(Z)$ is neither contained in a subset of $\bigcup_{k \geq 1} \mathcal{B}_k$ nor coarsely separated by $\mathcal{B}_k$ for some $k \geq 1$, then the image of $\rho$ must lie in the neighbourhood of some $B$-coset. Moreover, the size of this neighbourhood only depends on $A$, $B$, $G$, and the parameters of $\rho$.
\end{thm}

\noindent
Here is a corollary which can be fed with our Theorems \ref{thmIntro:Symmetric/buildings} or \ref{thmIntro:horocyclic products} to get concrete applications.

\begin{cor}\label{cor:EmbeddingThm}
Let $A,B$ be two finitely generated groups such that $A$ has subexponential growth. Let $G$ be a finitely presented group with exponential growth that belongs to $\mathfrak{M}_{exp}$. Then, for every coarse embedding $\rho : G \to A \wr B$, $\rho(G)$ must lie in a bounded neighbourhood of some $B$-coset. Moreover, the size of this neighbourhood only depends on $A$, $B$, $G$, and the parameters of $\rho$.
\end{cor}

\noindent
Note that any finitely generated group $G$ acting properly cocompactly on a space appearing in Theorems \ref{thmIntro:Symmetric/buildings} or \ref{thmIntro:horocyclic products} satisfies the two conditions of the previous theorem: finite presentation and exponential growth. 

\medskip \noindent
We now present applications to quasi-isometric classification of wreath products up to quasi-isometry. 
We start with a general result, which we deduce (with some substancial work) from Theorem \ref{thm:FullEmbeddingThm}.
\begin{thm}\label{thm:AppliQI}
Let $A_1,A_2,B_1,B_2$ be four finitely generated groups. Assume that 
\begin{itemize}
	\item $B_1$ (resp.\ $B_2$) does not quasi-isometrically embed in a power of $A_2$ (resp. $A_1$);
	\item for every quasi-isometrically embedded subspace $S \subset B_1$ (resp. $S \subset B_2$) that also quasi-isometrically embeds into a power of $A_2$ (resp. $A_1$), $S$ does not coarsely separate $B_1$ (resp. $B_2$).
\end{itemize}
Then every quasi-isometry $A_1\wr B_1 \to A_2 \wr B_2$ induces a quasi-isometry $B_1 \to B_2$, and there are quasi-isometric embeddings $A_1 \hookrightarrow A_2^p$ and $A_2 \hookrightarrow A_1^q$ for some $p,q \geq 1$. 
\end{thm}

\noindent
This theorem applies to the cases where 
$A_1$ and $A_2$ have subexponential growth and $B_1$ and $B_2$ belong to $\mathfrak{M}_{exp}$. Exploiting the existence of quasi-isometries $A_1 \hookrightarrow A_2^p$ and $A_2 \hookrightarrow A_1^q$, we obtain the following corollary. 

\begin{cor}\label{cor:AppliQI}
Let $A_1,A_2,B_1,B_2$ be four finitely generated groups such that $A_1$ is torsion-free step-$k$ nilpotent and such that $B_1$ and $B_2$ belong to $\mathfrak{M}_{exp}$.
If $A_1\wr B_1$ and  $A_2 \wr B_2$ are quasi-isometric, then so are $B_1$ and $B_2$, and $A_2$ has a finite-index torsion-free step-$k$ nilpotent subgroup. 
\end{cor}

\noindent
Let us conclude this section with a couple of explicit examples.

\begin{ex}
For every $k \geq 1$, let $N_k$ be the free step $k$ nilpotent group on two generators. The solvable groups $\mathrm{N}_p \wr \mathrm{BS}(1,n)$ and $\mathrm{N}_q \wr \mathrm{BS}(1,m)$ are quasi-isometric if and only if $p=q$ and $n,m$ are powers of a common number. 
\end{ex}

\begin{ex}
For all $n \geq 3$ and $m \geq 2$, let $B(n,m)$ denote the graph product over an $n$-cycle all whose vertex-groups are $\mathbb{Z}/m\mathbb{Z}$. For every uniform lattice $H$ in $\mathbb{H}^p$, the groups $\mathbb{Z} \wr B(n,m)$ and $\mathbb{Z} \wr H$ are quasi-isometric if and only if $p=m=2$ and $n \geq 5$. 
\end{ex}

\subsection{Large-scale geometry of amalgamated free products}\label{secIntro:trees}

Our notion of coarse separation is designed to provide obstruction to the existence of coarse embeddings to free products with amalgamations.
In particular, we have the following easy fact (see Theorem \ref{sep in graphs of groups general} for a more general statement).
Given a finitely generated group $G$ and a subgroup $H$, we define the relative growth of $H$ to be its growth with respect to the distance induced from a word distance in $G$.

\begin{prop}\label{sep in graphs of groups}
Let $G$ be the fundamental group of a finite graph of groups whose edge-groups have subexponential relative growth (resp.\ relative growth at most $r^d$), then any coarse embedding from a bounded degree graph in $\mathfrak{M}_{exp}$ (resp.\ in $\mathfrak{M}_{d+1}$)  to $G$ must land in a bounded neighbourhood of a left coset of some vertex-group.
\end{prop}

\noindent
This applies in particular to free products with amalgamation over non-distorted subgroups of subexponential growth. 

\medskip \noindent
\begin{cor}
Let $X$ be one of spaces listed in Theorems \ref{thmIntro:Symmetric/buildings} and \ref{thmIntro:horocyclic products}. Then any coarse embedding from $X$ to a free product $A_1\ast_{B_1} A_2\ldots \ast_{B_k} A_k$, where the $B_i$ are non-distorted and have subexponential growth, must land in a bounded neighbourhood of a coset of some~$A_j$. 
\end{cor}

\noindent
Similarly, applying Theorem \ref{thmIntro:nilpot} yields the following corollary:

\begin{cor}
Assume $G$ is finitely generated group (or a connected Lie group) of polynomial growth of degree at least $d$. Then any coarse embedding from $G$ to an amalgamated free product $A_1\ast_{B_1} A_2\ldots \ast_{B_{k-1}} A_k$, where the $B_i$ are non-distorted and of polynomial growth of degree at most $d-2$, must land in a bounded neighbourhood of a coset of some~$A_j$. 
\end{cor}

\noindent
Here is a concrete situation where the last result applies:

\begin{cor}\label{corIntro: chordal}
Let $\Gamma$ be a finite chordal graph and $X$ a bounded degree graph not coarsely separable by a family of subspaces of growth $O(r^{\mathrm{clique}(\Gamma)})$ (e.g.\ $X\in \mathfrak{M}_{d}$ for $d<\mathrm{clique}(\Gamma)$). Then there does not exist any coarse embedding of $X$ in the right-angled Artin group $A(\Gamma)$. 
\end{cor}

\noindent
For instance, $\mathbb{H}^3$ does not coarsely embed into $A(\Gamma)$. It is worth mentioning that it is proved in \cite{MR2422070} that $A(\Gamma)$ does not contain the fundamental group of a closed surface of genus $\geq 2$ as a subgroup, which suggests that the previous observation might also hold for $\mathbb{H}^2$. 

\medskip \noindent
Proposition \ref{sep in graphs of groups} has the following consequence for quasi-isometries.

\begin{cor}\label{corIntro:qi between graphs of groups}
Let $G_1$ and $G_2$ be finitely generated groups admitting graph of groups decompositions. Assume that there exists $n \in \mathbb{N^*}$ such that 
    \begin{itemize}
        \item all vertex-groups lie in $\mathfrak{M}_d$ (resp.\ $\mathfrak{M}_{exp}$),
        \item  all edge-groups  are finitely generated with relative growth $=o (r^{d-1})$ (resp.\ have subexponential relative growth),
    \end{itemize}
    then any quasi-isometry between $G_1$ and $G_2$ quasi-preserves the vertex-groups.
\end{cor}

\noindent
This should be compared to  \cite[Theorem 3.1]{papasoglu2002quasi}, \cite{papasoglu2007group}, and \cite[Theorem 1.6]{mosher2011quasi}. 
In those results, the conclusion is (much) stronger as the quasi-isometry quasi-preserves the edges groups as well. 
However, combining this simple corollary with Theorems~\ref{thmIntro:Symmetric/buildings} and~\ref{thmIntro:horocyclic products} yields the following concrete application, which is not covered by those quoted results: let $H_1,L_1,H_2,L_2$ be finitely generated groups, each one acting properly cocompactly on any of the spaces listed in those two theorems, and let $K_1$ and $K_2$ be two groups with subexponential growth. Corollary \ref{corIntro:qi between graphs of groups} applies to the pair of groups 
$G_1=(H_1\times K_1)\ast_{K_1}(L_1\times K_1)$  and $G_2=(H_2\times K_2)\ast_{K_2}(L_2\times K_2)$. Note that, for $i=1,2$, we have $G_i\simeq(H_i\ast L_i)\times K_i$, ensuring that $K_i$ is not distorted. 

\subsection{Applications to distortion of coarse embeddings}

Combining Theorems \ref{thmIntro:nilpot} with a coarse version of the Jordan separation theorem, we obtain:

\begin{thm}\label{thmIntro:lower bound distortionNilp}
Let $H$ and $G$ be connected simply connected nilpotent Lie groups of dimension $k$ and $k+1$ respectively. Suppose that their respective growth functions satisfy $\beta_H(r) \simeq r^p$, $\beta_G(r) \simeq r^n$, with $p \leq n-1$, and let $f : H \to G$ be a coarse embedding with control functions $\rho_-$ and $\rho_+$. Then $\rho_-(r) \lesssim r^{\frac{p}{n-1}}$. In particular, $f$ cannot be a quasi-isometric embedding if $p < n-1$.  
\end{thm}  

\noindent
The mere non-existence of a quasi-isometric embedding can alternatively be deduced from a famous result of Pansu, which implies that a quasi-isometric embedding between connected simply connected nilpotent Lie groups $f : H \to G$ gives rise to a continuous bi-Lipschitz morphism between their Carnot graded associated groups. His result actually provides a much more general obstruction to the existence of quasi-isometric embedding between nilpotent groups. However, our approach is more elementary and provides those quantitative estimates for coarse embeddings, which are new. Using instead Theorems~\ref{thmIntro:Symmetric/buildings} or~\ref{thmIntro:horocyclic products}, we obtain:

\begin{thm}\label{thmIntro:lower bound distortionExp}
Let $G$ be a connected simply connected nilpotent Lie group of dimension $k$ and $X$ a symmetric space of non-compact type distinct from $\mathbb{H}_{\mathbb{R}}^2$, or a horocyclic product of rank one symmetric spaces. If $f : G \to X$ is a coarse embedding with control functions $\rho_-$ and $\rho_+$, and $\dim X = k+1$, then $\rho_-(r) \lesssim \log(r)$. In particular, $f$ cannot be a quasi-isometric embedding.
\end{thm}

\noindent
A particular instance of this result is the case of a coarse embedding of $\mathbb R^2$ to $\mathrm{SOL}=\mathbb R^2\rtimes \mathbb R$. The fact that no quasi-isometric embedding from $\mathbb R^2$ to $\mathrm{SOL}$ exists can be deduced from the analysis of their asymptotic cones (the topological dimension of the asymptotic cone of $\mathrm{SOL}$ being $1$ by \cite{cornulierCones}). However the upper bound on $\rho_-$ is new, and sharp as the normal subgroup isomorphic to $\mathbb R^2$ of $\mathrm{SOL}$ is exponentially distorted.
Another interesting instance is the case of a coarse embedding\footnote{Once again, the fact that no quasi-isometry exists follows from an analysis of the asymptotic cones: indeed the asymptotic cone of $\mathbb{H}_{\mathbb{R}}^2\times \mathbb{H}_{\mathbb{R}}^2$ being a product of two real trees, it has topological dimension $2$.} from $\mathbb R^3$ to $\mathbb{H}_{\mathbb{R}}^2\times \mathbb{H}_{\mathbb{R}}^2$. Note that it is an open question whether such a coarse embedding exists. If it does, it cannot send any copy of $\mathbb{R}^2$ quasi-isometrically \cite[Theorem 1.6]{bensaid2022coarse}.

\subsection*{Overview of the proofs}
\subsubsection*{Coarse separation and persistent families}
To show that a graph of bounded degree $X$ belongs to $\mathfrak{M}_{exp}$ or $\mathfrak{M}_{D}$, the general idea is to look for a ``well connected'' family of subsets, which we ``move'' along a path from one connected component to the other until we get a big intersection with the separating set. More precisely, we first look for a family $\left\{ A_x(r) \right\}_{x\in X, r \geq 0}$ such that, for any $x,y \in X$ and any $r \geq 0$: $A_x(r) \subset B_X(x,4r)$, $| A_x(r)| = |A_y(r)|$, and, if $x$ and $y$ are adjacent vertices, then $| A_x(r) \cap A_y(r)| \geq \alpha |A(r)| $, where $\alpha >0$ is a fixed constant. $\left\{ A_x(r) \right\}_{x\in X, r \geq 0}$ is called an \emph{$\alpha$-persistent family}, see Definition \ref{persistent}.

\medskip \noindent
If a family $\mathcal{S}$ coarsely separates $X$, there exists $L \geq 0$ such that, for any $r \geq 0$, there exist $S \in \mathcal{S}$ and two balls $B(x,4r),B(y,4r)$ in different components of $X \backslash S^{+L}$. Following a path from $x$ to $y$, the intersection property guarantees the existence of a point $z$ such that any connected component of $A_z(r) \backslash S^{+L}$ has size $\leq (1-\frac{\alpha}{2}) | A_z(r)| $, see Figure \ref{fig:sketchseparation}. 
\begin{figure}[!ht]
    \centering
    \includegraphics[width=0.8\linewidth]{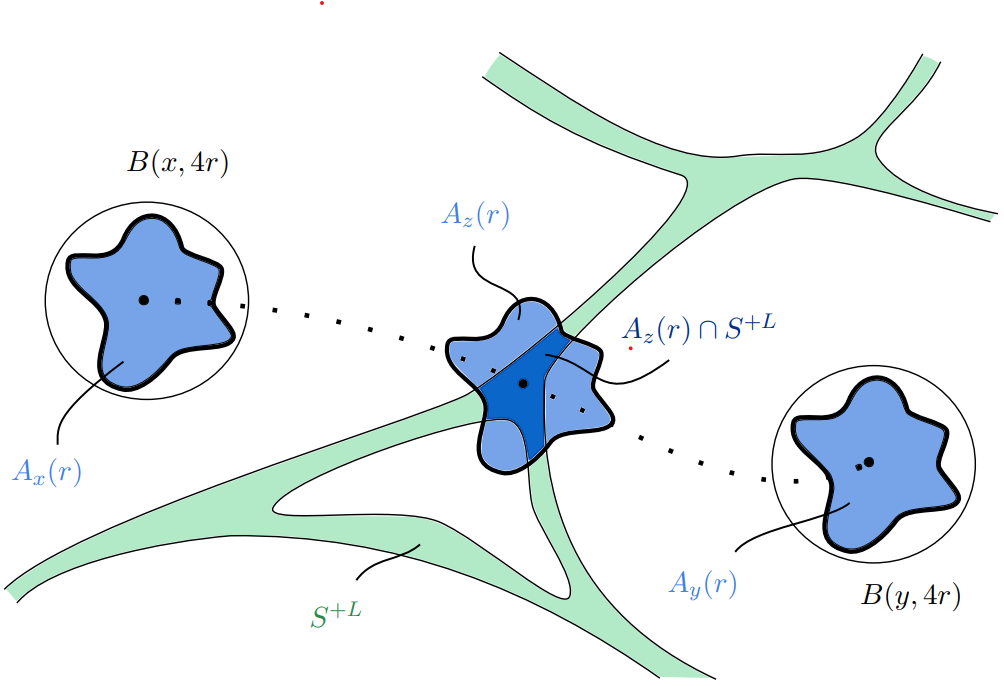}
    \caption{Persistent family and separation}
    \label{fig:sketchseparation}
\end{figure}
\\This is where the connectivity of the persistent family comes into play. If $K$ is a bounded subset, we denote by $\cut (K)$ the minimum size of a subset of $K$ that needs to be removed in order to cut it into smaller pieces of size $\leq \delta |K|$. Therefore, by the previous argument, the existence of an $\alpha$-persistent family $\left\{ A_x(r) \right\}_{x\in X, r \geq 0}$ implies that, for any $r>0$, there exists $z \in X$ such that
$$V_{\mathcal{S}}(r) \gtrsim \cut (A_z(r)), $$
where $\delta = 1-\frac{\alpha}{2}$. If moreover $\cut (A_z(r)) \gtrsim e^r$ (resp. $\cut (A_z(r)) \gtrsim r^D$), then $X \in \mathfrak{M}_{exp}$ (resp. $X \in \mathfrak{M}_{D}$). 

\medskip \noindent
In most cases, the persistent families are either balls, i.e.\ for groups of polynomial growth, or a constant thickening of spheres. Explicit persistent families with exponentially growing cuts are given in the following cases. For rank one symmetric spaces and Bourdon's hyperbolic buildings, and using techniques from \cite{hume2020poincare}, we pull-back the $1$-Poincar\'e inequality, satisfied in the boundary, to a constant thickening of spheres. This gives a lower bounded on their Cheeger constant, which in turn gives a lower bound on their cut, see Lemma \ref{persistent in rank 1} and Proposition \ref{Cut in rank 1}.  By a curve counting argument, we compute a lower bound of the Cheeger constant of spheres (more precisely of annuli) in products of trees, see Proposition \ref{persistent family of product trees}. For Diestel--Leader graphs, a lower bound of the Cheeger constant has already been computed for a family of subsets in \cite{hume2022poincare}, which turns out to be persistent, see Proposition \ref{persistent family of dl}. For higher rank symmetric spaces and thick Euclidean buildings, we exploit the fact that they admit nice embeddings of products of trees: any pair of points is contained in the image of a quasi-isometric embedding of a product of trees (with uniform constants) \cite{bensaidnguyen},\cite{fisher2018quasi}. Finally, the case of general horocyclic products $X \bowtie Y$ is dealt with in a similar way: we show that any pair of points can be joined by the images of three quasi-isometric embeddings (with uniform constants) of the Diestel--Leader graph. 
\subsubsection*{Applications}
The applications to graphs of groups follow from the main observation that the fundamental group of a graph of groups is coarsely separated by the cosets of its edge subgroups.

\medskip \noindent
The applications to wreath products are more involved. They require an Embedding Theorem as given by Theorem~\ref{thmIntro:FullEmbeddingThm}, which, roughly speaking, states that, given a coarse embedding $\rho$ from some finitely presented group $G$ to a wreath product $A \wr B$ of finitely generated groups, then either $\rho(G)$ is contained in a neighbourhood of a coset of $B$, or it is contained in a neighbourhood of a power of $A$, or it is separated by a family of subspaces of powers of $A$. Thus, given a coarse embedding or a quasi-isometry between two wreath products $\varphi : A_1 \wr B_1 \to A_2 \wr B_2$, if $B_1,B_2$ are finitely presented and if they cannot be contained in nor separated by subspaces of powers of $A_1,A_2$ respectively, then we know that $\varphi$ sends $B_1$-cosets at finite Hausdorff distance from $B_2$-cosets. Using the machinery of aptolicity introduced in \cite{GTbig} and developped further in \cite{GTlike}, we can then deduce precise information on the structure of our coarse embeddings and quasi-isometries.

\medskip \noindent
So Theorem~\ref{thmIntro:FullEmbeddingThm} is really our main contribution to the large-scale geometry of wreath products. The strategy is the following. Given a wreath product $A \wr B$, we can find truncated presentations of the form 
$$(\text{graph product over $\Gamma_n$ of copies of $A$}) \rtimes B$$
with a graph $\Gamma_n$ that looks more and more like a complete graph as $n \to + \infty$. Then every coarse embedding to $A \wr B$ from a coarsely simply connected metric space $Z$ (e.g.\ a finitely presented group) factors through a coarse embedding to one of these truncations. Geometric models can be constructed for these truncations, which turn out to admit natural wallspace structures. Then, if the image of $Z$ in $A \wr B$ cannot be separated by specific families of subspaces, then it cannot cross walls in an essential way, and so must be confined in some area, namely a $B$-coset. 

\medskip \noindent
Our proof follows the strategy used in \cite{GTbig}. The major difference comes from the geometric model for our truncations. Because our lamp groups may be infinite, our model is no longer a \emph{graph of pointed cliques in a quasi-median graph}, but a \emph{graph of pointed cliques in a quasi-median graph endowed with a coherent system of metrics} (using notions introduced in \cite{Qm}). This introduces technical difficulties which we overcome.

\subsection*{Organisation of the paper}

After a preliminary section \S \ref{sec:Prelim} devoted to our notion of coarse separation, our main theorems are proved in \S \ref{section:CoarseSep} (first \S \ref{sebsec:tools} presents the general tools and methods, while the remaining subsections are devoted to the proofs of the various cases). The two following relatively short sections \S \ref{section graph of groups} and \ref{sec:distortion} are respectively dedicated to the applications to graphs of groups and to the distortion of coarse embeddings. The important and relatively long section \S \ref{section:Wreath} is concerned with our results on wreath products of groups. Finally, the paper ends with \S \ref{section:questions} where we address a few open problems. 

\subsection*{Acknowledgements}
The first author is grateful to the Max-Planck Institute for Mathematics in Bonn for its financial support.

\section{Preliminaries}\label{sec:Prelim}
Let us first introduce the notations and definitions that will be used throughout the rest of the paper.
\subsection{Coarse separation}
Let $(X,d)$ be a metric space, $x \in X$, and $r \geq 0$. We denote by $B(x,r)$ the closed ball of radius $r$ centred at $x$ and by $S(x,r)$ the sphere of radius $r$ centred at $x$. Given $\alpha \geq 0$ and $A \subset X$, we denote the $\alpha$-neighbourhood of $A$ as 
$$A^{+ \alpha}: = \{x \in X \textup{ such that } d(x,A) \leq \alpha\}.$$
\begin{definition}\label{coarse connected}
Let $X$ be a metric space and $Y \subset X$ a subspace. Fix a constant $k>0$.
\begin{itemize}
	\item A \emph{$k$-path} (\emph{of length $n$}) in $Y$ is a sequence of points $x_0,\ldots, x_n \in Y$ satisfying $d(x_i,x_{i+1}) \leq k$ for every $0 \leq i \leq n-1$. 
	\item If any two points in $Y$ are connected by a $k$-path in $Y$, then $Y$ is \emph{$k$-coarsely connected}. 
	\item If $Y$ is a maximal $k$-coarsely connected subspace of $X$, then $Y$ is a $k$-coarsely connected component of $X$. Equivalently, $Y$ is an equivalent class with respect to the relation ``being connected by a $k$-path''. 
    \item $Y$ is \emph{$k$-coarsely geodesic} if there exists a map $\sigma : [0,+ \infty) \to [0,+ \infty)$ such that any two points $a,b \in Y$ are connected by a $k$-path in $Y$ of length $\leq \sigma(d(a,b))$. 
\end{itemize}
\end{definition}

\noindent
The following definition of \emph{coarse separation} is the central concept of the article.

\begin{definition}\label{weak separation}
Let $X$ be a metric space and $\mathcal{Z}$ a collection of subspaces. A coarsely connected subspace $Y \subset X$ is \emph{coarsely separated by $\mathcal{Z}$} if there exist $k,L \geq 0$ such that $Y$ is $k$-connected and such that, for every $D \geq 0$, there is some $Z \in \mathcal{Z}$ such that $Y \backslash Z^{+L}$ has at least two $k$-coarsely connected components with points at distance $\geq D$ from $Z$. We will call $L$ the \emph{thickening constant} of $\mathcal{Z}$.
\end{definition}

\noindent
The picture to keep in mind is when $X$ is a geodesic metric space and $Y=X$. Then $\mathcal{Z}$ coarsely separates $X$ if one can find an $L \geq 0$ such that, whatever large $D$ is, there is some $Z \in \mathcal{Z}$ such that $X\backslash Z^{+L}$ contains at least two connected components with points at least $D$ apart from $Z$. 

\medskip \noindent
For instance, a finitely generated group $G$ is one-ended if and only if its Cayley graphs (constructed from finite generating set) are not separated by their vertex-sets. More generally, $G$ is separated by a subgroup $H$ (more precisely, $\{H\}$ separates $G$) if and only if $(G,H)$ is multi-ended (see \cite[Theorem~2.3]{MR1347406}).

\medskip \noindent
As expected, being coarsely separated is naturally preserved by coarse equivalence. Namely:

\begin{lemma}\label{lem:CoarseSep}
Let $\rho : X \to M$ be a coarse equivalence between two metric spaces. Let $Y \subset X$ be a coarsely geodesic subspace and $\mathcal{Z}$ a collection of subspaces of $X$. If $\mathcal{Z}$ coarsely separates $Y$ in $X$, then $\rho(\mathcal{Z})$ coarsely separates $\rho(Y)$ in $M$. 
\end{lemma}

\begin{fact}\label{fact:CoarseSup}
Let $X$ be a metric space, $Y \subset X$ a coarsely geodesic subspace, and $\mathcal{Z}$ a collection of subspaces of $X$. If $\mathcal{Z}$ coarsely separates $Y$, then there exists $k_0>0$ such that, for every $k \geq k_0$, the following holds. First, $Y$ is $k$-coarsely geodesic. Next, there exists some $L_0 \geq 0$ such that, for every $D \geq 0$, there is some $Z \in \mathcal{Z}$ such that at least two $k$-coarsely connected components of $Y \backslash Z^{+L_0}$ contain points at distance $\geq D$ from~$Z$. 
\end{fact}

\begin{proof}
Because $Y$ is coarsely geodesic and $\mathcal{Z}$ coarsely separates $Y$, we know that there exist $k_0>0$ and $L\geq 0$ such that $Y$ is $k_0$-coarsely geodesic and, for every $D \geq 0$, there is some $Z \in \mathcal{Z}$ such that at least two $k_0$-coarsely connected components of $Y \backslash Z^{+L}$ contain points at distance $\geq D$ from $Z$. Now, fix a $k \geq k_0$ and set $L_0:=L+k_0 \ell$ where $\ell$ is such that any two points of $Y$ at distance $\leq k$ can be connected by a $k_0$-path of length $\leq \ell$. 

\medskip \noindent
Because $k \geq k_0$, $Y$ is clearly $k$-coarsely geodesic. Next, let $D\geq 0$. We know that there exist $a,b \in Y$ at distance $\geq D$ from $Z$ and lying in distinct $k_0$-coarsely connected components of $Y \backslash Z^{+L}$. Consider a $k$-path $x_0, \ldots, x_n$ in $Y$ connecting $a$ and $b$. We can extend this path to a $k_0$-path by connecting any two consecutive point by a $k_0$-path of length $\leq \ell$. Such an extended path must intersect $Z^{+L}$, say at some point $p$. Let $0 \leq i \leq n-1$ be the index such that $p$ belongs to the $k_0$-path connecting $x_i$ and $x_{i+1}$. Then $x_i$ necessarily belongs to $Z^{+L+k_0 \ell} = Z^{+L_0}$. Thus, $a$ and $b$ belong to distinct $k$-coarsely connected components of $Y\backslash Y^{+L_0}$, as desired.
\end{proof}

\begin{proof}[Proof of Lemma~\ref{lem:CoarseSep}.]
Fix a coarse inverse $\eta : M \to X$ of $\rho$. Say that $\eta \circ \rho$ lies at distance $\leq Q$ from the identity and that $\rho,\eta$ are $(\sigma_-,\sigma_+)-$, $(\tau_-,\tau_+)$-coarse maps for some continuous increasing maps $\sigma_\pm, \tau_\pm$. Let $k_0>0$ and $L_0 \geq 0$ denote the constants given by Fact~\ref{fact:CoarseSup} for the coarse separation of $Y$ by $\mathcal{Z}$. Let us fix various auxiliary constants that will be useful for us:
\begin{itemize}
	\item Set $k:= \max(\sigma_+(k_0),k_0)$ and $k':= \max (Q,\tau_+(k),k_0)$.
	\item Let $L'$ denote the constant given by Fact~\ref{fact:CoarseSup} for $k'\geq k_0$ and the coarse separation of $Y$ by $\mathcal{Z}$.
	\item Set $L:= Q+\tau_+(L')$.
\end{itemize}
Now, we are ready to verify that $\rho(\mathcal{Z})$ coarsely separates $\rho(Y)$. First, notice that $Y$ is $k$-coarsely geodesic since $k \geq k_0$. Next, let $D \geq 0$ be a constant. Fix some $D'$ sufficiently large so that $\sigma_-(D') \geq D$. We know that there exist $Z \in \mathcal{Z}$ and two points $a,b \in Y$ such that $d(a,Z),d(a,Z) \geq D'$ and such that $a,b$ belongs to distinct $k'$-coarsely connected components of $Y\backslash Z^{+L'}$. We want to deduce from this that $\rho(a),\rho(b)$ lie at distance $\geq D$ from $\rho(Z)^{+L}$ and in distinct $k$-coarsely connected components of $\rho(Y) \backslash \rho(Z)^{+L}$.

\medskip \noindent
The first assertion is clear since
$$d(\rho(a),\rho(Z)) \geq \sigma_-(d(a,Z)) \geq \sigma_-(D') \geq D,$$
and similarly for $\rho(b)$. Next, consider a $k$-path $x_0,\ldots, x_n$ connecting $\rho(a)$ and $\rho(b)$. Then $a,\eta(x_0),\ldots, \eta(x_n),b$ is a $\max(Q,\tau_+(k))$-path, and a fortiori a $k'$-path. Necessarily, such a path has to intersect $Z^{+L'}$, say $\eta(x_i) \in Z^{+L'}$ for some $0 \leq i \leq n$. Then
$$\rho(\eta(x_i)) \in \rho(Z^{+L'}) \subset \rho(Z)^{\tau_+(L')}, \text{ hence } x_i \in \rho(Z)^{+Q+\tau_+(L')} \subset \rho(Z)^{+L}.$$
This concludes the proof of our lemma.
\end{proof}

\noindent
Finally, we end this subsection by recording the following observation, which will be useful in Section~\ref{section:EmbeddingTheorem}.

\begin{lemma}\label{lem:SepButNotCoarse}
Let $X$ be a graph, $Y \subset X$ a coarsely geodesic subspace, and $\mathcal{Z}$ a collection of subspaces. If every subspace in $\mathcal{Z}$ separates $X$ but $\mathcal{Z}$ does not coarsely separate $Y$, then there exist a constant $L \geq 0$ such that, for every $Z \in \mathcal{Z}$, at most one connected component $C$ of $X \backslash Z$ satisfies $C \cap Y \nsubseteq Y^{+L}$. 
\end{lemma}

\begin{proof}
Suppose that the conclusion of our lemma does not hold, i.e.\ for every $i \geq 1$ there exist $Z_i \in \mathcal{Z}$, two connected components $A,B$ of $X \backslash Z_i$, and vertices $a_1,a_2,\ldots \in A$, $b_1,b_2, \ldots \in B$ satisfying $d(a_i,Z),d(b_i,Z) \geq i$. We want to prove that $\mathcal{Z}$ then coarsely separates $Y$.

\medskip \noindent
Fix an arbitrary $k >0$ such that $Y$ is $k$-coarsely connected, and fix some $D \geq 0$. For every $i>D+k$, the points $a_i$ and $b_i$ clearly lie at distance $>D$ from $Z_i^{+k}$. We claim that $a_i$ and $b_i$ belong to distinct $k$-coarsely connected components of $Y \backslash Z_i^{+k}$. So fix an arbitrary $k$-path between $a_i$ and $b_i$. Extend it to a path in $X$ by connecting any two consecutive vertices with a geodesic in $X$. Because $Z_i$ separates $X$, necessarily our extended path must intersect $Z_i$, say at some vertex $v$. Along our extended path, $v$ lies between two vertices, say $p$ and $q$, of our $k$-path. Then $d(p,Z) \leq d(p,v) \leq d(p,q) \leq k$. Thus, our $k$-path intersects $Z_i^{+k}$. This concludes the proof. 
\end{proof}

\subsection{Volume growth}
\begin{definition}\label{boundedgeom and volume}
    A metric space $X$ has \emph{bounded geometry at scale $s$} if, for all $r\geq 0$, there exists $n\geq 0$ such that every ball of radius $r$ in $X$ can be covered by $n$ balls of radius $s$.
    \\ If $X$ has bounded geometry at scale $s$ and $A \subset X$ is a bounded subset, we define its \emph{$s$-volume}, denoted $|A|_s$, as the minimal number of closed balls of radius $s$ needed to cover $A$. 
\end{definition}

\noindent
If $\Gamma$ is a graph, we endow it with the shortest path metric. Therefore, a connected graph is a $1$-coarsely connected metric space. A graph $\Gamma$ with bounded valency is a metric space with bounded geometry at all scales, and we will always take the radii of the balls in the covering to be $s=1/5$, so that every finite subset $A \subset \Gamma$ satisfies $|A|_s = \#A $.
\begin{remark}\label{r s}
    To simplify notation in the following, unless otherwise stated, the coarse connectivity constant $k$ will always be taken $\geq 1$, and we will consider spaces of bounded geometry at scale $s=1/5$, that we will just refer to as spaces of bounded geometry. The $\frac{1}{5}$-volume of a subset $A$ will be called its volume, and denoted $|A|$. For the general case, we can always take $k$ and $s$ such that $k \geq 5s$. This choice is motivated by the fact that, if $A,B$ are bounded subsets in $X$ satisfying $d(A,B) \geq 1$, then $|A \cup B| = |A|+|B| $, because a minimal set of balls of radii $1/5$ covering $A$ is disjoint from any minimal set of balls of radii $1/5$ covering $B$.
 \end{remark}
 \begin{lemma}\label{neighb of neighb}
 The following assertions hold.
     \begin{itemize}
         \item[(i)] Let $Z$ be a metric space of bounded geometry, $A$ a bounded subset, and $\alpha \geq 0$. Then 
         $$  | A^{+ \alpha}| \leq \sup_{z\in Z}|B(z,\alpha) | \cdot |A|.                  $$
         \item[(ii)] Let $X$ be a $k$-coarsely geodesic metric space. For every $\alpha \geq \beta \geq k $, there exists $t \geq 0$ such that, for any bounded subset $A$,
     $$    A^{+ \alpha} \subset (A^{+ \beta})^{+ t}.            $$
     If moreover $X$ has bounded geometry, then
     $$     | A^{+ \alpha}| \leq \sup_{x\in X}|B(x,t) | \cdot |A^{+ \beta}|.               $$
     \end{itemize}
 \end{lemma}
 \begin{proof}
$(i)$ We have $ A^{+ \alpha}=\cup_{z \in A } B_Z(a,\alpha) $ so $| A^{+ \alpha}| \leq \sup_{z\in Z}|B(z,\alpha) | \cdot |A|$.

\medskip \noindent
$(ii)$ Let $\sigma$ be the map as in definition \ref{coarse connected} and let $x \in A^{+ \alpha}$. Let $a \in A$ such that $d(x,a) \leq \alpha$. There exists a $k$-path $x_0 = x, \dots , x_p = a$ of length $p = \sigma(\alpha)$ from $x$ to $a$. We have $d(x_{p-1},a) \leq k \leq \beta$, so $x_{p-1} \in A^{+ \beta}$. We have $d(x,x_{p-1}) \leq t $, so $x \in \left(A^{+ \beta}\right)^{(p-1)k}$. If moreover $X$ has bounded geometry, the result follows from $(i)$.
 \end{proof}
\begin{definition}\label{growth function of a subset}
    Let $X$ be a metric space with bounded geometry. We define its \emph{growth rate} $\beta_X$ as
    \begin{center}
        $\forall r>0$, $\beta_X(r) = \sup_{x \in X}|B_X(x,r)|$.
    \end{center}
    \begin{itemize}
        \item If $S \subset X$ is a subset, we define its \emph{growth rate} $\beta_S$ as
    \begin{center}
        $\forall r>0$, $\beta_S(r) = \sup_{s \in S}|B_X(s,r) \cap S|$.
    \end{center}
    \item If $\mathcal{S}$ is a family of subsets of $X$, we define its growth as 
    \[V_{\mathcal{S}}(r)=\sup_{s\in S;S\in \mathcal S}|B_X(s,r)\cap S|.\]
     We say that $\mathcal{S}$ has \emph{subexponential growth} if $\limsup_{r\rightarrow\infty} \frac{1}{r}\log(V_{\mathcal{S}}(r)) = 0$.
    \end{itemize}
\end{definition}

\noindent
Similarly, we define the growth rate of finitely generated groups and the relative growth rate of a finitely generated subgroup as follows.

\begin{definition}\label{relative growth}
Let $G$ be a group generated by a finite set $S$ satisfying $S^{-1}=S$, and let $d$ denote the word metric on $G$ with respect to $S$. The \emph{growth rate} of $G$ is the growth rate of the metric space $(G,d)$.

\medskip \noindent
If $H$ is a finitely generated subgroup of $G$, we define its \emph{relative growth rate} $\beta_H^G : \mathbb{N} \to \mathbb{N}$ as the number of elements $h\in H$ such that $d(e,h) \leq r$. We say that $H$ has \emph{subexponential relative growth} if $\limsup_{r\rightarrow\infty} \frac{1}{r}\log\left(\beta_H^G(r)\right) = 0$.
\end{definition}

\noindent
The growth rates do not depend on the choice of the generating set of $G$, up to the following equivalence relation $\simeq$.

\medskip \noindent
We denote by $f \gtrsim g $ for functions $f,g : \mathbb{R}_{\geq 0} \to \mathbb{R}_{\geq 0}$ if there exists a constant $C >0$ such that, for all $t\geq 0$, $f(t) \geq C g(Ct+C) +C$. And we write $f \simeq g$ if $f \gtrsim g $ and $g \gtrsim f$.
\begin{remark}
Note that, if $H$ is a finitely generated subgroup of a finitely generated group $G$, then $\beta_H^G \gtrsim \beta_H$ with $\beta_H^G \simeq \beta_H$ if and only if $H$ is undistorted in $G$.
\end{remark}
\begin{definition}
Consider the following families of spaces.
    \begin{itemize}
        \item We denote by $\mathfrak{M}_{exp}$ the class of metric spaces with bounded geometry such that any coarsely separating family of subsets $\mathcal{S}$ must have exponential growth, i.e.\ $V_{\mathcal{S}}(r) \gtrsim e^{r}$.
        \item Let $d>0$. We denote by $\mathfrak{M}_d$ the class of metric spaces with bounded geometry such that any coarsely separating family of subsets $\mathcal{S}$ must have growth at least $r^d$, i.e.\ $V_{\mathcal{S}}(r) \gtrsim r^d$.
    \end{itemize}
\end{definition}

\noindent
It is easy to see that, if $d\leq d'$, then $\mathfrak{M}_{exp} \subset \mathfrak{M}_{d'}\subset \mathfrak{M}_d$. Moreover, belonging to $\mathfrak{M}_{exp}$ or $\mathfrak{M}_d$ is invariant under quasi-isometry between coarsely geodesic metric spaces by Lemma \ref{lem:CoarseSep}.
\begin{definition}\label{persistent}
     Let $X$ be a $k$-coarsely connected metric space of bounded geometry, $r_0 \geq 0$, and $\alpha >0$. A family $\left\{ A_x(r) \right\}_{x\in X, r \geq r_0}$ of subsets of $X$ is called \emph{$\alpha$-persistent} if, for every $ r\geq r_0$:
     \begin{itemize}
         \item $\forall x\in X$, $A_x(r) \subset B_X(x,4r)$;
         \item $\forall x,y \in X$, $| A_x(r)| = |A_y(r)|$, which we denote by $|A(r)|$;
         \item for all $x,y \in X$ satisfying $d(x,y) \leq k$, $| A_x(r) \cap A_y(r)| \geq \alpha |A(r)| $. 
     \end{itemize}

\end{definition}
\begin{ex}\label{exple persist}
Let $X$ be a connected vertex-transitive graph of bounded valency $d$.
\begin{itemize} 
    \item[(i)] The family of balls $B(x,r)$, where $x \in X$ and $r>0$, is $\frac{1}{d+1}$-persistent. Indeed, $X$ is 1-coarsely connected, and the second condition follows from the vertex-transitivity. Let $r>0$ and $x,y \in X$ satisfying $d(x,y)=1$. We have $$| B(x,r) \cap B(y,r)| \geq |B(x,r-1)| \geq \frac{1}{d+1}  |B(x,r)|. $$ 
    \item[(ii)] Another example is given by $A_x(r) = S(x,r)^{+ t} \subset B_X(x, 4r)$, where $t \geq 1$ is fixed and $r\geq t$. Again by vertex-transitivity, we have the second condition, and, if $d(x,y)=1$, then $S(x,r) \subset S(x,r)^{+ t} \cap S(y,r)^{+ t}$. We have $S(x,r)^{+ t} = \bigcup_{z \in S(x,r)}B(z,t) $, so $|S(x,r)^{+ t}| \leq |S(x,r)| |B(t)|$ where $|B(t)|$ denotes the size of a (any) ball of radius $t$. Therefore $| S(x,r)^{+ t} \cap S(y,r)^{+ t}  | \geq \frac{1}{|B(t)|} |S(x,r)^{+ t}| $, and the set $\left\{ A_x(r) \right\}_{x\in X, r \geq t}$ is $\frac{1}{|B(t)|}$-persistent.
\end{itemize}
\end{ex}

\section{Coarse separation problem}\label{section:CoarseSep}
\subsection{General setting and main tools}\label{sebsec:tools}
\subsubsection{Cutsets and main results}
We start by recalling the definition of the cut of a subset, due to Benjamini--Schramm--Tim{\'a}r, which quantifies its connectivity. 
\begin{definition}[\cite{benjamini2012separation}]
    Let $X$ be a $k$-coarsely connected metric space of bounded geometry and let $\delta \in (0,1)$. Let $S \subset A$ be bounded subsets of $X$ and $r \geq k$. Then $S$ is an \emph{$(r,\delta)$-cut} of $A$ if any $r$-coarsely connected component of $A \backslash S$ has size $\leq \delta |A|$. We denote by $\textup{cut}^{\delta}_r(A)$ the minimal size of a $(r,\delta)$-cut of $A$.
\end{definition}
\begin{remark}
    When $r$ is equal to the coarse connectivity constant $k$ of $X$, we will just write $\cut$ instead of $\cut_r$, in particular for connected graphs where $k=1$. 
\end{remark}
\begin{prop}\label{main}
    Let $X$ be a $k$-coarsely connected metric space of bounded geometry and $\left\{ A_x(r) \right\}_{x\in X, r\geq r_0}$ an $\alpha$-persistent family. Set $\delta = 1-\frac{\alpha}{2}$ and let $S \subset X$ be a subset and $(X_i)_{i \in I}$ a partition of $ X \backslash S$. Assume that there exist $r \geq r_0$, $x,y \in X$ and $i_0 \in I$ such that
    \begin{itemize}
        \item $| A_x(r) \cap X_{i_0}| \geq \delta |A(r)|$;
        \item $| A_y(r) \cap X_{i_0}| < \delta |A(r)|$.
    \end{itemize}
    Then there exists $z\in X$ such that $|A_{z}(r) \cap X_i| \leq \delta  |A|  $ for any $i\in I$. In particular, it is the case when $A_x(r)$ and $A_y(r)$ are contained in different components of the partition. Moreover, if the partition $(X_i)_{i \in I}$ correspond to the $k$-coarsely connected components of $X \backslash S$, then $A_z(r) \cap S$ is a $\delta$-cut of $A_z(r)$.
\end{prop}
\begin{proof}
    Let $x=x_0, \dots,x_m = y$ be a $k$-path.
    To simplify notation, we write $A_p = A_{x_p}(r)$ for $p\in [0,m]$, and $|A|=|A(r)|$. Let $s = \min \left\{ p\in [1,m] \textup{ such that } | A_p \cap X_{i_0}| < (1-\frac{\alpha}{2}) |A|   \right\}$. So 
    \begin{equation}\label{*}
    \left\{
\begin{split}
     | A_{s-1} \cap X_{i_0}| \geq \left(1-\frac{\alpha}{2}\right) |A|,
    \\  | A_s \cap X_{i_0}| < \left(1-\frac{\alpha}{2}\right) |A|.
\end{split}
\right.
\end{equation}
The inclusion $A_{s-1} \cap X_{i_0} \subset (A_{s} \cap X_{i_0}) \sqcup (A_{s-1} \backslash A_s)$ implies 
\begin{equation}\label{**}
    |A_{s-1} \cap X_{i_0}| \leq |(A_{s} \cap X_{i_0})| + |(A_{s-1} \backslash A_s)|.
\end{equation}
On the other hand, $|A_{s-1}\cap A_s| \geq \alpha |A| $. So $|(A_{s-1} \backslash A_s)| \leq (1-\alpha) |A| $. Therefore the first inequality in (\ref{*}) and (\ref{**}) imply
$$     \frac{\alpha}{2}  |A| \leq   |(A_{s-1} \cap X_{i_0})| - (1-\alpha)   |A|     \leq  |(A_{s} \cap X_{i_0})|  . $$
Since $|A_s \cap X_i| + |A_s \cap X_{i_0}| \leq |A| $ for any $i \in I \backslash\{i_0\}$, we have $|A_s \cap X_i| \leq (1-\frac{\alpha}{2})  |A|  $. Combined with the second inequality in (\ref{*}), we conclude that $|A_{x_s}(r) \cap X_i| \leq \delta  |A|  $ for any $i\in I$.
\end{proof}

\begin{cor}\label{cor of main}
    Let $X$ be a $k$-coarsely connected metric space of bounded geometry, and let $\left\{ A_x(r) \right\}_{x\in X, r\geq r_0}$ be an $\alpha$-persistent family. If $\mathcal{S}$ is a family of subsets that coarsely separates $X$, then there exists a constant $C'$ such that, for any $r$ big enough, there exist $S \in \mathcal{S}$ and $z \in X$ such that $|B_X(z,5r)\cap S | \geq C' \, \cut(A_{z}(r)) $, where $\delta = 1 - \frac{\alpha}{2}$.
\end{cor}
\begin{proof}
    Let $D\geq 0$ be the thickening constant of $\mathcal{S}$. For $4r$, there exist $S \in S$ and $x,y\in X$ such that $B(x,4r)$ and $B(y,4r)$ are in different $k$-coarsely connected components of $X \backslash S^{+ D}$. Therefore $A_{x}(r)$ and $A_{y}(r)$ are in different $k$-coarsely connected components. By Proposition \ref{main}, where $S$ is replaced by $S^{+ D}$ and the partition is given by the $k$-coarsely connected components of $X \backslash S^{+ D}$, there exists $z \in X$ such that $|A_{z}(r)\cap S^{+ D}| \geq \cut(A_{z}(r))$. Note that 
    $$B(z,4r) \cap S^{+ D} \subset (B(z,4r+D) \cap S)^{+D} , $$ 
    so, by lemma \ref{neighb of neighb},
    $$   |A_{z}(r)\cap S^{+ D}| \leq    | B(z,4r) \cap S^{+ D} | \leq |(B(z,4r+D)\cap S)^{+D}| \leq  \beta_X(D) \cdot |B(z,4r+D)\cap S|  .$$
    We conclude that, for $C' = \frac{1}{\beta_X(D)} $ and $r\geq \max(r_0,D)$, $|B(z,5r)\cap S| \geq C' \, \cut(A_{z}(z)) $.
\end{proof}

\noindent
We show that, if a metric space $X$ has a persistent family with exponentially growing cuts, then it belongs to $\mathfrak{M}_{exp}$. Moreover, it holds for $X \times Y$ where $Y$ is any coarsely connected metric space. (All product spaces will be endowed with the $L^1$ product metric.)

\begin{thm}\label{invariance by taking product}
    Let $X$ be a $k$-coarsely connected metric space of bounded geometry and let $\left\{ A_x(r) \right\}$ be an $\alpha$-persistent family. Assume that there exist constants $C,\beta >0$ such that, for all $x\in X$ and $r$ big enough, $\cut(A_x(r)) \geq C e^{\beta r}$ where $\delta = (1-\frac{\alpha}{2})$. Then
    \begin{itemize}
        \item[(i)] $X \in \mathfrak{M}_{exp}$.
        \item[(ii)] If moreover $Y$ is a $k$-coarsely connected metric space, then $X \times Y \in \mathfrak{M}_{exp}$.
    \end{itemize}
\end{thm}
\begin{proof}
   $(i)$ Let $\mathcal{S}$ be a family of subsets that coarsely separates $X$. The previous corollary implies that, for any $r$ big enough, there exist $S \in \mathcal{S}$ and $z \in X$ such that $|B_X(z,5r)\cap S | \geq C' \, \cut(A_{z}(r)) $. By taking $s \in B_X(z,5r)\cap S $, we have 
   $$|B_X(s,10r)\cap S | \geq |B_X(z,5r)\cap S | \geq C' \, \cut(A_{z}(r))     \geq  C'C e^{\beta r} .$$
   Therefore 
   $$      V_\mathcal{S}(10r) \geq    C'C e^{\beta r}  .$$
   $(ii)$ Let $Y$ be a $k$-coarsely connected metric space, denote $Z = X \times Y$, and let $\mathcal{S}$ be a family of subsets that coarsely separates $Z$, with thickening constant $D$. Let $r>0$ be big enough and $S \in \mathcal{S}$ such that $ Z \backslash S^{+ D}$ contains two balls of radius $4r$ in different $k$-connected components. Let $(Z_i)_{i \in I}$ be the $k$-coarsely connected components of $ Z \backslash S^{+ D}$. Up to replacing $D$ with a bigger thickening constant, we can assume that, for all $i\ne j \in I$, $d_Z(Z_i,Z_j) >k$. We distinguish two cases. 
   
\medskip \noindent 
\emph{First case:} If there exist $y \in Y$, $x_1,x_2 \in X$, and $i_1,i_2 \in I$ such that  
   \begin{equation*}
   \left\{
       \begin{split}
           | \left( A_{x_1}(r)\times \{y\}\right) \cap Z_{i_1}|   \geq \delta |A(r)|
        \\ |  \left( A_{x_2}(r)\times \{y\}\right) \cap Z_{i_2}| < \delta |A(r)|
       \end{split}
       \right..
   \end{equation*}
    Then, by Proposition \ref{main}, there exists $x\in X$ such that $\left( A_{x}(r)\times \{y\}\right) \cap S $ is a $\delta$-cut of $A_{x}(r)\times \{y\}$. Hence, we have on the one hand
    $$  | \left( A_{x}(r)\times \{y\}\right) \cap S |  \geq  \cut( A_{x}(r)\times \{y\}) = \cut( A_{x}(r))    \geq      C e^{\beta r}  ;          $$
    and on the other hand 
    $$ A_{x}(r)\times \{y\} \subset B_X(x,4r)\times \{y\} \subset    B_Z(\left(x,y),4r\right).        $$
   Therefore 
   $$  | B_Z(\left(x,y),4r\right)  \cap S |   \geq   C e^{\beta r}  .     $$
   \emph{Second case:} For every $y\in Y$, there exists $i_y \in I$ such that 
   \begin{center}
       $\forall x \in X$, $ | \left( A_{x}(r)\times \{y\}\right) \cap Z_{i_y}|   \geq \delta |A(r)|.        $
   \end{center}
   Since $ Z \backslash S^{+ D}$ contains two balls of radius $4r$ in different $k$-connected components, there must exist $y,y' \in Y$ such that $i_y \ne i_{y'}$. Moreover, $Y$ is $k$-coarsely connected, so by considering a $k$-path in $Y$ from $y$ to $y'$, we can suppose that $d_Y(y,y')\leq k$. To simplify notation, we write $i_y = 1$ and $i_{y'}=2$. Let $x\in X$, we have 
   \begin{equation*}
   \left\{
       \begin{split}
           | \left( A_{x}(r)\times \{y\}\right) \cap Z_{1}|   \geq \delta |A(r)|
        \\ |  \left( A_{x}(r)\times \{y'\}\right) \cap Z_{2}| \geq \delta |A(r)|
       \end{split}
       \right..
   \end{equation*}
   Note also that 
   $$ |A_x(r)| = | A_{x}(r)\times \{y'\}   | \geq  |  \left( A_{x}(r)\times \{y'\}\right) \cap Z_{1}| + |  \left( A_{x}(r)\times \{y'\}\right) \cap Z_{2}|.         $$
   Therefore
   \begin{equation*}
   \left\{
       \begin{split}
           &| \left( A_{x}(r)\times \{y\}\right) \cap Z_{1}|   \geq \delta |A(r)|
        \\ &|  \left( A_{x}(r)\times \{y'\}\right) \cap Z_{1}| \leq (1-\delta) |A(r)|
       \end{split}
       \right..
   \end{equation*}
   Let $\phi : X \times \{y\} \to X \times \{y'\} $ be the translation sending $(a,y)$ to $(a,y')$. For all $i \in I \backslash \{1\}$, $d_Z(Z_1,Z_i)>k$; so, for every $z \in \left( A_{x}(r)\times \{y\}\right) \cap Z_{1}$, $\phi(z)$ is either still in $Z_1$ or is in $S$. Hence 
   \begin{equation*}
       \begin{split}
           | \left( A_{x}(r)\times \{y\}\right) \cap Z_{1}| &= | \phi \left( \left( A_{x}(r)\times \{y\}\right) \cap Z_{1} \right)| 
           \\ &= | \phi \left( \left( A_{x}(r)\times \{y\}\right) \cap Z_{1} \right)\cap Z_{1}|+ | \phi \left( \left( A_{x}(r)\times \{y\}\right) \cap Z_{1} \right) \cap S|
           \\ &= |  \left( A_{x}(r)\times \{y'\}\right) \cap Z_{1}|+ |  \left( A_{x}(r)\times \{y'\}\right) \cap S|
           \\ & \leq (1-\delta) |A(r)| + |  B_Z((x,y'),4r) \cap S|.
       \end{split}
   \end{equation*}
   Since $  | \left( A_{x}(r)\times \{y\}\right) \cap Z_{1}| \geq \delta |A(r)| $ and $2 \delta -1 = 1- \alpha >0$, we conclude that
   $$ |  B_Z((x,y'),4r) \cap S| \geq (2 \delta -1) |A(r)|  \geq (2 \delta -1) \cut( A_{x}(r))    \geq  (2 \delta -1)    C e^{\beta r} . \qedhere$$ 
\end{proof}

\noindent
The next theorem shows that, if a metric space $X$ satisfies the conditions of the previous theorem and admits quasi-isometric embeddings (with uniform constants) into a metric space $Y$ to join any pair of its points, then $Y \in \mathfrak{M}_{exp}$.

\begin{thm}\label{separating by embedding}
    Let $X$ be a $k$-coarsely connected metric space of bounded geometry and let $\left\{ A_x(r) \right\}$ be an $\alpha$-persistent family. Assume that there exist constants $C,\beta >0$ such that, for all $x\in X$ and $r$ big enough, $\cut(A_x(r)) \geq C e^{\beta r}$ where $\delta = (1-\frac{\alpha}{2})$. Let $Y$ be a coarsely connected metric space of bounded geometry and suppose that there exist constants $\lambda \geq 1$, $K,M \geq 0$ such that, for any $y_1,y_2 \in Y$, there exists a $(\lambda,K)$-quasi-isometric embedding $f: X \to Y$ satisfying $y_1,y_2 \in f(X)^{+M}$. Then $Y \in \mathfrak{M}_{exp}$.

    \begin{proof}
        Let $k'\geq 0$ such that $Y$ is $k'$-coarsely connected. Up to replacing by a bigger constant, we can assume that $k' \geq \lambda k + K$. Let $\mathcal{S}$ be a family of subsets that coarsely separates $Y$, with thickening constant $D'$. Let $R$ be big enough, $S \in \mathcal{S}$, and $y_1,y_2 \in Y$ such that $B_Y(y_1,R)$ and $B_Y(y_2,R)$ are in different $k'$-coarsely connected components of $Y \backslash S^{+ D'}$. Denote by $(Y_i)_{i \in I}$ the $k'$-coarsely connected components of $Y \backslash S^{+ D'}$. Let $f : X \to Y$ be a $(\lambda,K)$-quasi-isometric embedding such that $y_1,y_2 \in f(X)^{+M}$. Let $x_1,x_2 \in X$ such that $d(f(x_1),y_1)\leq M$ and $d(f(x_2),y_2)\leq M$.
        
\medskip \noindent
Necessarily, $B_Y(f(x_1),R-M)$ and $B_Y(f(x_2),R-M)$ are entirely contained in different $k'$-coarsely connected components of $Y \backslash S^{+ D'}$. Since $f$ sends a $k$-path in $X$ to a $k'$-path in $Y$, $x_1$ and $x_2$ are in different $k$-connected components of $X \backslash f^{-1}(S^{+ D'})$. Moreover, for $r = \frac{R-M-K}{4\lambda}$, $A_{x_1}(r)$ and $A_{x_2}(r)$ are in different $k$-coarsely connected components of $X \backslash f^{-1}(S^{+ D'})$. By Proposition \ref{main}, there exists $x \in X$ such that $$|A_x(r)\cap f^{-1}(S^{+ D'})| \geq \cut(A_x(r) \geq C e^{\beta r} . $$ 
        It is easy to see that quasi-isometric embeddings coarsely preserve the volume (see \cite[Lemma 3.1]{bensaid2022coarse} for a proof), i.e.\ there exists a constant $\theta>0$ such that, for any bounded subset $A\subset X$, $|f(A)|\geq \theta |A|$. The inclusion $f(A_{x}(r)) \subset B_Y(x,R)$ implies, by an argument similar to the one in the proof of Corollary \ref{cor of main}, that, for $R$ big enough,
        $$    \frac{C'}{\theta} |B_Y(f(x),2R) \cap S|     \geq  \frac{1}{\theta} |B_Y(f(x),R) \cap S^{+ D'}| \geq C e^{\beta r} $$
        where $C' = \sup_{y \in Y}|B_Y(y,D')|$. Hence, by taking $s \in B_Y(f(x),2R) \cap S $,
        $$ |B_Y(s,4R) \cap S|  \geq \frac{C \theta}{C'} \exp \left(  \beta   \frac{R-M-K}{2\lambda}      \right)  .$$
        This concludes the proof.$\qedhere$
    \end{proof}
\end{thm}

\subsubsection{Cheeger and Poincar\'e constants}
The purpose of this subsection is to introduce the main tools to get a lower bound on the $\delta$-cuts.
\begin{definition}(Compare with \cite[Definition 6.3]{hume2020poincare})
\\Let $X$ be a $k$-coarsely connected metric space, let $r\geq k \geq 1$, and let $B$ be a bounded subset of $X$. We define the \emph{$r$-boundary} of $B$ as 
$$\partial_r B = \{ x \in X \backslash B \mid d(x,b) \leq r \textup{ for some } b \in B \} = B^{+ r} \backslash B .$$
The \emph{$r$-Cheeger constant} of a bounded set $A$ is 
$$h_r(A) = \min \left\{ \frac{|\partial_r B \cap A|}{|B|}, B \subset A \textup{ and } 0<|B|\leq \frac{|A|}{2}  \right\}.$$
\end{definition}
\begin{remark}
    If $X$ has bounded geometry, then for any bounded subset $A$, $h_r(A) \leq \beta_X(r)$ (take $B$ to be any point in $A$). 
\end{remark}
\noindent
The following proposition is adapted from \cite[Proposition 2.2]{hume2014continuum}.
\begin{prop}\label{cut general}
Let $X$ be a $k$-coarsely connected metric space of bounded geometry and $r\geq k \geq 1$. Let $A \subset X$ be a bounded subset and $\delta \in (0,1)$. Then
    $$\textup{cut}^{\delta}_r(A) \geq \lambda \, h_r(A) |A| ,$$
where $\lambda = \frac{\min\left(\frac{1}{4}, \frac{1-\delta}{2}\right)}{\beta_X(r)}$.
\end{prop}
\begin{proof}
    Let $S$ be an $(r,\delta)$-cut of $A$, and let $A_1, \ldots, A_p$ denote the $r$-coarsely connected components of $A \backslash S$. So, for any $i$, $|A_i| \leq \delta |A| $. Note that, for all distinct $1 \leq i,j \leq p$, $d(A_i,A_j) > r \geq 1$, so, by Remark \ref{r s}, $|A_i \cup A_j| = |A_i|+|A_j|$.
    
\medskip \noindent
Suppose that $\delta \leq 1/2$, i.e.\ $S$ is a $\frac{1}{2}$-cut of $A$. If $|S| \geq \frac{1}{2}|A|$, we are done by the previous remark. If $|S| < \frac{1}{2}|A|$, then $\sum_{i=1}^p |A_i| \geq \frac{1}{2}|A|  $, and we can find a  $D = \cup_{i\in I}A_i $, $I \subset \{1,\dots,p\}$, such that $\frac{1}{4}|A| \leq | D| \leq \frac{1}{2}|A| $. Since $\partial_r D \cap A \subset S$, we have  $$ h_r(A) \leq \frac{|\partial_r D \cap A|}{|D|} \leq \frac{|S|}{ \frac{1}{4} |A|}.        $$
    Suppose now that $\delta >1/2$. If $|S| \geq \frac{1-\delta}{2} |A| $, we are done. If all the $A_i$ have size $\leq \frac{1}{2} |A| $, then $S$ is again a $\frac{1}{2}$-cut of $S$ and we conclude as in the previous case. Suppose that $|S| < \frac{1-\delta}{2} |A| $ and that $| A_1 | > \frac{1}{2} |A|$. Let $D = \cup_{i=2}^p A_i$. So $|D| \leq \frac{1-\delta}{2} |A|$, and $\partial_r D \cap A \subset S$. Therefore 
    $$ h_r(A) \leq \frac{|\partial_r D \cap A|}{|D|} \leq \frac{|S|}{ \frac{1-\delta}{2} |A|}.        $$
    Hence, in all cases,  $|S| \geq \min(\frac{1}{4}, \frac{1-\delta}{2}) h_r(A) |A| . \qedhere$
\end{proof}

\noindent
We recall that all spaces we consider have bounded geometry at scale $s=1/5$ (see Remark \ref{r s}), and that the volume $|A|$ of a subset is its $\frac{1}{5}$-volume.
\begin{lemma}\label{volume and counting measure}
    Let $X$ be a metric space of bounded geometry. Let us denote by $\nu$ the counting measure. Let $\eps \geq 1$, $A$ a bounded subset of $X$, and $Z \subset A$ a maximally $\eps$-separated net. Then for any $B \subset A$,
    $$ \frac{1}{\alpha} |B| \leq \nu \left( B^{+ \eps}\cap Z    \right) \leq \alpha |B| $$
    where $\alpha = \sup_{x \in X} |B_x(x,\eps)| $.
\end{lemma}
\begin{proof}
    Let $x \in B$. There exists $z \in Z$ such that $d(x,z)\leq \eps$. So $z \in B^{+ \eps}\cap Z$, and $x \in ( B^{+ \eps}\cap Z)^{+ \eps}$. So $B \subset ( B^{+ \eps}\cap Z)^{+ \eps} \subset \cup_{x \in B^{+ \eps}\cap Z}B_X(x,\eps) $. Therefore
    $$|B| \leq \alpha \, \nu \left( B^{+ \eps}\cap Z    \right). $$
    Let $n = |B^{+ \eps}|$ and $x_1, \dots , x_n$ such that $B^{+ \eps} \subset \cup_i B_X(x_i,1/5)$. Two different points $z_1, z_2 \in B^{+ \eps} \cap Z$ cannot be in a same ball $B_X(x_i,1/5)$ because $Z$ is $\eps$-separated and $\eps \geq 1$, so 
    $$\nu \left( B^{+ \eps}\cap Z    \right) \leq |B^{+ \eps}| \leq \alpha |B| . \qedhere$$
\end{proof}
\begin{definition}\label{Poincar\'e constant 1}
Let $(X,d)$ be a complete separable metric space, $\nu$ a finite Borel measure, and let $k>0$. Let $f: X \to \mathbb{R}$ be a measurable function. Its \emph{upper gradient at scale $k$}, $|\nabla_k f| : X \to \mathbb{R}$,  is defined as  
$$  |\nabla_k f|(x) = \sup_{y,y' \in B_X(x,k)}   | f(y)-f(y')   |      .          $$
The \emph{$L^1$-Poincar\'e constant at scale $k$} of $(X,d,\nu)$ is 
$$    h_k^1(X,d,\nu) = \inf_f \frac{||\nabla_k f||_1}{||f||_1} ,       $$
where the infinimum is over all $f \in L^1(X,\nu)$ such that $ \int_X f d \nu = 0$ and $f \ne 0$.
\end{definition}

\noindent
The $L^1$-Poincar\'e constant was introduced in \cite{hume2020poincare}, and is closely related to the Cheeger constant, see \cite[Proposition 6.10]{hume2020poincare}. We adapt their proof to show the following.
\begin{prop}\label{comparison cheeger Poincar\'e}
    Let $(X,d)$ be a $k$-coarsely geodesic metric space of bounded geometry, and $\eps$ such that $k \geq \eps \geq  1$. There exists a constant $M>0$ such that, for any bounded subset $A$ of $X$ and $Z \subset A$ a maximally $\eps$-separated net in $A$, we have
    $$  h_k^1(Z,d, \nu) \leq M \, h_k(A),          $$
    where $\nu$ is the counting measure.
\end{prop}
\begin{proof}
    Let $B \subset A$ such that $|B| \leq \frac{|A|}{2}$. Let $Z_B = B^{+ \eps}\cap Z$ and $f : Z \to \mathbb{R}$ equal to $\mathds{1}_{Z_B}$. We denote by $f_Z = \frac{1}{\nu(Z)} \int_Z f d \nu $ the average of $f$. Let us show that $\frac{||\nabla_kf||_1}{||f-f_Z||_1}$ is bounded from above by a uniform constant times $\frac{|\partial_k B \cap A|}{|B|}$. We have $f_Z = \frac{\nu(Z_B)}{\nu(Z)}$ and
$$||f-f_Z||_1 = \nu(Z_B) \left(  1- \frac{\nu(Z_B)}{\nu(Z)}   \right) + \left( \nu(Z)- \nu(Z_B)   \right) \frac{\nu(Z_B)}{\nu(Z)} \geq \nu(Z_B).  $$
By Lemma \ref{volume and counting measure}, $\nu(Z_B) \geq \frac{1}{\alpha}|B| $, where $\alpha = \sup_{x \in X} |B_x(x,\eps)| $. So
$$   ||f-f_Z||_1 \geq    \frac{1}{\alpha}|B| .  $$
Let $x \in Z$ such that $\nabla_k f (x) = 1$. So there exist $y,y' \in B_Z(x,k)$ such that one is in $Z_B$ and the other in $Z \backslash Z_B$. Suppose $y \in Z_B$. So $y' \in Z \backslash Z_B $ and $y' \in B^{+(2k+\eps)}$, i.e.\ $y' \in \partial_{2k+\eps}(B) \cap A$. Therefore
$$ \left\{  x\in Z \mid \nabla_k f (x) = 1   \right\}   \subset   \left( 
 \partial_{2k+\eps}(B) \cap A    \right)^{+k}    . $$
So 
\begin{equation}\label{nabla}
    \begin{split}
         ||\nabla_kf||_1 = \nu \left(   \left\{  x\in Z \mid \nabla_k f (x) = 1   \right\}   \right) 
         &= \mid \left\{  x\in Z \mid \nabla_k f (x) = 1   \right\} \mid 
         \\ & \leq \mid \left( 
 \partial_{2k+\eps}(B) \cap A    \right)^{+k} \mid 
 \\ & \leq \beta  \mid 
 \partial_{2k+\eps}(B) \cap A,
    \end{split}
\end{equation}
where $\beta = \sup_{x \in X} |B_x(x,k)| $. The second equality in (\ref{nabla}) comes from the fact that $Z$ is $\eps$-separated and $\eps \geq 1$, so $\nu = |\cdot| $ for its subsets. Moreover, the same proof of Lemma \ref{neighb of neighb} shows that there exists $t \geq 0$ such that 
$$  \partial_{2k+\eps}(B) \subset \partial_{k}(B)^{+t}.     $$
By denoting $\gamma = \sup_{x \in X} |B_x(x,t)| $, we have 
$$ ||\nabla_kf||_1 \leq \beta \gamma \, |\partial_{k}(B) \cap A|.$$
We conclude that, for $M = \alpha \beta \gamma$, we have that, for all $B \subset A$, 
$$ h_k^1(Z,d, \nu) \leq  \frac{||\nabla_kf||_1}{||f-f_Z||_1} \leq M \frac{|\partial_k B \cap A|}{|B|} . \qedhere$$
\end{proof}

\subsection{Groups of polynomial growth}
The goal of this section is to prove the following theorem, which implies Theorem \ref{thmIntro:nilpot} by results of Bass \cite{bass1972degree} and Guivarc'h \cite{guivarc1973croissance}.
\begin{thm}\label{thm cglc poly growth}
Let $G$ be a compactly generated locally compact group, $d$ the word metric with respect to a compact symmetric generating subset, and $\mu$ a left-invariant Haar measure. If $G$ has growth $\beta_G(r) \simeq r^n$, then $G \in \mathfrak{M}_{n-1}$.
\end{thm}

\noindent
Throughout this section, $G$ is a compactly generated locally compact group, $d$ the word metric with respect to a compact symmetric generating subset, $\mu$ a left-invariant Haar measure, and $\beta_G(r) \simeq r^n$. In particular, $(G,d)$ is coarsely connected, see \cite[Proposition 4.B.8]{cornulier2014metric}, and let $k>0$ such that it is $k$-coarsely connected. 
\begin{prop}\label{persistent in poly}
    There exists $r_0>0$ such that, for any $r \geq r_0$ and for any $x,y \in G$ satisfying $d(x,y) \leq k$, 
    $$    \mu(B(x,r)\cap B(y,r)) \geq \frac{4}{5} \mu(B(x,r)). $$
\end{prop}
\begin{proof}
    For $x \in G$ and $s>r \geq 0 $, let $C_{r,s}(x) = B(x,s)\backslash B(x,r)$. By \cite[Lemma~8.7]{hume2020poincare} (and \cite[Lemma 24]{tessera2007volume}), there exist $r_0'>0$ and $\epsilon>0$ such that, for all $r \geq r_0'$ and for all $x \in G$,
    $$    \mu(C_{r,(1+ \eps)r}(x)) \leq  \frac{1}{5} \mu(B(x,r)).   $$
    Take $r_0 = \max(r_0',\frac{k}{\eps})$ so that, for all $r \geq r_0$, $C_{r,r+k}(x) \subset  C_{r,(1+ \eps)r}(x)  $. Therefore
    $$ \mu(C_{r,r+k}(x)) \leq \frac{1}{5} \mu(B(x,r)) \leq \frac{1}{5} \mu(B(x,r+k))   .  $$
    So, for all $r \geq r_0+k$, $\mu(C_{r-k,r}(x)) \leq \frac{1}{5} \mu(B(x,r))$. In particular,
    $$  \mu(B(x,r-k)) \geq (1- \frac{1}{5})   \mu(B(x,r))   .      $$
    We conclude that, for all $x,y \in G$ satisfying $d(x,y) \leq k$ and for all $r \geq r_0+k$,
    $$      \mu\left(B(x,r) \cap B(y,r)\right) \geq \mu(B(x,r-k))    \geq \frac{4}{5} \mu(B(x,r)).  \qedhere    $$
\end{proof}
\begin{lemma}\label{poly 1/4}
    Let $K \subset G$ be a compact subset, $K = S \sqcup X_1 \sqcup \dots \sqcup X_p$ a measurable partition of $K$, and $B \subset K$ a measurable subset. If, for any $i$, $\mu(B\cap X_i) \leq \frac{3}{5}\mu(B)$, then 
    \begin{itemize}
        \item either $\mu(B\cap S) \geq \frac{3}{20}\mu(B)$,
        \item or there exists $J \subset \{1, \dots,p\}$ such that 
        $$ \frac{1}{4}\mu(B) \leq \mu(B \cap (\cup_{j \in J} X_j))\leq \frac{1}{2}\mu(B).$$
    \end{itemize}
\end{lemma}
\begin{proof}
    The proof is very similar to that of Proposition \ref{cut general}. We do it for the sake of clarity. Suppose that $\mu(B\cap S) < \frac{3}{20}\mu(B)$. If there exists $i_0$ such that $\mu(B\cap X_{i_0}) \geq \frac{1}{2}\mu(B)$, then, for $J = \{1, \dots,p\} \backslash \{i_0\}$,
    $$   \mu(B \cap (\cup_{j \in J} X_i)) =  \mu(B) - \mu(B\cap S) -   \mu(B\cap X_{i_0}).  $$
    So
    $$  \frac{1}{4} \mu(B) \leq \left( 1-\frac{3}{20}- \frac{1}{2} \right)  \mu(B) \leq  \mu(B \cap (\cup_{j \in J} X_i)) \leq \frac{1}{2} \mu(B) .    $$
    If, for all $i$, $\mu(B\cap X_{i}) < \frac{1}{2}\mu(B)$, then, since $\mu(B \cap (\cup_{i \in I} X_i)) \geq \frac{3}{20}\mu(B)$, we can find $J \subset \{1, \dots,p\}$ satisfying the desired inequality.
\end{proof}
\begin{proof}[Proof of Theorem \ref{thm cglc poly growth}]
    Let $r_0$ be as in Proposition \ref{persistent in poly}, $\mathcal{S}$ a family of subsets that coarsely separates $G$, and $D \geq 0$ its thickening constant. For $r \geq r_0$, let $S \in \mathcal{S}$ such that there exist two balls $B(x,r)$ and $B(y,r)$ in two different $k$-coarsely connected components of $G \backslash S^{+ D}$. Let $(X_i)_{i \in I}$ be the $k$-coarsely connected components of $G \backslash S^{+ D}$. Since any two different components are at distance $>k$, any compact subset of $G$ intersects only finitely many components. So, up to restricting to a compact subset containing the $r$-neighbourhood of a $k$-path $\gamma$ from $x$ to $y$, we can suppose that $I$ is finite. Moreover, up to replacing each $X_i$ with its closure $\overline{X_i}$ and $S^{+D}$ with $\overline{S^{+D}} \backslash \cup_{i \in I} \overline{X_i} $, we can assume that $S^{+D} \sqcup \bigsqcup_{i \in I} X_i$ is a measurable partition.
    
\medskip \noindent
Using Proposition \ref{persistent in poly}, the same argument of Proposition \ref{main} implies that there exists $z \in \gamma$ such that, for any $i \in I$
    $$    \mu(B(z,r)\cap X_{i}) \leq \left(1-\frac{4/5}{2}\right)\mu(B(z,r)) = \frac{3}{5}\mu(B(z,r)) .$$
    By Lemma \ref{poly 1/4}, either $\mu(B(z,r)\cap S) \geq \frac{3}{20}\mu(B(z,r))$ or there exists $J \subset I$ such that $A = 
 B(z,r) \cap (\cup_{j \in J} X_j)  $ satisfies $\frac{1}{4}\mu(B(z,r)) \leq \mu(A)\leq \frac{1}{2}\mu(B(z,r))$. By \cite[Proposition 8.10]{hume2020poincare}, there exists a constant $C>0$ (independent of $r$) such that 
 $$\mu\left(A^{+1} \cap (A^c)^{+1}\right) \geq C r^{n-1}.$$ 
 However, 
 $$A^{+1} \cap (A^c)^{+1} \subset \left( B(z,r)\cap S^{+ D}\right)^{+1},$$ 
 so 
 $$\mu\left( \left( B(z,r)\cap S^{+ D}\right)^{+1} \right) \geq C r^{n-1}.$$
 It is easy to see that, for any bounded subset $K \subset G$, $\mu(K) \leq \lambda |K|$, where $| \cdot |$ is the $\frac{1}{5}$-volume (see Remark \ref{r s}) and $\lambda$ is the measure of a (any) ball of radius $\frac{1}{5}$. Therefore
 $$\left|  \left( B(z,r)\cap S^{+ D}\right)^{+1} \right| \geq \frac{C}{\lambda} r^{n-1}.$$
As in the proof of Corollary \ref{cor of main}, 
 $$ \beta_G(1) \beta_G(D) \cdot \left|  B(z,r+D)\cap S\right|  \geq \beta_G(1) \cdot \left|  B(z,r)\cap S^{+ D}\right|  \geq \left|  \left( B(z,r)\cap S^{+ D}\right)^{+1} \right|.$$
and by taking $s \in B(z,r+D)\cap S $, there exists a constant $C'$ (independent of $r$) such that for $r$ big enough
\begin{equation}\label{s in poly case}
    | B(s,2r)\cap S| \geq C' r^{n-1}.
\end{equation}
  This concludes the proof.
\end{proof}

\subsection{Rank one symmetric spaces and hyperbolic buildings}
Recall that rank one symmetric spaces of non-compact type are the $\mathbb{H}_{\mathbb{K}}^n$, where $n \geq 2$ and $\mathbb{K}$ is the real numbers $\mathbb{R}$, the complex numbers $\mathbb{C}$, the quaternions $\mathbb{H}$, or the octonions $\mathbb{O}$ (and in the last case $n = 2$). The hyperbolic buildings $I_{p,q}$, for $p \geq 5$ and $q \geq 3$, that we consider here were studied by Bourdon \cite{bourdon1997immeubles} and Bourdon--Pajot \cite{bourdon1999poincare}. The building $I_{p,q}$ is defined as the unique simply connected hyperbolic $2$-complex such that its $2$-cells (chambers) are right-angled hyperbolic $p$-gons, two of its $2$-cells share at most one edge or one vertex, and the link of each vertex is the complete bipartite graph $K_{q,q}$ with $q+q$ vertices. An apartment of $I_{p,q}$ is a copy of the hyperbolic plane of curvature $ -1 $ tiled by the chambers. When equipped with the path metric induced from the hyperbolic plane, it becomes a Gromov hyperbolic space. The graph product of cyclic groups
$$  \Gamma_{p,q} = \langle  s_1, \dots , s_p \mid s_i^q = 1, [s_i, s_{i+1}]=1      \rangle, $$
where the indices are modulo $p$, acts geometrically on $I_{p,q}$. We will also denote by $\Gamma_{p,q}$ its Cayley graph with respect to the generating set $(s_i)_{i}$. So $\Gamma_{p,q}$ and $I_{p,q}$ are quasi-isometric. From now on, we will work with $\Gamma_{p,q}$ instead of $I_{p,q}$, and we always assume $p \geq 5$ and $q \geq 3$. Note that if $p \leq 3$, then $\Gamma_{p,q}$ is a finite group; if $p=4$, then $\Gamma_{p,q}$ is virtually a direct sum of two non-abelian free groups; if $p \geq 5$ and $q=2$, then $\Gamma_{p,q}$ is virtually a surface group.

\medskip \noindent
When $X$ is either a rank one symmetric space of non-compact type $\mathbb{H}_{\mathbb{K}}^n$ or a Cayley graph $\Gamma_{p,q}$, it was shown in \cite{jerison1986poincare} (see also \cite[Section 11]{heinonen1998quasiconformal}) and in \cite{bourdon1999poincare} that its boundary $\partial X$ admits a metric satisfying a $1$-Poincar\'e inequality. More precisely, there exists a visual metric $\rho$ on $\partial X$ such that:
\begin{itemize}
    \item[(i)] There exists $C\geq 1$ such that, for all $ \xi, \eta \in \partial X$ and for any bi-infinite geodesic $\gamma : \mathbb{R} \to  X$ satisfying $\gamma(+\infty) = \xi$ and $\gamma(-\infty) = \eta$, we have
    $$ \frac{1}{C} a^{-d(o,\gamma)} \leq   \rho(\xi, \eta) \leq C a^{-d(o,\gamma)} , $$
    with $a=e$ for symmetric spaces and $a = \exp( \tau(p,2) )$, where $\tau(p,2) = \textup{Argch}(\frac{p-2}{2})$, for $\Gamma_{p,q}$.
    \item[(ii)] Its Hausdorff measure $\mu$ is Ahlfors-regular with dimension $Q_X$, i.e.\ there exists $C'>0$ such that, for all $\xi \in \partial X$ and $r \in (0, \textup{diam}(\partial X) ) $
    $$ \frac{1}{C'} r^{Q_X} \leq \mu(B_{\rho}(\xi,r)) \leq C' r^{Q_X}     ,       $$
with $Q_X = (n+1) \dim_{\mathbb{R}}\mathbb{K}-2 $ when $X =\mathbb{H}_{\mathbb{K}}^n $ and equals $1+ \frac{\log(q-1)}{\tau(p,2)} $ for $\Gamma_{p,q}$.
    \item[(iii)] $(X,\rho,\mu)$ satisfies a $1$-Poincar\'e inequality.
\end{itemize}
We recall that a metric measure space $(Z, d, \mu)$ admits a $1$-Poincar\'e inequality if there exist constants $C$, $\lambda \geq 1$ such that for every ball $B$ in $X$ and every measurable function $u: Z \to \mathbb{R}$, $B$ has measure in $(0,+ \infty)$ and
$$ \dashint_B |u-u_B|d\mu \leq  C \, \textup{diam}B \dashint_{\lambda B}  \rho \, d \mu                 ,$$
where $\rho$ is an upper gradient for $u$, and for any $A \subset Z$ and $f : Z \to R$ a measurable function,
$$ f_A =  \dashint_A f = \frac{1}{\mu(A)} \int_A f d\mu.$$
We refer to \cite{keith2003modulus} for more details and for other equivalent definitions.

\medskip \noindent
We are now ready to prove Theorem \ref{thmIntro:Symmetric/buildings} for rank one symmetric spaces and Bourdon's hyperbolic buildings. We will first give a proof using the Poincar\'e inequality in the boundary, and then we give a second proof for rank one symmetric spaces using Theorem~\ref{thmIntro:nilpot}.
\subsubsection{Proof using the Poincar\'e inequality}
The goal is to apply Theorem \ref{invariance by taking product}. We will consider the family $A_x(r) = S_X(o,r)^{+3}$ for $x \in X$ (based at the vertices for $\Gamma_{p,q}$) and $r>0$. We start by showing that it is an $\alpha$-persistent family, for some $\alpha \in (0,1)$, and then we give a lower bound on their $\cut$, for $\delta = 1 - \frac{\alpha}{2}$.
\begin{lemma}\label{persistent in rank 1}
    Let $X$ be either a rank one symmetric space or the Cayley graph of $\Gamma_{p,q}$. Then the family $A_x(r) = S_X(o,r)^{+3}$, for $x \in X$ (based at the vertices for $\Gamma_{p,q}$) and $r\geq1$, is $\alpha$-persistent with $\alpha = \frac{1}{|B_X(3)|} $.
\end{lemma}
\begin{proof}
See Examples \ref{exple persist}
    $\qedhere$
\end{proof}

\noindent
In \cite{hume2020poincare}, a construction of Bourdon--Pajot \cite{bourdon2003cohomologie} was used to get a pointed simplicial graph $(\Gamma,o)$ that is quasi-isometric to $X$ and whose spheres $S(o,r)$ approximate the Gromov boundary $\partial X$. By pulling-back the Poincar\'e inequality from $\partial X$ to these spheres, they were able to get a lower bound on the $L^1$-Poincar\'e constant of $S(o,r)$. However, this only holds for spheres based at $o$, since the graph $\Gamma$ is not cocompact. Instead of using the graph $\Gamma$, we will use $X$ itself to pull-back the Poincar\'e inequality to its spheres. Thus showing:
\begin{prop}\label{Cut in rank 1}
    Let $X$ be either a rank one symmetric space of non-compact type or the Cayley graph of $\Gamma_{p,q}$. For any $\delta \in (0,1)$, there exists a constant $\gamma$ such that, for any $x \in X$ and any $r>0$,
    $$   \cut(A_x(r)) \geq \gamma \, a^{(Q_X-1)r}.$$
\end{prop}

\noindent
Before starting the proof, let us first state the following lemma.

\begin{lemma}\label{thin ideal triangles}
    Let $X$ be a geodesic $\delta$-hyperbolic metric space. Let $x \in X$, $\gamma$ a bi-infinite geodesic, $\gamma(+\infty) = \eta$, $\gamma(-\infty) = \xi$, and let $\alpha,\beta$ be geodesic rays such that $\alpha(0) = x$, $\alpha(+\infty) = \xi$, $\beta(0) = x$, and $\beta(+\infty) = \eta$. Then there exists $z \in \gamma$ such that 
    $$\max (d(z,\alpha),d(z,\beta)) \leq 10 \delta.$$
\end{lemma}
\begin{proof}
    The ideal triangle $\alpha \cup \beta \cup \gamma$ is $5 \delta$-thin (see \cite[Chapter 11.11]{dructu2018geometric}). 
We conclude by the continuity of the distance function in $\gamma$. $\qedhere$
\end{proof}

\noindent
Let $X$ be either a rank one symmetric space of non-compact type or the Cayley graph of $\Gamma_{p,q}$, and let $\delta$ be a Gromov-hyperbolicity constant of $X$. Set $\epsilon = \max( 103 \delta,1)$. For every integer $r >0$, we note $S_r = S_X(o,r)$ and let $Z_r \subset S_r$ be a maximally $\epsilon$-separated net. So
    \begin{equation*}
     \left\{
        \begin{split}
         &\forall x \neq y \in Z_r, \, d(x,y) > \epsilon \\&
        S_r \subset \cup_{z \in Z_r} B(z,\epsilon)
        \end{split}
        \right..
    \end{equation*}
    Note that $X$ is visual, i.e.\ there exists $D>0$ such that every $x\in X$ is at distance at most $D$ from a geodesic ray emanating from $o$ (see \cite[Lemma 3.1]{bestvina1991boundary}). So, up to replacing $\epsilon$ with $\epsilon+D$ and moving the points of $Z_r$ by a distance at most $D$ in $S_r$, we can suppose that, for every $x\in Z_r$, there is a geodesic ray emanating from $o$ that extends $[o,x]$; let us denote by $\xi_x \in \partial X$ its limit point. Denote 
    $$\Tilde{Z_r} = \{ \xi_x, x\in Z_r \} \subset \partial X.$$
    \begin{claim}\label{claim discret}
        There exist $\lambda,\lambda' >0$ such that, for any $r>0$,
        \begin{equation}\label{discret}
        \left\{
            \begin{split}
                & (i) \, \cup_{\xi \in \Tilde{Z_r}} B_{\rho}(\xi,\lambda' a^{-r} )  = \partial X 
            \\& (ii) \, \, B_{\rho}(\xi,\lambda a^{-r}) \cap B_{\rho}(\eta,\lambda a^{-r}) = \emptyset  ,\, \forall \xi,\eta \in \Tilde{Z_r} 
            \end{split}
            \right..
        \end{equation}
    \end{claim}
    \begin{proof}[Proof of the claim]
   $(i)$ Let $\xi \in \partial X$. Then $[o,\xi)$ intersects $S_r$ at some $p$. Let $x\in Z_r$ such that $p\in B(x,\epsilon)$. Let $\gamma$ be a geodesic satisfying $\gamma(+\infty) = \xi_x$ and $\gamma(-\infty) = \xi$, let $\alpha$ be a geodesic ray from $o$ to $\xi_x$, and let $\beta$ be a geodesic ray from $o$ to $\xi$. There exists a constant $k$ only depending on $X$ and $\delta$ such that $d(o,\gamma) > (\xi_x,\xi)_o - k$, where $(\cdot , \cdot)_o$ denotes the Gromov product based at $o$ (see \cite[Chapter 3, 3.18]{bridson2013metric}). In particular, 
    $$   d(o,\gamma) > \liminf_{s,t \to +\infty} (\alpha(t),\beta(s))_o - k .    $$
    For $t,s >r$, 
    \begin{equation*}
        \begin{split}
            (\alpha(t),\beta(s))_o &= \frac{1}{2}\left( d(o,\alpha(t)) + d(o,\beta(s)) - d(\alpha(t),\beta(s))\right)
            \\ &= \frac{1}{2}\left(r + d(x,\alpha(t)) + r + d(x,\beta(s)) - d(\alpha(t),\beta(s))\right)
            \\ &= r+ \frac{1}{2}\left(d(x,\alpha(t)) + d(x,\beta(s)) - d(\alpha(t),\beta(s))\right).
        \end{split} 
    \end{equation*}
    By the triangle inequality we have
    $$  d(\alpha(t),\beta(s)) \leq   d(x,\alpha(t)) + d(p,x)+ d(p,\beta(s))    . $$
    So
    $$  -\epsilon \leq   d(x,\alpha(t))+ d(p,\beta(s)) -d(\alpha(t),\beta(s))   . $$
    Therefore, $(\alpha(t),\beta(s))_o \geq r- \frac{\epsilon}{2} $ and $d(o,\gamma) \geq r- \frac{\epsilon}{2} - k .$ Hence, for $\lambda' = C a^{k+\epsilon/2}$, we have
    $$   \rho(\xi,\xi_x) \leq \lambda' a^{-r}  .  $$
    $(ii)$ Now, let $x,y \in Z_r$ be two different points, and let us show that $\rho(\xi_x,\xi_y) \geq \frac{1}{C} a^{-r}$. Again, let $\gamma$ be a geodesic satisfying $\gamma(+\infty) = \xi_y$ and $\gamma(-\infty) = \xi_x$. Let $\gamma_x$ be a geodesic ray emanating from $o$ such that $\gamma_x(r) = x$ and $\gamma_x(+ \infty) = \xi_x$, and let $\gamma_y$ be a geodesic ray emanating from $o$ satisfying $\gamma_y(r) = y$ and $\gamma_y(+ \infty) = \xi_y$. By Lemma~\ref{thin ideal triangles}, there exists $z \in \gamma$ such that $\max \left( d(z,\gamma_x),d(z,\gamma_y)  \right)  \leq 10 \delta  $. Let $t,s \in \mathbb{R}$ such that $\gamma_x(t) = z_x$  and $\gamma_y(s) = z_y$, where $d(z,z_x) \leq 10 \delta$ and $d(z,z_y) \leq 10 \delta$.
    
\medskip \noindent
Let us show that either $t$ or $s$ is $\leq r- h$, where $h = 10 \delta$. If not, we have 
    $$d(\gamma_x(t),\gamma_y(s)) = d(z_x, z_y)\leq d(z_x,z)+d(z,z_y) \leq 20 \delta.$$
    By the triangle inequality, we have $|t-s| \leq 20 \delta$, so 
    $$    d(\gamma_x(t),\gamma_y(t))   \leq 40 \delta    . $$
Since $\gamma_x$ and $\gamma_y$ are fellow-travelling (\cite[Lemma 11.4]{dructu2018geometric}),
    $$    d(\gamma_x(r-h),\gamma_y(r-h))   \leq 2( 40 \delta + \delta) = 82 \delta    .              $$
    Hence 
    $$  \epsilon < d(x,y) =  d(\gamma_x(r),\gamma_y(r))  \leq 82 \delta + 2h = 102 \delta        .     $$
    This contradicts $\epsilon \geq 103 \delta$. Suppose that $t \leq r-h$,
    $$      d(o,\gamma) \leq d(o,z_x) + d(z_x,z) = r-h +h .                $$
    We conclude that $\rho(\xi_x,\xi_y) \geq \frac{1}{C} a^{-r}$. $\qedhere$
    \end{proof}
    \begin{proof}[Proof of Propsition \ref{Cut in rank 1}]
    Let $r>0$ be fixed, and let us denote $\rho_r = a^r \rho$ and $\mu_r = a^{Qr}\mu$. Following the proof of  \cite[Theorem 11.1]{hume2020poincare} , the $1$-Poincar\'e inequality satisfied by $\rho$ gives a lower bound on the $L^1$-Poincar\'e constant of $(\partial X, \rho_r, \mu_r )$, in particular there exists $\beta>0$ such that $h_2^1(\partial X, \rho_r, \mu_r ) \geq \beta a^{-r}$. Let $\nu$ be the counting measure on $\Tilde{Z_r}$. Then $(\Tilde{Z_r}, \rho_r, \nu )$ is a discretization of $(\partial X, \rho_r, \mu_r )$, in the sense of \cite[Section 5]{hume2020poincare}. By (\ref{discret}), there exists a measurable partition
    $$   \partial X =  \bigsqcup_{\xi \in \Tilde{Z_r} }   P_{\xi} ,    $$
    where 
     $$ B_{\rho_r}(\xi, \lambda) =   B_{\rho}(\xi, \lambda a^{-r})  \subset  P_{\xi} \subset    B_{\rho}(\xi, \lambda' a^{-r})  = B_{\rho_r}(\xi, \lambda') .    $$
     Moreover, the Ahlfors-regularity of $\mu$ implies that, independently of $r$, we have  
     \begin{equation}\label{ahlfors}
         \lambda^Q \leq   \mu_r(P_{\xi}) \leq \lambda'^Q .
     \end{equation} 
     The Poincar\'e constant $h_2^1$ of a space and that of its discretization are comparable (\cite[Lemma 5.8 and 3.3]{hume2020poincare}): there exists $\beta' >0$ such that 
     $$h_{40}^1(\Tilde{Z_r}, \rho_r, \nu ) \geq \beta' a^{-r}.$$
     By taking $h = \log_a(40C) + 10 \delta $ in the proof of Claim~\ref{claim discret}, we can show that, for all $x,y \in Z_r$, $d(x,y) > D$ implies that $\rho_r(\xi_x,\xi_y)>40$, where $D = \max(2 \log_a(40C)+103 \delta, 2)$. 

\medskip \noindent
Let us also denote by $\nu$ the counting measure in $Z_r$. We have then
     \begin{claim}
         $ h_{D}^1(Z_r, d, \nu ) \geq h_{40}^1(\Tilde{Z_r}, \rho_r, \nu ) .$
     \end{claim}
     \begin{proof}[Proof of the claim]
         We associate to every measurable map $f : (Z_r,d) \to \mathbb{R}$ the map $\Tilde{f} : (\Tilde{Z_r},\rho_r) \to \mathbb{R}$ such that $\Tilde{f}(\xi_x) := f(x) $ for every $ \xi_x \in \Tilde{Z_r}$. It is immediate that $||f||_1 = ||\Tilde{f}||_1$. For all $x \in Z_r$,
         \begin{equation*}
             \begin{split}
                 |\nabla_D f|(x) & = \sup_{y,y' \in B(x,d)}|f(y)-f(y')| = \sup_{y,y' \in B(x,d)}|\Tilde{f}(\xi_y)-\Tilde{f}(\xi_{y'})|
                 \\ & \geq \sup_{\xi_{y},\xi_{y'} \in B_{\rho_r}(\xi_x,40)}|\Tilde{f}(\xi_y)-\Tilde{f}(\xi_{y'})|
                  \geq |\nabla_{40} \Tilde{f}|(\xi_x).\qedhere
             \end{split}
         \end{equation*} 
     \end{proof}
     
\noindent
     Now, let $Z_r'$ be a maximally $\eps$-separated subset of $S_r^{+3}$ that contains $Z_r$. This is always possible by starting from $Z_r$ thanks to Zorn's lemma. Again, $Z_r$ is a discretization of $Z_r'$, 
     \begin{equation*}
         \left\{
     \begin{split}
        &\forall x \neq y \in Z_r', \, d(x,y) > \epsilon
        \\& Z_r' \subset \cup_{z \in Z_r} B(z,\epsilon+3)
        \end{split}
\right..
     \end{equation*}
So there exist $D'\geq D$, $\beta''>0$ such that 
$$h_{D'}^1(Z_r', d, \nu ) \geq \beta'' h_{D}^1(Z_r, d, \nu ) .$$
     Since $X$ is $1$-coarsely geodesic and $D' \geq 1$, $X$ is $D'$-coarsely geodesic. Because $D' \geq \epsilon \geq 1$, it follows from Proposition~\ref{comparison cheeger Poincar\'e} that there exists a constant $M>0$, only depending on $X$, $D$, and $\eps$, such that, for any $r>0$,
     \begin{equation}\label{cheeger Sr}
          M \, h_D(S_r^{+3}) \geq h_{D}^1(Z_r', d, \nu )     \geq \beta'' a^{-r}.
     \end{equation}
     The inequality (\ref{ahlfors}) implies that 
     $$  C_1 \nu(\Tilde{Z_r}) \leq     a^{Qr}\mu(\partial X) \leq C_2 \nu(\Tilde{Z_r}) ,      $$
     where $C_1 = \lambda^Q$ and $C_2 = \lambda'^Q$.
     By construction, $\nu(\Tilde{Z_r}) = \nu(Z_r)$ and $Z_r$ is maximally $\eps$-separated in $S_r$, so, by Lemma \ref{volume and counting measure},
     \begin{equation*}
         C_2' a^{Qr}   \leq \frac{1}{\alpha}  \nu(Z_r) \leq |S_r|  \leq \alpha \, \nu(Z_r) \leq C_1' a^{Qr},
     \end{equation*}
where $C_1' = \frac{\alpha \, \mu(\partial X)}{C_1}  $ and $C_2' = \frac{ \mu(\partial X)}{ \alpha \, C_2}  $. Therefore, there exists $C_1''>0$ such that
 \begin{equation}\label{taille Sr}
         C_2' a^{Qr}  \leq |S_r^{+3}|  \leq C_1'' a^{Qr}.
     \end{equation}
We conclude from the inequalities (\ref{cheeger Sr}) and (\ref{taille Sr}), and using Proposition \ref{cut general}, that, for any $\delta \in (0,1)$, there exists a constant $\gamma$ independent of $r$ such that
\begin{equation*}
    \cut_D(S_r^{+3}) \geq \gamma \, a^{(Q-1)r}. \qedhere
\end{equation*}
\end{proof}
\noindent
Since $Q=1$ only when $X = \mathbb{H}_{\mathbb{R}}^2$ and $Q>1$ in all the other cases, the proof of Theorem~\ref{thmIntro:Symmetric/buildings} for rank one symmetric spaces and Bourdon's hyperbolic buildings follows from Lemma \ref{persistent in rank 1} and Proposition \ref{Cut in rank 1} by applying Theorem \ref{invariance by taking product}.

\subsubsection{Proof using the polynomial growth of horospheres in rank one}
 We give a second proof of Theorem \ref{thmIntro:Symmetric/buildings} for rank one symmetric spaces of non-compact type,  which does not use the Poincar\'e inequality in the boundary, but instead uses Theorem \ref{thm cglc poly growth} and the fact that horospheres in rank one symmetric spaces are exponentially distorted nilpotent Lie groups. 
\begin{proof}
    Let $X$ be a rank one symmetric space of non-compact type, $\mathcal{S}$ a family of subsets that coarsely separates $X$, and $D \geq 0$ its thickening constant. Let $r >0$ be big enough and $S \in \mathcal{S}$ such that there exist two balls $B_X(x,r)$ and $B_X(y,r)$ in two different coarsely connected components of $X \backslash S^{+ D}$. In rank one symmetric spaces, $\textup{Isom}_0(X)$ acts transitively on the set of geodesic rays of $X$, so all horospheres are isometric; and, for any pair of points in $X$, there is a horosphere containing both of them. Let $H$ be a horosphere containing $x$ and $y$ and let $N<\textup{Isom}_0(X)$ be the nilpotent Lie group that stabilizes $H$. By results of Bass \cite{bass1972degree} and Guivarc'h \cite{guivarc1973croissance}, there exists $d \geq 1$ such that $\beta_{G}(n) \simeq n^d$. Equipped with a left-invariant Riemannian metric, $N$ is bi-Lipschitz equivalent to $(H,d_H)$ by the orbit map, where $d_H$ is the path metric along the horosphere $H$. Note that the bi-Lipschitz constants are independent of $x$ and $y$ since all horospheres are isometric. Up to this bi-Lipschitz map, and to simplify notations, Theorem \ref{thm cglc poly growth} applies to $(H,d_H)$. Moreover, it is known that horospheres are exponentially distorted in rank one. In particular, there exists $\lambda >0$ such that, for any $z \in H$ and $R>0$,
    $$ B_H\left(z,e^{\lambda R}\right) \subset B_X(z, R) .$$
    Therefore, $B_H\left(x,e^{\lambda r}\right)$ and $B_H\left(y,e^{\lambda r}\right)$ are in different coarsely connected components of $H \backslash S^{+ D}$. By the proof of Theorem \ref{thm cglc poly growth}, see the inequality \ref{s in poly case}, there exists $s \in H \cap S$ such that 
    $$| B_H\left(s,2e^{\lambda r} \right)\cap S| \geq k e^{\lambda (d-1) r}.$$
Hence 
$$| B_X\left(s,2r \right)\cap S| \geq k e^{\lambda (d-1) r}.$$
   Since $d$ is equal to $1$ if and only if $X = \mathbb{H}_{\mathbb{R}}^2 $, this concludes the proof.
\end{proof}

\subsection{Higher rank symmetric spaces and Euclidean buildings}
In this section, we prove Theorem \ref{thmIntro:Symmetric/buildings} for higher rank symmetric spaces of non-compact type and Euclidean buildings. We refer to \cite{kleiner1997rigidity} for the background material on Euclidean buildings.

\medskip \noindent
We will start by giving the proof for products of $3$-regular trees $T_3 \times \dots \times T_3 $, then using suitable embeddings of such products into symmetric spaces and Euclidean buildings, we can conclude using Theorem \ref{separating by embedding}. Let us now focus on $T_3 \times T_3 $.

\medskip \noindent
Let us first recall the definition of Busemann functions. Let $X$ be a CAT(0) space, $x \in X$, $\xi \in \partial X$, and $\gamma$ the geodesic ray from $x$ to $\xi$. The \emph{Busemann function} $b_{\gamma} : X \to \mathbb{R}$ of $X$ with respect to $\gamma$, or equivalently with respect to $x$ and $\xi$, is defined by  
$$    b_{\gamma}(z) = \lim_{t \to + \infty} (t - d(z, \gamma(t)).$$
Usually, the Busemann function is defined as $-b_{\gamma}$, but we choose this convention for convenience.

\medskip \noindent
For $i=1,2$, let $X_i$ be a copy of $T_3$, fix some $\xi_i \in \partial X_i$, and let $b_i$ denote a Busemann function of $X_i$ with respect to $\xi_i$ and some base point, which we just denote by $b$ for short (it will be clear from the context whether $b$ refers to the Busemann function for $X_1$ or $X_2$). Let us denote $X=X_1 \times X_2$. Here, we only work with the product of two trees for clarity, but our argument applies to products of more trees without major modifications.

\medskip \noindent
\textbf{Notation:} For $i=1,2$ and $x \in X_i$ a vertex:
\begin{itemize}
    \item  we let $\hat{x}^k$ be its \emph{$k$-th ancestor}, i.e.\ the vertex in the geodesic ray $[x, \xi_i)$ that is at distance $k$ from $x$;
    \item we define the \emph{binary tree below $x$}, denoted $T(x) \subset X_i$, as the connected component of $\{  z \in X_i \mid b(z) \leq b(x)         \} $ containing $x$.
\end{itemize}
For $x=(x_1,x_2) \in X$ and $k \in \mathbb{N}$, define $A_{x}(k)$ as the subgraph of $X$ induced by the vertices $(T(x_1) \times T(x_2))\bigcap ( S_X(x,k-1)\cup S_X(x,k))$. It is isomorphic  to an annulus $(k-1,k)$ around the origin in a product of two binary rooted trees.
\begin{prop}\label{persistent family of product trees}
    $\left\{ A_{x}(k) \right\}_{x\in X, k \in \mathbb{N}^*}$ is a $\frac{1}{8}$-persistent family. Moreover, there exists $C>0$ such that, for any $x \in X$ and $k \in \mathbb{N}^*$,
    $$ \textup{cut}^{\frac{15}{16}}((A_{x}(k)) \geq  C \, 2^k . $$
\end{prop}
\begin{proof}
For any $x = (x_1,x_2) \in X$ and $k \in \mathbb{N}^*$, 
\begin{equation}\label{size aok}
     |A_x(k)| = |A_x(k)\cap S(x,k-1)| + |A_x(k)\cap S(x,k)| \leq 2 |A_x(k)\cap S(x,k)| = (k+1) 2^{k+1} .
\end{equation}
Let $(x_1,x_2),(x_1',x_2) \in X $ such that $d(x_1,x_1') = 1$. If $b(x_1) < b(x_1')$, then $A_{(x_1,x_2)}(k)\cap A_{(x_1',x_2)}(k)$ contains $A_{(x_1,x_2)}(k-1)$ and 
$$  | A_{(x_1,x_2)}(k)\cap A_{(x_1',x_2)}(k) | \geq k \,2^{k-1}.   $$
Same applies if $b(x_1) < b(x_1')$, and also for the second coordinate. Hence, for all $(x_1,x_2),(x_1',y_2') \in X$ satisfying $d((x_1,x_2),(x_1',x_2')) = 1$,
$$  \frac{| A_{(x_1,x_2)}(k)\cap A_{(x_1',x_2)}(k) |}{|A(k)|} \geq \frac{k \,2^{k-1}}{(k+1) 2^{k+1}}\geq \frac{1}{8}.   $$
Therefore, $\left\{ A_{x}(k) \right\}_{x\in X, k \in \mathbb{N}^*}$ is $\frac{1}{8}$-persistent. Let us now give a lower bound of its $L^1$-Poincar\'e constant (at scale $1$). Let $o = (o_1,o_2) \in X$ and $k \in \mathbb{N}$ be fixed. The elements of $A_o(k)$ can be represented, as described in the proof of \cite[Theorem 3.5]{benjamini2012separation}, by binary sequences of length $\in \{k-1,k\}$ with a dot at some place inside. The dot represents the root $o$, and two such sequences are neighbours if one is obtained from the other by adding a digit either on the right or on the left of the sequence. Therefore, for an element $(x,y) \in A_o(k)$, the sub-sequence to the right of the dot represents $y$, and the one to the left of the dot, read from the dot, represents $x$.

\medskip \noindent
We consider the following paths, also described in \cite{benjamini2012separation}. To go from $(x,y)$ to $(x',y')$ in $A_o(k)$, we take the following path in three steps:
\begin{description}
    \item[Step 1:] If $(x,y) \in A_o(k) \cap S(o,k) $, we delete the leftmost digit in $(x,y)$ and add a digit to the right of the sequence. If $(x,y) \in A_o(k) \cap S(o,k-1)$, start by adding a digit to the right of the sequence, then delete the leftmost digit. Do so repeatedly until the dot is to the left of the sequence, i.e.\ until reaching a word of the form $(o_1,z)$.
    \item[Step 2:] Now add the first digit of $x'$ to the left, and delete the rightmost digit. Follow this process until having $x'$ to the left of the dot (no choices to make), and continue until the dot is to the right of the sequence, i.e.\ until reaching a word of the form $(z',o_2)$.
    \item[Step 3:] Add the first digit of $y'$ to the right, and delete the leftmost digit. Follow this process until reaching $(x',y')$.
\end{description}
For $0\leq s \leq k$, let $R^s = \{ (x,y) \in A_o(k) \mid  d(o_2,y)=s  \}$. We will follow the same steps as in the proof of \cite[Theorem 3.4]{hume2022poincare}. For $(x,y),(x',y') \in A_o(k)$ such that $(x,y) \in R^t$, $(x',y') \in R^s$, and $s \leq t$, we assign the family of paths from $(x,y)$ to $(x',y')$ described before. Let us denote it by $C_{(x,y),(x',y')}$. There are $2^{(k-1)-t}$ choices of sub-paths in the first step. In the second step, there are no choices to make before reaching $x'$ to the right of the dot. For the rest of the path, we can make $2^{s-1}$ choices. Finally, there are no choices to make in step 3. Therefore, there are $2^{k-t+s-2}$ paths in $C_{(x,y),(x',y')}$. 

\medskip \noindent
Moreover, each edge of $A_o(k)$ appears in at most $2^{2k-t+s}$ paths from some point in $R^t$ to some point in $R^s$. Indeed, an edge appears in the first step of some path from $R^t$ to $R^s$ at most $2^{k-t+s}$ times: if the edge $(a,b)(a,b')$ appears in the first step of some path in $C_{(x,y),(x',y')}$, then $d(o_2,b) \geq t $, $y$ is uniquely determined and there are at most $2^{d(o_2,b)-t+1}$ choices for $x$. After reaching $(a,b')$, there are $2^{k-1-d(o_2,b')} = 2^{k-2-d(o_2,b)}$ choices of paths to complete the first step, and there are $2^s$ choices to complete step $2$, see Figure \ref{fig path in tree}.
\begin{figure}[!ht]
    \centering
    \includegraphics[width=0.9\linewidth]{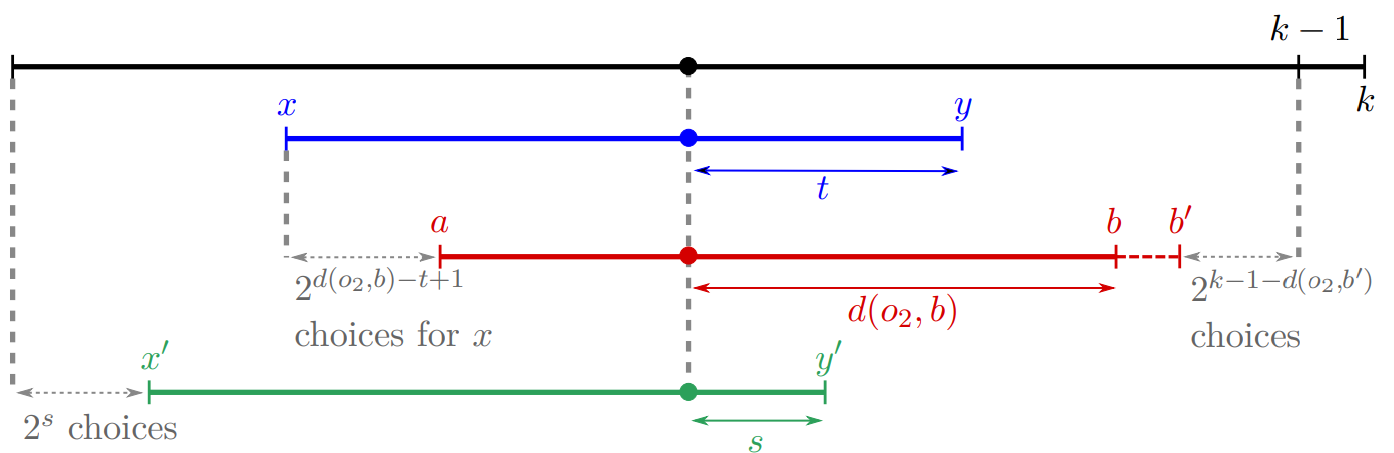}
    \caption{Paths from $R^t$ to $R^s$ passing through the edge $(a,b)(a,b')$ in step $1$.}
    \label{fig path in tree}
\end{figure}
\\The same computations imply that an edge appears in step $3$ at most $2^{k-t+s}$ times. However, if $(a,b)(a,b')$ appears in the second step and, say, $t \leq d(o_2,b)$, then $y$ is uniquely determined and we are left with $2^{k-t}$ choices for $x$, $2^{k-d(o_2,b)}$ choices of paths at the end of step $1$, $2^{d(o_2,b)}$ choices to complete step $2$, and $2^s$ choices in step 3. So the edge appears in at most $2^{2k-t+s}$ paths from $R^t$ to $R^s$. The other cases are dealt with similarly.

\medskip \noindent
Let $f : A_o(k) \to \mathbb{R}$ be a non-constant, zero-mean function. If $e = (a,b)$ is an edge in $A_o(k)$, we denote $|\nabla f(e)| = |f(b)-f(a)|$. Since every vertex in $A_o(k)$ has valency $2$ or $3$, we have
$$   \sum_{e \in E(A_o(k))}  |\nabla f(e)| \leq \frac{3}{2}  ||\nabla f||_1   .         $$
By the triangle inequality along all the paths considered, and averaging, we have
$$   \sum_{z,z' \in A_o(k)}  |f(z')-f(z)| \leq 2 \sum_{t \geq s}\sum_{x \in R^t} \sum_{y \in R^s} \frac{1}{2^{k-t+s-2}}  \sum_{\gamma \in C_{x,y}} \sum_{e \in E(\gamma)}  |\nabla f(e)|.     $$
Since each edge of $A_o(k)$ appears in at most $2^{2k-t+s}$ paths from some point in $R^t$ to some point in $R^s$, we have
$$    \sum_{x \in R^t} \sum_{y \in R^s}   \sum_{\gamma \in C_{x,y}} \sum_{e \in E(\gamma)} \frac{1}{2^{2k-t+s}}  |\nabla f(e)|       \leq  \sum_{e \in E(A_o(k))}  |\nabla f(e)| .      $$
Therefore
\begin{equation}\label{nabla f}
    \begin{split}
        \sum_{z,z' \in A_o(k)}  |f(z')-f(z)| &\leq   2  \sum_{t \geq s} \frac{2^{2k-t+s}}{2^{k-t+s-2}} \sum_{e \in E(A_o(k))}  |\nabla f(e)| 
        \\&=  2^{k+2} (k+1)(k+2) \sum_{e \in E(A_o(k))}  |\nabla f(e)|
        \\& \leq  \frac{2^{k+1} (k+1)(k+2)}{3} ||\nabla f||_1 .
    \end{split}
\end{equation}
Now, $\sum_{z \in A_o(k)} f(z) = 0$ implies 
\begin{equation*}
    \begin{split}
        |A_o(k)| \, ||f||_1  &=  |A_o(k)| \sum_{z \in A_o(k)}  \Big| f(z) - \frac{1}{|A_o(k)|}\sum_{z' \in A_o(k)} f(z')  \Big|  
        \\&\leq \sum_{z,z' \in A_o(k)}  | f(z) - f(z') |
        \\&\leq \frac{2^{k+1} (k+1)(k+2)}{3} ||\nabla f||_1.
    \end{split}
\end{equation*}
Using (\ref{size aok}), we conclude that 
\begin{equation}\label{h1 of product }
    h^1(A_o(k))\geq \frac{3}{4(k+2)} .
\end{equation}
For $\delta = 1-\frac{1/8}{2} = \frac{15}{16}$, by applying Proposition \ref{cut general}, Proposition \ref{comparison cheeger Poincar\'e}, and (\ref{h1 of product }), there exist constants $\lambda >0$ and $C>0$ such that
$$   \cut(A_x(k)) \geq \lambda \, h_1(A_{x}(k)) |A_x(k)| \geq \frac{3/4 (k+1)}{32 \, M(k+2)} 2^k   \geq C \,  2^k.    \qedhere$$
\end{proof}

\noindent
Thus, we recover \cite[Theorem 3.5]{benjamini2012separation}, and the same proof works for any finite product $T_3 \times \dots \times T_3 $ without major modifications.

\medskip \noindent
In \cite{bensaidnguyen}, and building on results of \cite{fisher2018quasi}, it is shown that:

\begin{theorem*}[\cite{bensaidnguyen}\cite{fisher2018quasi}]
    Let $X$ be either a symmetric space of non-compact type of rank $k$ or a thick Euclidean building of rank $k$. There exist constants $\lambda \geq 1, K,M \geq 0$ such that, for any $x,y \in X$, there exists a $(\lambda,K)$-quasi-isometric embedding of the product of $k$ copies of $T_3$ into $X$ such that $x,y \in f(X)^{+M}$.
\end{theorem*}

\noindent
Thus Theorem \ref{thmIntro:Symmetric/buildings} for higher rank symmetric spaces and Euclidean buildings is an immediate consequence of Theorem \ref{separating by embedding}.

\subsection{Horocyclic products}\label{horo}
This section is dedicated to the proof of Theorem \ref{thmIntro:horocyclic products}.

\medskip \noindent
For $i=1,2$, let $X_i$ be regular trees, $\gamma_i$ a geodesic ray in $X_i$ such that $\gamma_i(0) \in V(X_i)$, and $b_i$ the Busemann function of $X_i$ with respect to $\gamma_i$. The \emph{horocyclic product} $X_1 \bowtie X_2$ is defined as the graph whose vertex set is 
$$ V(X_1 \bowtie X_2)  = \{  (x_1,x_2) \in X_1 \times X_2 \mid b_1(x_1)+b_2(x_2)=0  \},  $$
and two vertices $(x_1,x_2),(x_1',x_2')$ are joined by an edge if $d(x_1,x_2)=d(x_1',x_2')=1$. The horocyclic product $T_{p+1} \bowtie T_{q+1}$, $p,q \geq 2$, is called the Diestel-Leader graph $\dl(p,q)$. 

\medskip \noindent
In the general case, for $i=1,2$, let $X_i$ be either a $(k+1)$-regular tree or a rank one symmetric space of non-compact type, $\gamma_i$ a geodesic ray in $X_i$, and $b_i$ the Busemann function of $X_i$ with respect to $\gamma_i$. The \emph{horocyclic product} $X_1 \bowtie X_2$ is defined as 
$$ X_1 \bowtie X_2  = \{  (x_1,x_2) \in X_1 \times X_2 \mid b_1(x_1)+b_2(x_2)=0  \},  $$
equipped with the $L^1$ product metric.

\medskip \noindent
For example, if $n\geq 2$, $T_{n+1} \bowtie \mathbb{H}_{\mathbb{R}}^2$ is quasi-isometric to the solvable Baumslag-Solitar group $\textup{BS}(1,n) = \langle a,b \mid b a b^{-1} = a^n \rangle$, whereas $\mathbb{H}_{\mathbb{R}}^2 \bowtie \mathbb{H}_{\mathbb{R}}^2$ is quasi-isometric to the solvable Lie group $\textup{SOL}_3$ (this is well-known: see for instance \cite[\S 3.2]{hume2022poincare}).
\begin{prop}[\cite{hume2022poincare}]\label{HMT}
    There exists $\lambda \geq 1$ such that, if $\gamma_1$ is a geodesic ray in $T_3$, $b_{1}$ the associated Busemann function, $\gamma_2$ a geodesic ray in $\mathbb{H}_{\mathbb{R}}^2$, and $b_2$ the associated Busemann function, then there exists a $\lambda$-biLipschitz embedding $f : T_3 \to \mathbb{H}_{\mathbb{R}}^2$ such that $f(\gamma_1) = \gamma_2$ and $b_2(f(x)) = b_1(x)$. Moreover, this embedding induces a $\lambda$-biLipschitz embedding $\dl(2,2) \to \mathbb{H}_{\mathbb{R}}^2 \bowtie \mathbb{H}_{\mathbb{R}}^2 $.
\end{prop}

\noindent
 Since $\mathbb{H}_{\mathbb{R}}^2$ isometrically embeds into any rank one symmetric space of non-compact type $X$, and $\textup{Isom}(X)$ acts transitively on its set of geodesic rays, then $\mathbb{H}_{\mathbb{R}}^2 \bowtie \mathbb{H}_{\mathbb{R}}^2 $ can be replaced in Proposition \ref{HMT} with $X_1 \bowtie X_2 $, for $X_i$ any rank one symmetric space of non-compact type or $T_{k+1}$.

\medskip \noindent
Let $X_1,X_2$ two copies of $T_3$. For $i=1,2$, let $\gamma_i$ be a geodesic ray in $X_i$ such that $\gamma_i(0) \in V(X_i)$, and denote $\xi_i = \gamma_i(+ \infty)$. Let $b_{i}$ be the corresponding Busemann functions, which we will just denote by $b$, and let $X = X_1 \bowtie X_2$ the horocyclic product with respect to the Busemann functions.

\medskip \noindent
Let $r \in \mathbb{N}$ and $(o_1,o_2) \in X_1 \times X_2$ such that $b(o_1) + b(o_2) = r$. Consider the following subset of $X$, see \cite[Theorem 4.3]{hume2022poincare}.
\begin{equation*}
V_{(o_1,o_2)}(r) = \left\{ (z_1,z_2) \in X:  \, \,
    \begin{split}
           & d_{X_1}(o_1,z_1) = r - b(z_1),\\&  d_{X_2}(o_2,z_2) = - b(z_2),\\& 0 \leq d_{X_1}(o_1, z_1), d_{X_2}(o_2, z_2) \leq r   
    \end{split}
    \right\}.
\end{equation*}
For $x = (x_1,x_2) \in X$ and $r \in \mathbb{N}$, we consider $A_x(r) = V_{(x_1^r,x_2)}(r) \subset X$. Recall that $x_1^r$ is the $r$-th ancestor of $x_1$ in $X_1$, see Figure \ref{fig persist in horo}.
\begin{figure}[!ht]
    \centering
    \includegraphics[width=\linewidth]{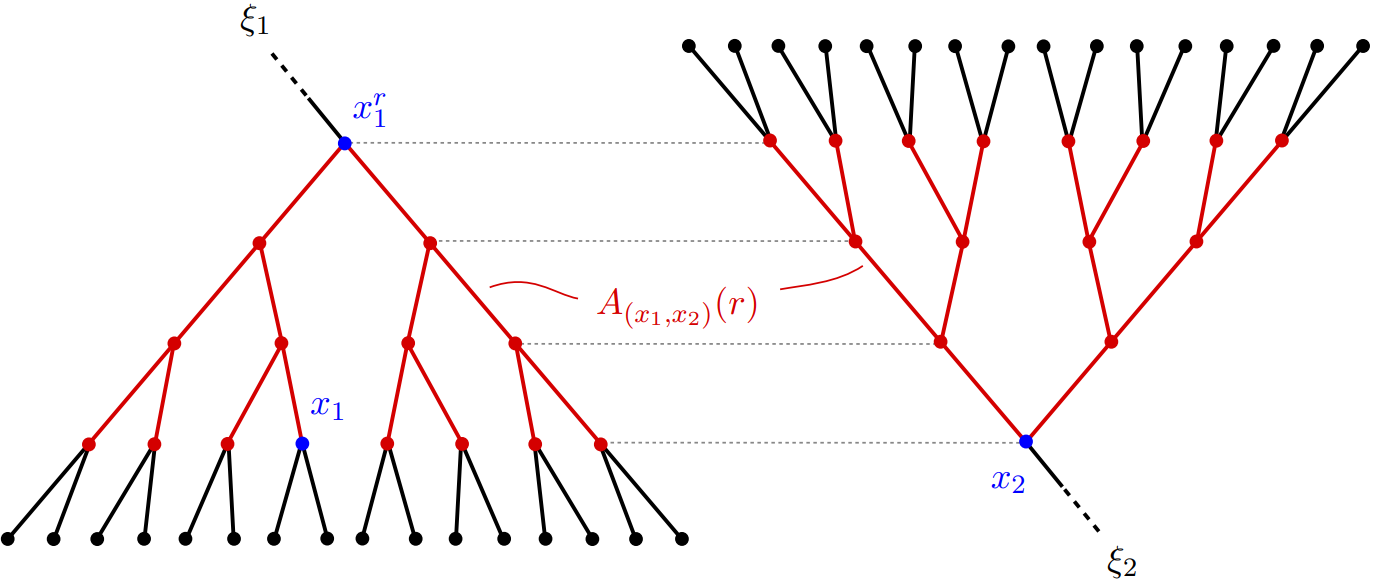}
    \caption{$A_x(r)\subset X$, for $r=3$.}
    \label{fig persist in horo}
\end{figure}
\begin{prop}\label{persistent family of dl}
    $\left\{ A_{x}(r) \right\}_{x\in X, r \in \mathbb{N}^*}$ is a $\frac{1}{4}$-persistent family in $X$. Moreover, there exists $C>0$ such that, for any $x \in X$ and $r \in \mathbb{N}^*$,
    $$ \cut(A_{x}(r)) \geq  C \, 2^r  $$
    where $\delta = \frac{7}{8}$.
\end{prop}
\begin{proof}
Let $x = (x_1,x_2) \in X$ and $r \in \mathbb{N}^*$. We have
$$A_x(r) \subset B_{X_1}(x_1^r,r)\times B_{X_2}(x_2,r) \subset B_{X_1}(x_1,2r)\times B_{X_2}(x_2,r) \subset B_{X}(x,4r).$$
Let $x' = (x_1',x_2') \in X$ be adjacent to $(x_1,x_2)$. If $b(x_1') = b(x_1)-1$, then $ A_x(r-1) \subset  A_{x'}(r) $ and 
$$    |A_x(r) \cap A_{x'}(r)| \geq |A_x(r-1)|   .     $$
If $b(x_1') = b(x_1)+1$, then $ A_{x'}(r-1) \subset  A_{x}(r) $ and 
$$    |A_x(r) \cap A_{x'}(r)| \geq |A_{x'}(r-1)|   .  $$
We have $|A(r)| =(r+1)2^r $, so $|A(r-1)| \geq \frac{1}{4}|A(r)| $. Therefore,
$$    |A_x(r) \cap A_{x'}(r)| \geq \frac{1}{4} |A(r)|   .  $$
It was shown in \cite[Theorem 4.3]{hume2022poincare}, and using \cite[Proposition 6.10]{hume2020poincare}, that there exists a constant $C'>0$ such that $h_1(A_x(r)) \geq C'/r$. Therefore, by Proposition~\ref{cut general}, there exist $\lambda >0$ and $C>0$ such that
$$ \cut(A_x(r)) \geq \lambda \, h_1(A_x(r)) |A_x(r)| \geq \frac{\lambda \, C'}{r} (r+1)2^r  \geq   C \, 2^r. \qedhere$$
\end{proof}
\begin{cor}
    $\dl(2,2) \in \mathfrak{M}_{exp}$.
\end{cor}
\begin{proof}
    $\dl(2,2)$ is $1$-coarsely connected, so the result is a consequence of Theorem \ref{invariance by taking product}.\qedhere
\end{proof}

\noindent
For the general case $Y_1 \bowtie Y_2$, we use suitable embeddings of $\dl(2,2)$ into $Y_1 \bowtie Y_2$. The idea is to apply a similar argument as for higher rank symmetric spaces and Euclidean buildings. We will (almost) join every pair of points in $ Y = Y_1 \bowtie Y_2$ by the images of three $\lambda$-biLipschitz embeddings of $\dl(2,2)$, for a fixed $\lambda$. Moving the persistent family of subsets of $\dl(2,2)$ along these embeddings yields a big intersection with the separating family of subsets.

\begin{proof}[Proof of Theorem \ref{thmIntro:horocyclic products}]
For $i=1,2$, let $\xi_i \in \partial X_i$ be the boundary point corresponding to the Busemann function $b_{Y_i}$. We know that $Y= Y_1 \bowtie Y_2 $ is coarsely $1$-connected. Let $\lambda$ be the Lipschitz constant of Proposition~\ref{HMT}. Since $\lambda \geq 1$, it is also coarsely $\lambda$-connected. 
In what follows, we fix the coarse connectivity constant of $Y$ to be $\lambda$.

\medskip \noindent
Let $\mathcal{S}$ be a family of subsets that coarsely separates $Y$, with thickening constant $D \geq \lambda$. In particular, $\mathcal{S}$ with thickening constant $D+2$ is also coarsely separating. Let $r \in \mathbb{N}^*$ be fixed. For $R = 4 \lambda r$, there exist $S \in \mathcal{S}$ and $(y_1,y_2),(y_1',y_2')\in Y$ such that $B_Y((y_1,y_2),R)$ and $B_Y((y_1',y_2'),R)$ are in different $\lambda$-coarsely connected components of $Y \backslash S^{+(D+2)}$. Let us denote by $(C_i)_{i \in I}$ the $\lambda$-coarsely connected components of $Y \backslash S^{+(D+2)}$. Without loss of generality, we assume that, for $i = 1,2$, $b_{Y_i}(y_i) = 0$; and let $\gamma_i$ be a geodesic in $Y_i$ satisfying $\gamma_i(0) = y_i$ and $\gamma_i(+ \infty) = \xi_i$. Up to moving $(y_1',y_2')$ by a distance $\leq 1/2$, we may also assume that $b_{Y_1}(y_1) = - b_{Y_2}(y_2) \in \mathbb{Z}$.

\medskip \noindent
For $i=1,2$, let $X_i$ be a copy of $T_3$, $x_i \in X_i$, $\eta_i \in \partial X_i$, $\beta_i$ a bi-infinite geodesic such that $\beta_i(0) = x_i$ and $\beta_i(+\infty) = \eta_i$. Let us also denote by $b_{X_i}$ the Busemann function of $X_i$ with respect to $x_i$ and $\eta_i$. By Proposition \ref{HMT}, for i=1,2, let $f_i : X_i \to Y_i$ be the $\lambda$-biLipschitz embedding such that $f_i(\beta_i(t)) = \gamma_i(t)$ and $b_{Y_i}(f_i(x)) = b_{X_i}(x)$. Let us denote by $f : X \to Y$ the induced embedding of the horocyclic products, i.e.\ the restriction of $(f_1,f_2) : X_1 \times X_2 \to Y_1 \times Y_2$ to $X = X_1 \bowtie X_2$. Note that $f$ sends isometrically the geodesic $(\beta_1(t),\beta_2(-t))_{t \in \mathbb{R}}$ to the geodesic $(\gamma_1(s),\gamma_2(-s))_{s \in \mathbb{R}}$.

\medskip \noindent
Let us also consider, for $i=1,2$, the geodesic $\gamma_i'$ in $Y_i$ passing through $y_i'$ satisfying $\gamma_i'(+\infty) = \xi_i$ and $b_{Y_i}(\gamma_i'(t)) = t$. Take $\phi_i \in \textup{Isom}(Y_i)$ the parabolic isometry such that $\phi_i(\gamma_i(t))= \gamma_i'(t)$, i.e.\ $\phi_i$ is in the nilpotent part $U$ of the Iwasawa decomposition $\textup{Isom}_0(X) = KAU$ with respect to $\gamma_i(0)$ and $\gamma_i(+ \infty)$ (note that in rank one, $U$ acts simply transitively on the set of geodesics $c$ such that $c(+\infty) = \gamma_i(+ \infty)$). In particular, $(\phi_1,\textup{Id}_{Y_2})$ and $(\textup{Id}_{Y_1},\phi_2)$ are isometries of $Y$. See Figure \ref{fig setting horocyclic prod} for the setting (where $Y_1=Y_2= \mathbb{H}^2_{\mathbb{R}}$).

\medskip \noindent
\begin{lemma}\label{asymptotic geod in rankone}
Let $X$ be a symmetric space of non-compact type, $\gamma$ a regular geodesic, $\textup{Isom}_0(X) = KAU$ the Iwasawa decomposition with respect to $\gamma (0)$ and $\gamma(+\infty)$, and $\phi \in U$. Then for any $\varepsilon>0$, and $R \geq0$, there exists $N \geq 0$ such that for $t \geq N$ and for any $z \in B_X(\gamma(t),R)$, $d_X(z, \phi(z)) \leq \varepsilon$.
\end{lemma}
\begin{proof}
Let us denote $\gamma(0) = x_0$, and let $X \in \textup{Lie}(A)$ such that $\gamma(t) = e^{tX}.x_0$. $\phi \in U$ so $d_X(e^{-tX}\phi e^{tX}.x_0,x_0) \xrightarrow[t \to + \infty] \, 0 $. Since $AU$ acts simply transitively on $X$, the orbit map $AU \to X$, $g \to g.x_0$ is a diffeomorphism and $ e^{-tX}\phi e^{tX} \xrightarrow[t \to + \infty] \, \textup{Id} $. By continuity of the orbit map and of the product in $AU$, for any $x \in X$ $d_X(e^{-tX}\phi e^{tX}.x,x) \xrightarrow[t \to + \infty] \, 0 $. Therefore, by compacity of the closed ball of radius $R$ centered at $x_0$, there exists $N \geq 0$ such that for $t \geq N$ and for any $x \in B_X(x_0,R)$, 
$$ d_X(\phi e^{tX}.x,e^{tX}x) = d_X(e^{-tX}\phi e^{tX}.x,x)  \leq \varepsilon . $$
The conclusion follows since $B_X(\gamma(t),R) = e^{tX}. B_X(x_0,R)$.
\end{proof}
For $i=1,2$, $d_{Y_i}(\gamma_i(t),\gamma_i'(t)) \xrightarrow[t \to + \infty] \, 0 $, so by the previous lemma, there exists an $N \in \mathbb{N}$ such that $t\geq N$ implies that, for any $z \in B_{Y_i}(\gamma_i(t),r)$, $d_{Y_i}(z,\phi_i(z)) < \frac{1}{2}$. In particular for any $z =(z_1,z_2) \in \Gamma_N$, 
\begin{equation}\label{<1/2}
    d_{Y_1}\left(f_1(z_1),\phi_1((f_1(z_1))\right) < 1/2 .
\end{equation}
\begin{figure}[!ht]
    \centering
    \includegraphics[width=0.9\linewidth]{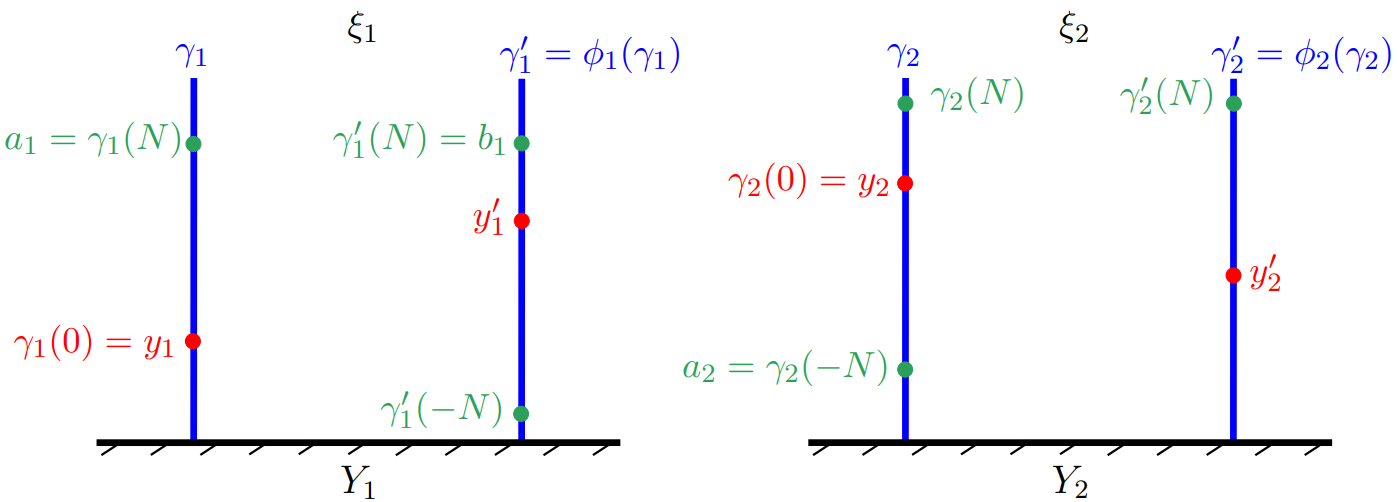}
    \caption{The setting if $Y_1=Y_2= \mathbb{H}^2_{\mathbb{R}}$ and $\xi_i$ are the points $y=\infty$.}
    \label{fig setting horocyclic prod}
\end{figure}

\medskip \noindent
Let us consider the following points (see Figure \ref{fig setting horocyclic prod}): 
\begin{equation*}
    \begin{split}
        & a_1 = f_1(\beta_1(N)) = \gamma_1(N);
        \\& a_2 = f_2(\beta_2(-N)) = \gamma_2(-N);
        \\& b_1 = \phi_1(a_1) = \gamma_1'(N); 
        \\& c_1 = \gamma_1'(-N); 
        \\& c_2 = \gamma_2(N); 
        \\& d_2 = \phi_2(c_2) = \gamma_2'(N).
    \end{split}
\end{equation*}

\medskip \noindent
Consider $\Gamma = A_{(x_1,x_2)}(r) \subset X $, and notice that $f(\Gamma) \subset B_Y((y_1,y_2),R)$.
For all $n\in \mathbb{Z}$, consider $\Gamma_n =   A_{(\beta_1(n),\beta_2(-n))}(r)  \subset B_X\left ((\beta_1(n),\beta_2(-n)), 4r  \right) $. In particular $\Gamma_0 = \Gamma$. 

\medskip \noindent
We will consider the images of $(\Gamma_n)_n$, in the following the path (see Figure \ref{fig path horocyclic prod}): 
\begin{equation*}
            (y_1,y_2) \xrightarrow[f(X)] \, (a_1,a_2) \xrightarrow[\textup{jump 1}] \, (b_1,a_2) \xrightarrow[g(X)] \, (c_1,c_2) \xrightarrow[\textup{jump 2}] \, (c_1,d_2) \xrightarrow[h(X)] \, (y_1',y_2') 
\end{equation*}
where $f = (f_1,f_2)$, $g = (\phi_1 \circ f_1,f_2)$, and $h = (\phi_1 \circ f_1,\phi_2 \circ f_2)$.

\medskip \noindent

\begin{figure}[!ht]
\begin{center}
    \includegraphics[width=0.8\linewidth]{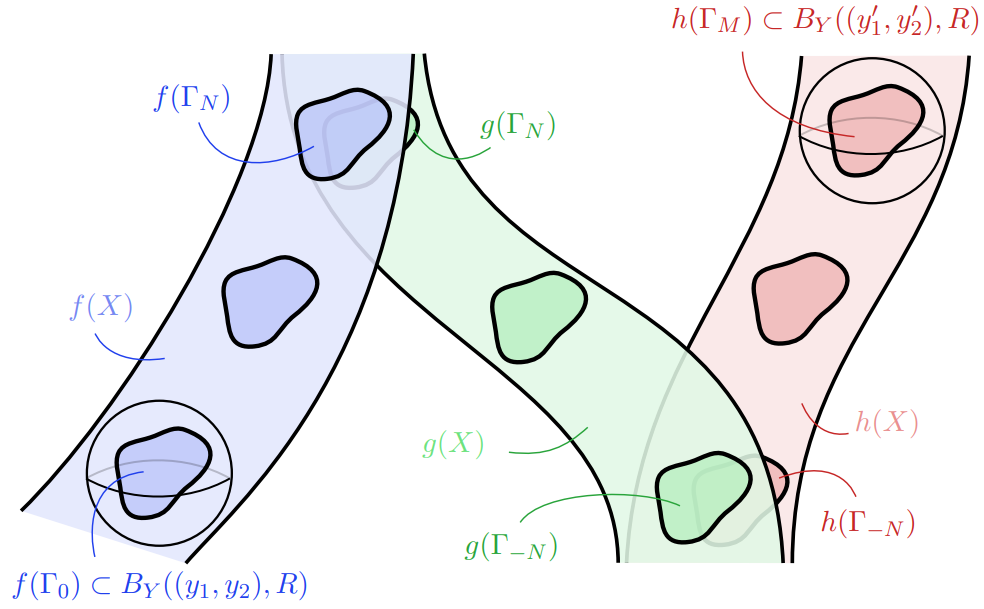}
    \caption{Images of $(\Gamma_n)_n$ by $f$,$g$, and $h$.}
    \label{fig path horocyclic prod}
\end{center}
\end{figure}

\medskip \noindent
Let $C_1$ be the component containing $B_Y((y_1,y_2),r)$ and $C_2$ the component containing $B_Y((y_1',y_2'),R)$. Let $\delta = \frac{7}{8}$ as in Proposition \ref{persistent family of dl}. The proof is as follows:

\medskip \noindent
\textbf{Case 1:} Assume that $|\Gamma_N \cap f^{-1}(C_1) | < \delta |\Gamma| $. Since $\Gamma_0 \subset f^{-1}(C_1) $, then, by Proposition~\ref{main}, there exists $i \in \{1, \dots, n\}$ such that $\Gamma_i \cap f^{-1}(S^{+(D+2)}) $ is a $\delta$-cut of $\Gamma_i$,
$$        | \Gamma_i \cap f^{-1}(S^{+(D+2)})| \geq \cut(\Gamma_i) \geq C \, 2^{r}  . $$
As in the proof of Theorem \ref{separating by embedding}, there exists a constant $\theta>0$ such that, for any bounded subset $A\subset X$, $|f(A)|\geq \theta |A|$. Let $z = (\beta_1(n),\beta_2(-n))$. The inclusion $f(\Gamma_i) \subset B_Y(z,R)$ implies, by an argument similar to the one in the proof of Corollary \ref{cor of main}, that, for $R$ big enough,
$$ C'|B_Y(z,2R) \cap S| \geq | B_Y(z,R) \cap S^{+(D+2)}|  \geq \theta | \Gamma_i \cap f^{-1}(S^{+(D+2)})|     $$
where $C' = \sup_{y \in Y}|B_Y(y,D+2)|$. Therefore
\begin{equation}\label{case 1}
     |B_Y(z,8 \lambda r) \cap S| \geq \frac{C \theta}{C'}2^r.  
\end{equation}
\textbf{Case 2:} Otherwise,  let $(C_j')_{j \in J'}$ be the $\lambda$-coarsely connected components of $Y \backslash S^{+(D+1)}$, and $C_1'$ the component containing $B_Y((y_1,y_2),r)$. 
\begin{claim}
    $|\Gamma_N \cap f^{-1}(C_1) | \geq  \delta |\Gamma| $ \textup{implies that} $|\Gamma_N \cap g^{-1}(C_1') | \geq  \delta |\Gamma| $.
\end{claim}

\noindent
Let us state the following trivial lemma.
\begin{lemma}
    Let $X$ be a $k$-coarsely connected metric space, $S \subset X$, and $D\geq k$. Let $x,x' \in X$ such that $x \in X \backslash S^{+(D+k)}$ and $d(x,x')<k$. Then $x$ and $x'$ are in a same $k$-coarsely connected component of $X \backslash S^{+ D}$.
\end{lemma}
\begin{proof}[Proof of the claim]
Let $z = (z_1,z_2) \in \Gamma_N $. By the inequality (\ref{<1/2}),
\begin{equation*}
    \begin{split}
        d_Y(f(z), g(z))  &= d_{Y_1}(f_1(z_1), g_1(z_1)) +  d_{Y_2}(f_2(z_2), g_2(z_2)) 
        \\&=  d_{Y_1}(f_1(z_1), \phi_1(f_1(z_1))) < \frac{1}{2}. 
    \end{split}
\end{equation*}
Therefore, $f(z) \in C_1 $ implies that $g(z) \in C_1'$. So $\Gamma_N \cap f^{-1}(C_1) \subset \Gamma_N \cap g^{-1}(C_1')$, and $|\Gamma_N \cap f^{-1}(C_1) | \geq  \delta |\Gamma_N| $ implies that $|\Gamma_N \cap g^{-1}(C_1') | \geq  \delta |\Gamma_N| $. $\qedhere$
\end{proof}

\noindent
Now, we follow the images of $\Gamma_{N}, \Gamma_{N-1}, \dots , \Gamma_{-N}$ by $g$. If $|\Gamma_{-N} \cap g^{-1}(C_1') | <  \delta |\Gamma| $, then it is done as in the first case.

\medskip \noindent
\textbf{Case 3:} If not, let $(C_j'')_{j \in J''}$ be the $\lambda$-coarsely connected components of $Y \backslash S^{+ D}$, and $C_1''$ the component containing $B_Y((y_1,y_2),r)$. We can show similarly: $|\Gamma_{-N} \cap g^{-1}(C_1') | \geq  \delta |\Gamma| $ implies that $|\Gamma_{-N} \cap h^{-1}(C_1'') | \geq  \delta |\Gamma| $.

\medskip \noindent
Finally, we follow the images of $\Gamma_{-N},\Gamma_{-N+1}, \dots, \Gamma_{M}$ by $h$, where $M = b_{Y_1}(y_1') = -b_{Y_2}(y_2')$. Note that $h(\Gamma_M) \subset C_2$. Indeed,
$$ h(\Gamma_M) \subset B_Y(h((\beta_1(M),\beta_2(-M))),R). $$
However,
\begin{equation*}
    \begin{split}
        h((\beta_1(M),\beta_2(-M))) &= (\phi_1 \circ f_1(\beta_1(M)),\phi_2 \circ f_2(\beta_2(-M))) \\&=  (\phi_1 (\gamma_1(M)),\phi_2 (\gamma_2(-M))) 
        \\&=  (\gamma_1'(M),\gamma_2'(-M)) = (y_1',y_2').
    \end{split}
\end{equation*}
The last equality follows from the choice $M = b_{Y_1}(y_1') = -b_{Y_2}(y_2')$. So we have 
   \begin{equation*}
   \left\{
       \begin{split}
          &|\Gamma_{-N} \cap h^{-1}(C_1'') | \geq  \delta |\Gamma|
        \\& |\Gamma_{M} \cap h^{-1}(C_1'') | =0
       \end{split}
       \right..
   \end{equation*}
   Therefore, there exists $j \in \{-N,-N+1, \dots, M\}$ such that $\Gamma_{j} \cap f^{-1}\left(S^{+ D}\right) $ is a $\delta$-cut of $\Gamma_j$, and we get an inequality similar to (\ref{case 1}). This concludes the proof. $\qedhere$
\end{proof}

\section{Application to graphs of groups}\label{section graph of groups}

The main result of this section is Theorem~\ref{sep in graphs of groups general}, which implies Proposition~\ref{sep in graphs of groups}, Corollary~\ref{corIntro: chordal}, and Corollary~\ref{corIntro:qi between graphs of groups}.

\medskip \noindent
Let us briefly recall the definition of a graph of groups. A graph of groups $(\Gamma,\mathcal{A})$ consists of the following data:
\begin{itemize}
    \item [(1)] An oriented and connected graph $\Gamma$. For each edge $e$, we denote by $\ori e$ its initial vertex and by $\ter e$ its terminal one.
    \item [(2)] An assignment of a \emph{vertex-group} $A_v$ to every vertex $v$ of $\Gamma$, an assignment of an \emph{edge-group} $A_e$ to every edge $e$, and an injective homomorphism $i_0 : A_e \to A_v$ (resp.\ $i_1 : A_e \to A_v$) whenever $\ori e = v$ (resp.\ $\ter e = v$).
\end{itemize}
A group $\pi_1(\Gamma,\mathcal{A})$ is associated to a graph of groups $(\Gamma,\mathcal{A})$, called its \emph{fundamental group}, see \cite{serre2002trees}. A topological definition of $\pi_1(\Gamma,\mathcal{A})$ was given in \cite{scott1979topological}. Let us recall it briefly. To a graph of groups $(\Gamma,\mathcal{A})$ one associates a \emph{graph of spaces} $X$ such that $\pi_1(X) = \pi_1(\Gamma,\mathcal{A})$ as follows. For each vertex $v$ of $\Gamma$ (resp.\ each edge $e$) take a finite simplicial complex $X_v$ (resp.\ $X_e$) such that $\pi_1(X_v) = A_v$ (resp.\ $\pi_1(X_e) = A_e$). Let $I$ be the unit interval. Then $X$ is defined by gluing the complexes $X_v$ and $X_e \times I$, for any $v$ vertex and $e$ edge, as follows. If $e$ is an edge and $i_0 : A_e \to A_{\ori e}$ is the associated injective homomorphism, let $f_0 : X_e \to X_{\ori e}$ be a simplicial map inducing it; and identify, for each $x \in X_e$,  $(x,0) \in X_e \times \left\{0\right\}$ to $f_0(x) \in X_{\ori e}$. Similarly, we identify $X_e \times \left\{1\right\} $ to $X_{\ter e}$. 

\medskip \noindent
The universal cover $\Tilde{X}$ is a union of copies of $\Tilde{X_v}$ and $\Tilde{X_e}\times I$. By giving each edge length one, $\Tilde{X}$ is quasi-isometric to $\pi_1(\Gamma,\mathcal{A})$. Moreover, the Bass-Serre tree $T$ of $\pi_1(\Gamma,\mathcal{A})$ can be obtained from $\Tilde{X}$, see \cite[Section 4]{scott1979topological}, by identifying each copy of $\Tilde{X_v}$ in $\Tilde{X}$ to a vertex and each copy of $\Tilde{X_e}\times I$ to a copy of $I$. This map $p : \Tilde{X}\to T$ is $\pi_1(\Gamma,\mathcal{A})$-equivariant. If $t$ is a vertex of $T$ (resp.\ a midpoint of an edge of $T$), we call $p^{-1}(t)$ a \emph{vertex-space} (resp.\ an \emph{edge-space}). It is easy to see that removing an edge-space separates $\Tilde{X}$ into two coarsely connected components with arbitrarily large balls (see \cite[Lemma 2.2]{papasoglu2012splittings} for a proof). Therefore, the family of edge-spaces coarsely separates $\Tilde{X}$.

\medskip \noindent
The proof of the following statement is similar to that of \cite[Lemma 3.2]{papasoglu2002quasi}.
\begin{thm}\label{sep in graphs of groups general}
Let $X$ be a graph of spaces associated to a graph of groups $(\Gamma,\mathcal{A})$, $Z$ a coarsely connected metric space, and $f : Z \to \Tilde{X}$ a coarse embedding. Either $f(Z)$ is coarsely separated by the family of edge-spaces of $\Tilde{X}$ or it is contained in a bounded neighbourhood of some vertex-space. 
\end{thm}
\begin{proof}
    Suppose that $f(Z)$ is not coarsely separated by the family of edge-spaces of $\Tilde{X}$. Let $T$ be the Bass-Serre tree of $\pi_1(\Gamma,\mathcal{A})$ obtained from $\Tilde{X}$, and $p : \Tilde{X} \to T$. Since $Z$ is coarsely connected and $p$ is simplicial, we can assume that $p \circ f(Z)$ is connected.
    \begin{claim}
        $p \circ f(Z)$ is bounded.
    \end{claim}
    \begin{proof}[Proof of the claim]
    If not, then, for every $n \in \mathbb{N}$, there exists $x_n,y_n \in Z$ such that $d\left(p \circ f(x_n),p \circ f(y_n)\right )> 3n$. Let $t_n \in T$ be the midpoint of an edge in the path from $p \circ f(x_n)$ to $p \circ f(y_n)$. Note that $p$ is distance non-increasing, so $B_{\Tilde{X}}(f(x_n),n)$ and $B_{\Tilde{X}}(f(y_n),n)$ are in different connected components of $\Tilde{X} \backslash p^{-1}(t_n) $. This implies that the family of edge-spaces $\left( p^{-1}(t_n) \right)_n$ coarsely separates $f(Z)$, which leads to a contradiction.
    \end{proof}
    
\noindent
    Let us denote $p \circ f(Z)$ by $Y$. Since $f(Z)$ cannot be coarsely separated by any edge-space of $\Tilde{X}$, it cannot have unbounded intersection with both components of $\Tilde{X} \backslash p^{-1}(t)$, for any $t$ midpoint of an edge of $T$. If there exists a midpoint $t$ of an edge of $T$ such that the intersection of $f(Z)$ with both components of $\Tilde{X} \backslash p^{-1}(t)$ is bounded, we are done. If not, for any $t_e$ midpoint of an edge $e$ of $Y$ there is exactly one component of $\Tilde{X} \backslash p^{-1}(t_e)$ with unbounded intersection with $f(Z)$, and we orient $e$ towards this component. In other words, $e$ points towards the subset of $Y$ onto which the component of $\Tilde{X} \backslash p^{-1}(t_e)$ having unbounded intersection with $f(Z)$ projects. If $v$ is a vertex of $Y$, there can be only one edge incident on $v$ pointing away from it. Therefore, following this orientation leads to a unique vertex $v$ towards which all the edges of $Y$ are pointing.
        \begin{claim}
        $f(Z)$ is contained in a bounded neighbourhood of $p^{-1}(v)$.
    \end{claim}
    \begin{proof}[Proof of the claim]
    If not, for every $n$ there exists $z_n \in Z$ such that $d\left(f(z_n),p^{-1}(v)\right)>3n $. Let us denote by $e_n$ the edge incident to $v$ that the path $[p \circ f(z_n),v]$ crosses, and $t_n$ its midpoint. Let $E_v$ be the set of edges incident on $v$. Since they are all pointing towards $v$, for any $e \in E_v$ there exists $N_e \in \mathbb{N}$ such that $n \geq N_e \implies e_n \ne e $. Therefore, for any $n$, there exists $m>n$ such that $e_m \ne e_n$, i.e.\ $p^{-1}(t_n)$ separates the balls $B_{\Tilde{X}}(f(z_n),n)$ and $B_{\Tilde{X}}(f(z_m),n)$ (because $n<m$). This leads to a contradiction as in the previous claim.   
    \end{proof}

\noindent
    This ends the proof of the theorem.
\end{proof}

\noindent
Note also that each vertex-space (resp.\ edge-space) of $\Tilde{X}$ corresponds to a left coset of some vertex-group (resp.\ edge-group) of the corresponding graph of groups. Therefore, Proposition~\ref{sep in graphs of groups} follows immediately.
\begin{remark}
    Note that $\Gamma$ is not necessarily finite in the previous theorem, but it should be in its applications. Indeed, in order for the condition of ``non separation by the family of edge-spaces'' to be satisfied, there should be a uniform upper bound on the growth of the edge-spaces, which is necessarily satisfied when $\Gamma$ is finite.
\end{remark}
\begin{proof}[Proof of Corollary \ref{corIntro:qi between graphs of groups}]
Let $f : G_1 \to G_2$ be a quasi-isometry, with quasi-inverse $g$, and let $H$ be a vertex-group of $G_1$. All the vertex-groups of $G_1$ and $G_2$ being finitely generated, $G_1$ and $G_2$ are finitely generated. Recall that the subgroup inclusion between finitely generated groups is a coarse embedding, therefore $f$ induces a coarse embedding $H \to G_2$. By the previous theorem, $f(H)$ is in a bounded neighbourhood of a left coset $K$ of some vertex-group of $G_2$. Similarly, $g(K)$ is in a bounded neighbourhood of a left coset $H'$ of some vertex-group of $G_1$. The rest follows as in the proof of \cite[Theorem 3.1]{papasoglu2002quasi}.
\end{proof}
\noindent
Finally, we prove a strong form of Corollary~\ref{corIntro: chordal}, replacing right-angled Artin groups with more general graph products.

\medskip \noindent
Recall that, given a (simplicial) graph $\Gamma$ and a collection of groups $\mathcal{G}= \{ G_u, u \in V(\Gamma)\}$ indexed by the vertices of $\Gamma$, the \emph{graph product} $\Gamma \mathcal{G}$ is the quotient
$$\left( \underset{u \in V(\Gamma)}{\ast} G_u \right) / \langle \langle [g,h]=1, \ g \in G_u, h \in G_v, \{u,v\} \in E(\Gamma) \rangle \rangle$$
where $E(\Gamma)$ denotes the edge-set of $\Gamma$. Roughly speaking, we take the free product of the groups from $\mathcal{G}$, referred to as \emph{vertex-groups}, and force any two adjacent vertex-groups to commute. One can think as graph products as providing an interpolation between free products (when $\Gamma$ has no edges) and direct sums (when $\Gamma$ is complete). Right-angled Artin (resp.\ Coxeter) groups coincide with graph products of infinite cyclic (resp.\ cyclic of order two) groups. Given a subgraph $\Lambda \subset \Gamma$, we denote by $\langle \Lambda \rangle$ the subgroup of $\Gamma \mathcal{G}$ generated by the vertex-groups given by the vertices of $\Lambda$. If $\Lambda$ is a \emph{clique} (i.e.\ a complete subgraph), we refer to $\langle \Gamma \rangle$ as a \emph{clique-subgroup}. 

\begin{thm}\label{thm chordal}
Let $\Gamma$ be a finite chordal graph and $\mathcal{G}$ a collection of finitely generated groups indexed by $V(\Gamma)$. Let $Z$ be a geodesic metric space and $\rho : Z \to \Gamma \mathcal{G}$ a coarse embedding. Either $\rho(Z)$ is contained in a neighbourhood of a coset of a clique-subgroup or it is coarsely separated by the cosets of clique-subgroups.
\end{thm}

\begin{proof}
Assume that $\rho(Z)$ is not coarsely separated by the cosets of clique-subgroups. We argue by induction over the number of vertices of $\Gamma$. If $\Gamma$ is a clique, there is nothing to prove. Otherwise, it follows from \cite{MR0130190} that $\Gamma$ contains a separating clique, say $\Lambda$. Decompose $\Gamma$ as an amalgamate of two proper subgraphs $\Gamma_1$ and $\Gamma_2$ along $\Lambda$. Then $\Gamma \mathcal{G}$ decomposes as the amalgamated product $\langle \Gamma_1 \rangle \ast_{\langle \Lambda \rangle} \langle \Gamma_2 \rangle$. Moreover, $\langle \Lambda \rangle$ is a retract (which can be seen through the morphism $\Gamma \mathcal{G} \to \langle \Lambda \rangle$ that kills the factors indexed by the vertices not in $\Lambda$), so it is an undistorted subgroup. Because $\rho(Z)$ is not coarsely separated by the cosets of $\langle \Lambda \rangle$, it follows from Theorem~\ref{sep in graphs of groups general} that it must be contained in the neighbourhood of a coset of either $\langle \Gamma_1 \rangle$ or $\langle \Gamma_2 \rangle$. Because $\langle \Gamma_i \rangle$ coincides with the graph product $\Gamma_i \{ G_u, u \in V(\Gamma_i)\}$ for $i=1,2$ and that $\Gamma_1,\Gamma_2$ are again chordal, we conclude that $\rho(Z)$ must be contained in a neighbourhood of a coset of a clique-subgroup
\end{proof}

\section{Distortion of coarse embeddings between coarse Poincar\'e duality groups and spaces}\label{sec:distortion}
A coarse version of the Jordan separation theorem was first proved in \cite{Schwartz1996, Farb1996}, and a stronger version in the setting of coarse Poincar\'e duality groups and spaces was given later in \cite{kapovich2005coarse}. Before stating it, let us recall that a metric space $X$ is \emph{uniformly contractible} if there is a function $\phi : \mathbb{R}_{+} \to \mathbb{R}_{+}$ such that, for any $x \in X$, $r>0$, a continuous map from a finite simplicial complex with image in $B_X(x,r)$ is contractible in $B_X(\phi(x),r)$. A version of this theorem that is easy to state is the following:
\begin{thm}{\cite[Theorem 9.73]{dructu2018geometric}}\label{jordan}
Let $X,Y$ be uniformly contractible simplicial complexes of bounded geometry, homeomorphic to $\mathbb{R}^{n-1}$ and $\mathbb{R}^{n}$ respectively. Then if $f : X \to Y$ is a simplicial coarse embedding, its image $f(X)$ coarsely separates $Y$.
\end{thm}

\noindent
Examples of spaces to which this theorem applies include:
\begin{itemize}
    \item symmetric spaces of non-positive curvature;
    \item connected, simply-connected nilpotent Lie groups;
    \item horocyclic products $X \bowtie Y$, where $X$ and $Y$ are rank one symmetric spaces of non-compact type.
\end{itemize}
Indeed, if $M$ is such a manifold, then it has bounded geometry, and is therefore quasi-isometric to a finite dimensional simplicial complex $X$ homeomorphic to it (\cite[Theorem 8.53]{dructu2018geometric}). Since $M$ is homogeneous, $X$ is uniformly contractible. Moreover, if $f : M_1 \to M_2$ is a coarse embedding between such manifolds, it induces a coarse embedding $\Tilde{f} : X_1 \to X_2$ between these simplicial complexes. The complexes $X_1$ and $X_2$ are finite-dimensional, uniformly contractible, and have bounded geometry, so by a ``connect the dots'' argument we can assume that $\Tilde{f}$ is continuous. Finally, up to a barycentric subdivision of $X_1$ (a finite number of iterations) and a simplicial approximation, we can assume that $\Tilde{f}$ is simplicial. Therefore, if $X_1$ and $X_2$ are homeomorphic to $\mathbb{R}^{n-1}$ and $\mathbb{R}^{n}$ respectively, then $\Tilde{f}(X_1)$ coarsely separates $X_2$, and, by Lemma \ref{lem:CoarseSep}, $f(M_1)$ coarsely separates $M_2$.

\medskip \noindent
By combining this theorem with Theorem \ref{thmIntro:nilpot}, we now prove Theorem \ref{thmIntro:lower bound distortionNilp}.
\begin{proof}[Proof of Theorem \ref{thmIntro:lower bound distortionNilp}]
    The image $f(H)$ coarsely separates $G$, so, by Theorem \ref{thmIntro:nilpot}, $f(H)$ has growth $\gtrsim r^{n-1}$: there exists $\alpha >0$ such that, for any $r$ big enough, there exists $h_r \in H$ such that 
    $$ | B_G(f(h_r),r) \cap f(H) | \geq \alpha r^{n-1}. $$ 
    We recall that coarse embeddings coarsely preserve the volume, so there exists a constant $\theta>0$ such that, for all $r>0$, $C_r = \{ h \in H \mid d(f(h_r),f(h)) \leq r       \} $ satisfies
    $$      | C_r |   \geq    \theta  r^{n-1} . $$
    Since $\beta_H(r) \simeq r^p$, there exists a constant $\delta>0$ such that $C_r \not \subset B_H\left(h_r, \delta r^{\frac{n-1}{p}}\right) $, i.e.\ there exists $h_r' \in C_r$ such that $d_H(h_r,h_r')> \delta r^{\frac{n-1}{p}} $. Hence
    $$           \rho_- ( \delta  r^{\frac{n-1}{p}} )    \leq     \rho_- (  d_H(h_r,h_r')  )  \leq   d_G(f(h_r),f(h_r')) \leq r  .  $$
    We conclude that $\rho_-(r) \lesssim r^{\frac{p}{n-1}}$.
\end{proof}

\noindent
By using Theorems \ref{thmIntro:Symmetric/buildings} and \ref{thmIntro:horocyclic products} instead, the proof of Theorem \ref{thmIntro:lower bound distortionExp} is similar.

\section{Large-scale geometry of wreath products}\label{section:Wreath}

\noindent
In this section, we focus on the geometry of wreath products. In \cite{GTbig}, we proved that, given a finite group $F$ and a finitely generated group $H$, every coarse embedding $Z \to F \wr H$ must have its image contained in a neighbourhood of an $H$-coset whenever $Z$ is coarsely simply connected and uniformly one-ended. In Section~\ref{section:EmbeddingTheorem}, we extend this Embedding Theorem by allowing $F$ to be infinite. In this case, the uniform one-endedness of $Z$ is replaced with a non-separation assumption depending on $F$. See Theorem~\ref{thm:FullEmbeddingThm}. Following \cite{GTbig, GTlike}, such an Embedding Theorem has strong consequences on the structure of quasi-isometries and coarse embeddings between wreath products. This is explained in Section~\ref{section:Aptolic}. Combined with the non-separation results proved in Section~\ref{section:CoarseSep}, we deduce various concrete applications, recorded in Sections~\ref{section:AppCoarseEmb} and~\ref{section:AppQI}.

\subsection{The embedding theorem for wreath products}\label{section:EmbeddingTheorem}

\noindent
This section is dedicated to the proof of the following Embedding Theorem:

\begin{thm}\label{thm:FullEmbeddingThm}
Let $A,B$ be two finitely generated groups. For all finite subset $E \subset B$ and colouring $c \in A^{(B)}$, set
$$\mathcal{B}(E,c):= \{ (f,x) \mid x \in E, f \in c A^{(E)} \}.$$
Also, set 
$$\mathcal{B}_k:= \{ \mathcal{B}(E,c) \mid E \subset B \text{ of size } \leq k, \ c \in A^{(B)} \}$$ 
for every $k \geq 1$, and $\mathcal{B}:= \bigcup_{k \geq 1} \mathcal{B}_k$. For every coarsely simply connected graph $Z$ and for every coarse embedding $\rho : Z \to A \wr B$, if $\rho(Z)$ is neither contained in a subset of $\mathcal{B}$ nor coarsely separated by $\mathcal{B}_k$ for some $k \geq 1$, then the image of $\rho$ must lie in the neighbourhood of some $B$-coset. Moreover, the size of this neighbourhood only depends on $A$, $B$, $Z$, and the parameters of $\rho$.
\end{thm}

\noindent
Notice that, if $A$ is finite, then the conclusion of the theorem holds whenever $Z$ is unbounded and uniformly one-ended. Therefore, Theorem~\ref{thm:FullEmbeddingThm} generalises the (algebraic version of the) Embedding Theorem proved in \cite{GTbig}.

\subsubsection{Preliminaries on quasi-median geometry}

\noindent
Similarly to \cite{GTbig}, our proof of the Embedding Theorem is fundamentally based on quasi-median geometry. In this section, we record the basic definitions and properties about quasi-median graphs that will be necessary later.

\paragraph{Quasi-median graphs.} There exist several equivalent definitions of quasi-median graphs. See for instance \cite{quasimedian}. Below, we give the definition used in \cite{Qm}.

\begin{definition}
A connected graph $X$ is \emph{quasi-median} if it does not contain $K_4^-$ and $K_{3,2}$ as induced subgraphs, and if it satisfies the following two conditions:
\begin{description}
	\item[(triangle condition)] for every triplet of vertices $a, x,y \in X$, if $x$ and $y$ are adjacent and if $d(a,x)=d(a,y)$, then there exists a vertex $z \in X,$ which is adjacent to both $x$ and $y$, satisfying $d(a,z)=d(a,x)-1$;
	\item[(quadrangle condition)] for every quadruplet of vertices $a,x,y,z \in X$, if $z$ is adjacent to both $x$ and $y$ and if $d(a,x)=d(a,y)=d(a,z)-1$, then there exists a vertex $w \in X$ that is adjacent to both $x,y$ and that satisfies $d(a,w)=d(a,z)-2$.
\end{description}
\end{definition}

\noindent
The graph $K_{3,2}$ is the bipartite complete graph, corresponding to two squares glued along two adjacent edges; and $K_4^-$ is the complete graph on four vertices minus an edge, corresponding to two triangles glued along an edge.

\medskip \noindent
However, the definition of quasi-median graphs is only given for completeness. Similarly to median graphs (i.e. one-skeletons of CAT(0) cube complexes), the geometry of quasi-median graphs essentially reduces to the combinatorics of separating subspaces referred to as \emph{hyperplanes}. Understanding this interaction between hyperplanes and geometry will be sufficient for us.

\begin{definition}
Let $X$ be a quasi-median graph. A \emph{hyperplane} $J$ is an equivalence class of edges with respect to the transitive closure of the relation that identifies two edges in the same $3$-cycle and two opposite edges in the same $4$-cycle. The \emph{carrier} of $J$, denoted by $N(J)$, is the subgraph generated by the edges in $J$. The connected components of the graph $X\backslash \backslash J$ obtained from $X$ by removing the interiors of the edges in $J$ are the \emph{sectors delimited by $J$}; and the connected components of $N(J) \backslash \backslash J$ are the \emph{fibres} of $J$. Two distinct hyperplanes $J_1$ and $J_2$ are \emph{transverse} if $J_2$ contains an edge in $N(J_1) \backslash J_1$. 
\end{definition}
\begin{figure}
\begin{center}
\includegraphics[clip, trim= 0 16cm 10cm 0, scale=0.4]{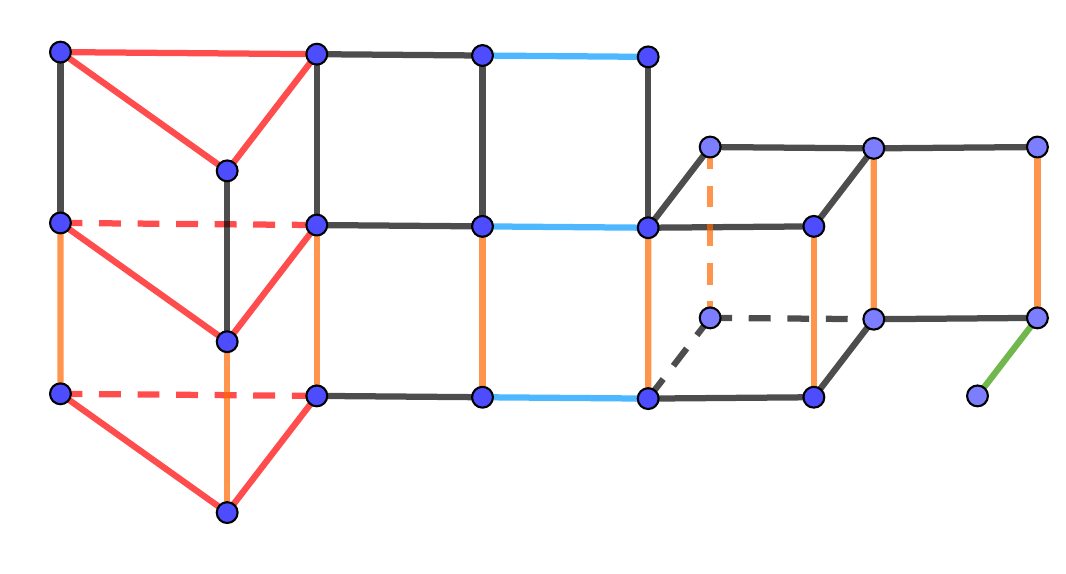}
\caption{Four hyperplanes in a quasi-median graphs. The orange hyperplane is transverse to the red and blue hyperplanes.}
\label{hyperplanes}
\end{center}
\end{figure}

\noindent
See Figure \ref{hyperplanes} for a few examples. The role of hyperplanes is highlighted by the following statement. For a self-contained proof, we refer to \cite[Propositions 2.15 and 2.30]{Qm}; see also \cite[Theorem 3.19]{MR4033512} (or \cite[Theorem 2.14]{AutQM} for proofs specific to the quasi-median graphs that will be considered in the next sections). 

\begin{thm}\label{thm:MainQM}
Let $X$ be a quasi-median graph. The following assertions hold:
\begin{itemize}
	\item[(i)] For every hyperplane $J$, $X\backslash \backslash J$ contains at least two connected components.
	\item[(ii)] Carriers and fibres of hyperplanes are gated.
	\item[(iii)] A path is a geodesic if and only if it crosses each hyperplane at most once.
	\item[(iv)] The distance between two vertices coincides with the number of hyperplanes separating them.
\end{itemize}
\end{thm}

\noindent
Recall that a subgraph $Y \subset X$ is \emph{gated} if, for every vertex $x \in X$, there exists a vertex $y \in Y$ such that, for every $z \in Y$, there exists a geodesic between $x$ and $z$ passing through $y$. We refer to the vertex $y$ as the \emph{projection of $x$ onto $Y$}. The property of being gated can be thought of as a strong convexity condition. In particular, being gated implies being convex. Also, observe that the projection $y$ coincides with the unique vertex of $Y$ minimising the distance to $x$. A useful statement, showing how hyperplanes interact with projections, is the following. We refer to \cite[Lemma 2.34]{Qm} or \cite[Corollary~3.20]{MR4033512} for a proof. 

\medskip \noindent
Quasi-median graphs are quite similar to median graphs. Essentially, the only differences are that edges are replaced with \emph{cliques} (i.e.\ maximal complete subgraphs), that cubes are replaced with \emph{prisms} (i.e.\ subgraphs which are products of cliques), and that cutting along a hyperplane may disconnected the graph into more than two connected components (possibly infinitely many). Following this analogy, we refer to the \emph{cubical dimension} of a quasi-median graph as the maximal number of factors in a prism, or, equivalently as a consequence of \cite[Proposition~2.73]{Qm}, as the maximal number of pairwise transverse hyperplanes.

\begin{lemma}[\cite{GTbig}]\label{lem:CubDist}
Let $X$ be a quasi-median graph. The inequality
$$d(x,y) \leq \mathrm{dim}_\square(X) \cdot \left( \begin{array}{c} \text{maximal number of pairwise non-} \\ \text{transverse hyperplanes separating $x,y$} \end{array} \right)$$
holds for all vertices $x,y \in X$.
\end{lemma}

\noindent
It is worth noticing that a quasi-median graph $X$ is naturally endowed with wallspace structure. Indeed, given a hyperplane $J$ and one of its fibers $F$, we refer to the sector delimited by $F$ as the sector delimited by $J$ containing $F$ and to the cosector delimited by $F$ as the complement of this sector (or, equivalently, to the union of all the sectors delimited by $J$ that are disjoint from $F$). The pairs (sector, cosector) given by all the fibres of all the hyperplanes of $X$ endow $X$ with the structure of a wallspace. For convenience, we refer to the sector or cosector delimited by a fibre as a \emph{piece}. 

\begin{lemma}\label{lem:InterSectors}
Let $X$ be a quasi-median graph of finite cubical dimension. For every fiber $F$, let $F^+$ denote the choice of the sector or cosector delimited by $F$. Assume that
\begin{itemize}
	\item for any two fibres $F_1$ and $F_2$, $F_1^+ \cap F_2^+ \neq \emptyset$;
	\item every non-increasing sequence $F_1^+ \supset F_2^+ \supset \cdots$ is eventually constant;
	\item for every hyperplane $J$, there is at least one fibre $F$ of $J$ for which $F^+$ is a sector.
\end{itemize}
Then $\bigcap\limits_{\text{$F$ fibre}} F^+$ is non-empty and reduced to a single vertex.
\end{lemma}

\begin{proof}
Given a hyperplane $J$, we know from the third item that there is at least one fibre $F$ of $J$ such that $F^+$ is a sector. Actually, it must be the only one, since otherwise one would contradict the first item (as two distinct sectors are always disjoint). Set $J^+:=F^+$. Notice that, for every fibre $L$ of $J$ distinct from $F$, $L^+ \supset J^+$ since $L^+$ must be the cosector delimited by $L$. We conclude that
$$\bigcap\limits_{F \text{ fibre}} F^+ = \bigcap\limits_{J \text{ hyperplane}} J^+ \neq \emptyset,$$
where the last assertion follows from \cite[Lemma~5.9]{GTbig}.
\end{proof}

\paragraph{Coherent systems of metrics.} Let $X$ be a quasi-median graph. A \emph{system of metrics} is the data, for every clique $C$ of $X$, of a metric $\delta_C$ on $C$. It is \emph{coherent} if, for all cliques $A,B$ contained in the same hyperplane, the projection of $A$ onto $B$ induces an isometry $(A,\delta_A) \to (B,\delta_B)$. Given a coherent system of metrics $\{(C,\delta_C) \mid C \subset X \text{ clique}\}$, one can naturally extend the $\delta_C$ to a new metric $\delta$ defined on $X$, namely, we set
$$\delta_J : (x,y) \mapsto \delta_C(\mathrm{proj}_C(x),\mathrm{proj}_C(y))$$
for all hyperplane $J$ and clique $C \subset J$, and we define
$$\delta : (x,y) \mapsto \sum\limits_{J \text{ hyperplane}} \delta_J(x,y).$$
The fact that our system is coherent implies that the $\delta_J$ are well-defined, i.e.\ the definition does not depend on the choice of the clique $C$ in $J$. This metric was introduced in \cite{Qm}. For instance, if $\delta_C(x,y)=1$ for all clique $C \subset X$ and distinct vertices $x,y \in C$, the metric $\delta$ coincides with the graph metric of $X$. 

\begin{definition}
Let $X$ be a quasi-median graph endowed with a system of coherent graph metrics $\{(C,\delta_C) \mid C \text{ clique}\}$. A \emph{broken geodesic} between two vertices $x,y \in X$ is a concatenation $\gamma_1 \cup \cdots \cup \gamma_n$ such that:
\begin{itemize}
	\item for every $1 \leq i \leq n$, $\gamma_i$ is a geodesic in some $(C,\delta_C)$ connecting two vertices $a_i,a_{i+1} \in C$;
	\item the path $a_1, \ldots, a_{n+1}$ defines a geodesic in $X$.
\end{itemize}
\end{definition}

\noindent
Roughly speaking, broken geodesics are obtained from geodesics in the quasi-median graph by connecting any two consecutive vertices by a geodesic in the clique they live in with respect to the local metric. According to \cite[Lemma~3.18]{Qm}, broken geodesics turn out to be geodesics with respect to the extended metric. 

\begin{lemma}\label{lem:BrokenGeod}
Let $X$ be a quasi-median graph endowed with a system of coherent graph metrics $\{(C,\delta_C) \mid C \text{ clique}\}$. Broken geodesics are geodesics in $(X,\delta)$.
\end{lemma}

\paragraph{Graph of pointed cliques.} The geometric spaces that will be relevant for us later are not quasi-median graphs per se, but can be described as graphs pointed cliques in quasi-median graphs (endowed with coherent system of metrics). More precisely:

\begin{definition}
Let $X$ be a quasi-median graph endowed with coherent system of graph metrics $\mathcal{M}= \{(C,\delta_C) \mid C \subset X \text{ clique} \}$. The \emph{graph of pointed cliques} $\mathrm{PC}(X,\mathcal{M})$ is the graph whose vertices are the pointed cliques $(C,x)$ of $X$ and whose edges connect two pointed cliques $(C_1,x_1),(C_2,x_2)$ whenever
\begin{description}
	\item[(slide)] $C_1=C_2$ and $\delta_{C_1}(x_1, x_2)=1$;
	\item[(rotation)] or $x_1=x_2$ and $C_1,C_2$ span a prism.
\end{description}
\end{definition}
\begin{figure}
\begin{center}
\includegraphics[width=0.6\linewidth]{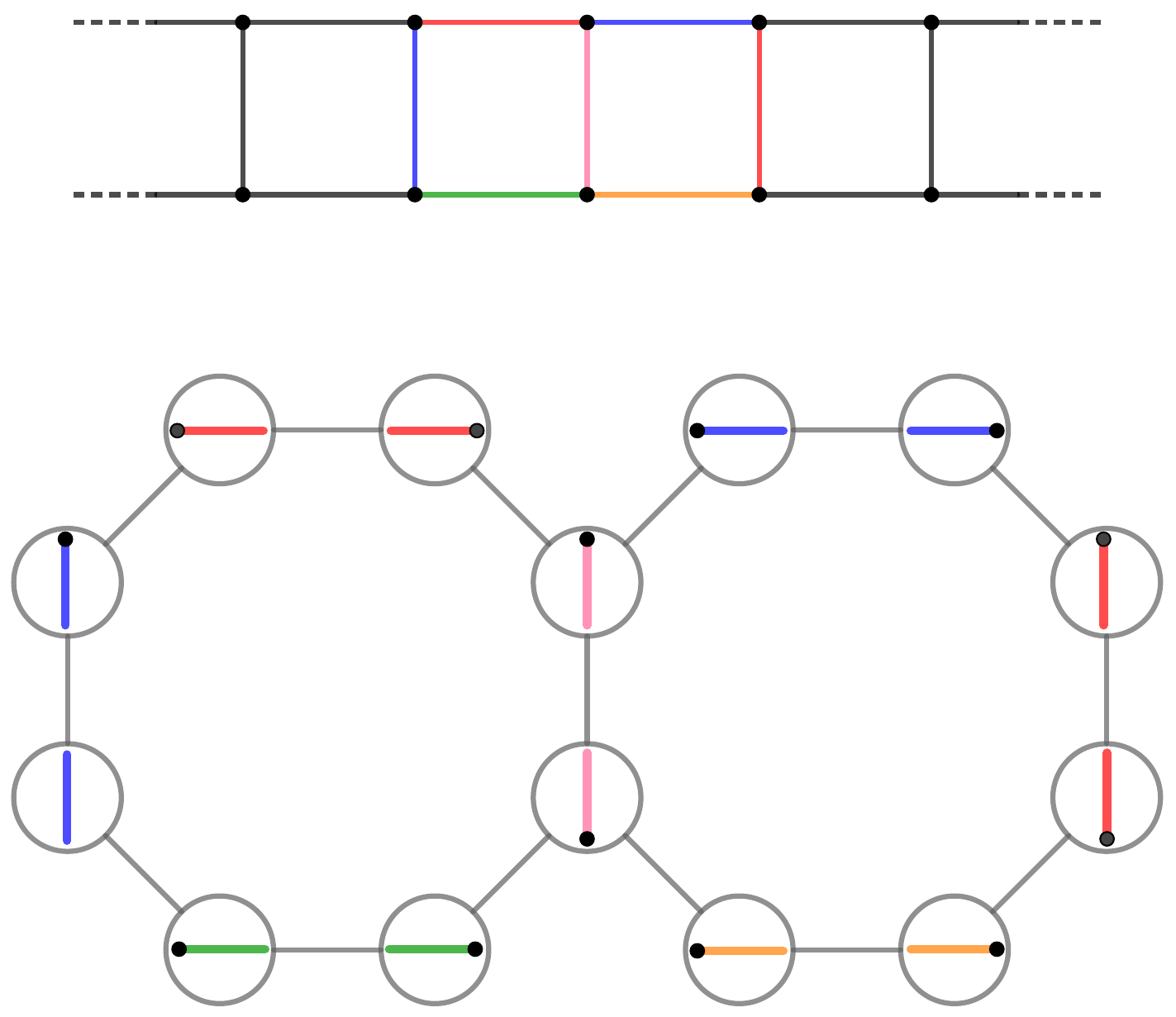}
\caption{A median graph and a piece of its graph of pointed edges.}
\label{PS}
\end{center}
\end{figure}

\noindent
See Figure~\ref{PS} for an example. Let us record a couple of observations about graphs of pointed cliques of quasi-median graphs. 

\begin{lemma}\label{lem:PCdistLow}
Let $X$ be a quasi-median graph endowed with coherent system of graph metrics $\mathcal{M}= \{(C,\delta_C) \mid C \subset X \text{ clique} \}$. The inequality
$$\delta(a,b) \leq d_{\mathrm{PC}}((A,a),(B,b))$$
holds for all pointed cliques $(A,a),(B,b)$.
\end{lemma}

\begin{proof}
There is a natural map $\pi : \mathrm{PC}(X,\mathcal{M}) \to X$ that sends every pointed clique to its marked vertex. Clearly, $\pi$ sends an edge to either an edge or a vertex. Consequently, $\pi$ is $1$-Lipschitz when $X$ is endowed with $\delta$. This fact is equivalent to the desired inequality. 
\end{proof}

\begin{lemma}\label{lem:PointedHyp}
Let $X$ be a quasi-median graph endowed with coherent system of graph metrics $\mathcal{M}= \{(C,\delta_C) \mid C \subset X \text{ clique} \}$. For all pointed cliques $(A,a)$ and $(B,b)$, if $A$ and $B$ belong to the same hyperplane, then
$$d_{\mathrm{PC}}((A,a),(B,b)) \leq 3 \delta(a,b)+1.$$
\end{lemma}

\begin{proof}
Let $a'$ denote the unique vertex of $A$ that belongs to the same fibre as $b$ of the hyperplane containing $A$ and $B$. Fix a geodesic
$$x_0=a', \ x_1, \ldots, \ x_{n-1}, \ x_n=b$$
in $X$. According to Theorem~\ref{thm:MainQM}, this geodesic lies in a fibre of $J$, so, for every $0 \leq i \leq n$, $x_i$ belongs to a (unique) clique $C_i \subset J$. Notice that $C_0=A$ and $C_n=B$. For every $0 \leq i \leq n-1$, let $K_i$ denote the unique clique containing the edge $[x_i,x_{i+1}]$. Notice that $C_i$ spans a prism with both $K_{i-1}$ and $K_i$, so $(C_i,a_i)$ is adjacent to both $(K_{i-1},x_i)$ and $(K_{i+1},x_i)$ in $\mathrm{PC}(X,\mathcal{M})$. By applying triangle inequalities along the vertices
$$(A,a), (A,a')=(C_0,x_0), (K_0,x_0), (K_0,x_1), (C_1,x_1), \ldots,$$ $$(C_{n-1},x_{n-1}), (K_{n-1},x_{n-1}), (K_{n-1},x_n), (C_n,x_n)= (B,b)$$
of $\mathrm{PC}(X,\mathcal{M})$, we deduce that
$$d_\mathrm{PC}((A,a),(B,b)) \leq 1+2 d_X(a,b) + \delta_A(a,a')+ \sum\limits_{i=0}^{n-1} \delta_{K_i}(x_i,x_{i+1})$$
Because $a,a'=x_0, \ldots, x_n=b$ defines a geodesic in $X$, it follows from Lemma~\ref{lem:BrokenGeod} that
$$\delta_A(a,a')+ \sum\limits_{i=0}^{n-1} \delta_{K_i}(x_i,x_{i+1})= \delta(a,b).$$
Since $d \leq \delta$, this yields the desired inequality.
\end{proof}

\noindent
Because quasi-median graphs are naturally cut by hyperplanes, our graphs of pointed cliques are then naturally endowed with a wallspace structure. More precisely:

\begin{definition}
Let $X$ be a quasi-median graph endowed with coherent system of graph metrics $\mathcal{M}= \{(C,\delta_C) \mid C \subset X \text{ clique} \}$. The \emph{bulkhead} associated to a fibre $F$ of a hyperplane $J$ is the subgraph
$$B_F:= \{ (C,p) \mid C \subset J, p \in C \cap F \}$$
of $\mathrm{PC}(X,\mathcal{M})$. The \emph{zones} delimited by $B_F$ are the subgraphs
$$\{ (C,p) \mid C \subset \text{sector delimited by } F, p \in C\}$$
and 
$$\{ (C,p) \mid C \subset \text{cosector delimited by } F, p \in C \} \cup \{ (C,p) \mid C \subset J, p \in C \backslash F\}$$
of $\mathrm{PC}(X, \mathcal{M})$. 
\end{definition}

\noindent
The idea to keep in mind is that, in a graph of pointed cliques, a bulkhead separates the zones it delimits. As justified by the next lemma, this is essentially due to the fact that hyperplanes separate in quasi-median graphs.

\begin{lemma}\label{lem:BulkSep}
Let $X$ be a quasi-median graph endowed with coherent system of graph metrics $\mathcal{M}= \{(C,\delta_C) \mid C \subset X \text{ clique} \}$. Fix a fibre $F$. In the graph of pointed cliques $\mathrm{PC}(X,\mathcal{M})$, the bulkhead $B_F$ separates the zones it delimits.
\end{lemma}

\begin{proof}
Let $(A,a)$ be a pointed clique contained in the sector $S$ delimited by $F$ and $(B,b)$ a pointed clique in the other piece delimited by $F$. Fix a path in $\mathrm{PC}(X,\mathcal{M})$ connecting $(A,a)$ to $(B,b)$. Along this path, let $(C,c)$ be the first point with $C$ not contained in $S$. Let $(D,d)$ be the previous vertex. Necessarily, $D$ is contained in $S$. In particular, $D$ is distinct from $C$, which implies that $(C,c)$ is obtained from $(D,d)$ by rotating the clique. This implies that $c=d$ belongs to $S$. Thus, $C$ is a clique that is not contained in $S$ but intersects it. The only possibility is that $C$ is contained in $J$ and that $c$ belongs to $F$. Thus, $(C,c)$ belongs to the bulkhead $B_F$. 
\end{proof}

\begin{remark}
It is worth noticing that, in general, the zones delimited by a bulkhead may not be connected.
\end{remark}

\subsubsection{Preliminaries on graph products}

\noindent
Given a simplicial graph $\Gamma$ and a collection $\mathcal{G}=\{G_u \mid u \in V(\Gamma)\}$ of groups indexed by the vertex-set $V(\Gamma)$ of $\Gamma$, the \emph{graph product} $\Gamma \mathcal{G}$ is the quotient
$$\left( \underset{u \in V(\Gamma)}{\ast} G_u \right) / \langle \langle [g,h]=1 \text{ if $g \in G_u$ and $h \in G_v$ for some $\{u,v\} \in E(\Gamma)$} \rangle \rangle$$
where $E(\Gamma)$ denotes the edge-set of $\Gamma$. The groups in $\mathcal{G}$ are referred to as \emph{vertex-groups}. If all the vertex-groups are isomorphic to a single group $G$, we denote the graph product by $\Gamma G$ instead of $\Gamma \mathcal{G}$. 

\medskip \noindent
Vertex-groups embed into the graph product. More generally, given an induced subgraph $\Lambda \subset \Gamma$, the subgroup generated by the vertex-groups indexed by the vertices in $\Lambda$, which we denote by $\langle \Lambda \rangle$, is naturally isomorphic to the graph product $\Lambda \mathcal{G}_{|V(\Lambda)}$ where $\mathcal{G}_{|V(\Lambda)}:= \{ G_u \mid u \in V(\Lambda) \} \subset \mathcal{G}$. 

\medskip \noindent
As observed in \cite{Qm}, graph products naturally act on quasi-median graphs. More precisely, they act by left-multiplication on the following Cayley graphs:

\begin{thm}\label{thm:GPandQM}
Let $\Gamma$ be a simplicial graph and $\mathcal{G}= \{G_u \mid u \in V(\Gamma)\}$ a collection of groups. The Cayley graph
$$\mathrm{QM}(\Gamma, \mathcal{G}) := \mathrm{Cayl} \left( \Gamma \mathcal{G}, \bigcup\limits_{u \in V(\Gamma)} G_u \backslash \{1\} \right)$$
is a quasi-median graph of cubical dimension $\mathrm{clique}(\Gamma)$. 
\end{thm}

\noindent
We refer to \cite[Proposition 8.2]{Qm} for a proof. The cliques, prisms, and hyperplanes of $\mathrm{QM}(\Gamma, \mathcal{G})$ are described as follows (see \cite[Lemma~8.6 and Corollaries~8.7 and~8.10]{Qm} or \cite[Lemmas~2.4 and~2.6, and Theorem~2.10]{AutQM}):
\begin{figure}
\begin{center}
\includegraphics[scale=0.4]{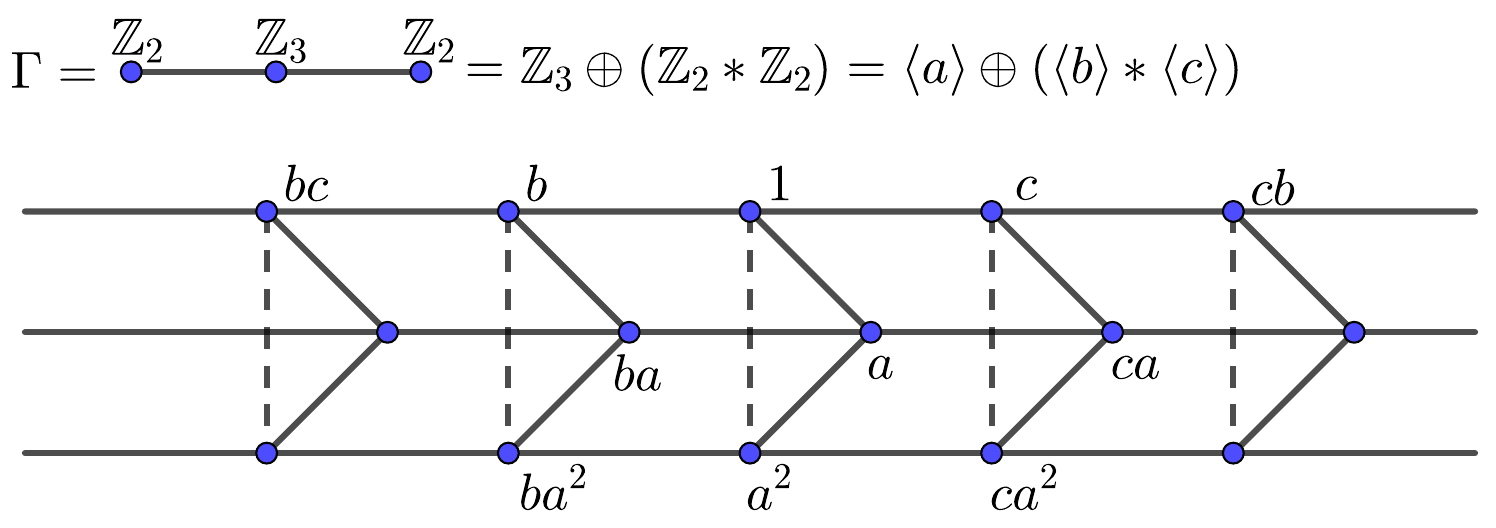}
\caption{Example of a quasi-median graph $\mathrm{QM}(\Gamma, \mathcal{G})$.}
\label{Cayl}
\end{center}
\end{figure}

\begin{lemma}\label{lem:QMcliques}
Let $\Gamma$ be a simplicial graph and $\mathcal{G}$ a collection of groups indexed by $V(\Gamma)$. The cliques of $\mathrm{QM}(\Gamma, \mathcal{G})$ coincide with the cosets of vertex-groups.
\end{lemma}

\begin{lemma}\label{lem:QMprisms}
Let $\Gamma$ be a simplicial graph and $\mathcal{G}$ a collection of groups indexed by $V(\Gamma)$. The prisms of $\mathrm{QM}(\Gamma, \mathcal{G})$ coincide with the cosets of the $\langle \Lambda \rangle$, where $\Lambda \subset \Gamma$ is a complete subgraph.
\end{lemma}

\begin{lemma}\label{lem:QMhyp}
Let $\Gamma$ be a simplicial graph and $\mathcal{G}$ a collection of groups indexed by $V(\Gamma)$. Fix a vertex $u \in V(\Gamma)$ and let $J_u$ denote the hyperplane containing the clique $G_u$. The cliques in $J_u$ are the cosets $gG_u$ where $g \in \langle \mathrm{star}(u) \rangle$. Consequently, the carrier $N(J_u)$ coincides with the subgraph $\langle \mathrm{star}(u) \rangle$ and splits at the Cartesian product $G_u \times \langle \mathrm{link}(u) \rangle$; and the stabiliser of $J_u$ in $\Gamma \mathcal{G}$ coincides with the subgroup $\langle \mathrm{star}(u) \rangle$.
\end{lemma}

\noindent
Recall that, given a graph $\Gamma$ and a vertex $u \in V(\Gamma)$, the \emph{star} (resp.\ the \emph{link}) of $u$ is the subgraph induced by $u$ and its neighbours (resp.\ by its neighbours). 

\medskip \noindent
Of course, in the quasi-median Cayley graph given by Theorem~\ref{thm:GPandQM}, vertex-groups are bounded. Roughly speaking, it is possible to ``unfold'' them thanks to a system of coherent metrics, as described by the statement:

\begin{lemma}\label{lem:GPsystem}
Let $\Gamma$ be a simplicial graph and $\mathcal{G}= \{ G_u \mid u \in V(\Gamma)\}$ a collection of finitely generated groups. For every $u \in V(\Gamma)$, fix a finite generating set $S_u$ of $G_u$ and let $|\cdot|_u$ denote the correspond word-length. For every clique $C$ of $\mathrm{QM}(\Gamma, \mathcal{G})$, write $C=gG_u$ for some $u \in V(\Gamma)$ and $g \in \Gamma \mathcal{G}$ of minimal length, and set
$$\delta_C : (gr,gs) \mapsto |r^{-1}s|_u.$$
Then $\{(C,\delta_C) \mid C \subset X \text{ clique}\}$ is a coherent system of graph metrics. 
\end{lemma}

\noindent
We refer the reader to \cite[Section~8]{Qm} for more information. It is worth noticing that, as a graph, $\mathrm{QM}(\Gamma, \mathcal{G})$ endowed with the graph metric associated to the previous coherent system of metrics coincides with the Cayley graph of $\Gamma \mathcal{G}$ given by the generating set $\bigcup_{u} S_u$ \cite[Claim~8.24]{Qm}. The point of view given by coherent systems of metrics allows us to highlight the underlying quasi-median structure.

\subsubsection{Approximations of wreath products}\label{section:Approximations}

\noindent
In this section, we show that wreath products of finitely generated groups admit truncated presentations that can be thought, geometrically, as graphs of pointed cliques in quasi-median graphs of graph produdcts endowed with coherent systems of metrics. 

\begin{definition}
Let $A,B$ be two groups and $\Gamma$ a simplicial graph on which $B$ acts freely-transitively. The semidrect product
$$A \square_\Gamma B:= \Gamma A \rtimes B$$
is a \emph{partial wreath product}, where $B$ acts on $\Gamma A$ by permuting the vertex-groups according to its action on $\Gamma$. 
\end{definition}

\noindent
When $B$ is a complete graph, $A \square_\Gamma B$ coincides with the wreath product $A \wr B$. When $B$ has no edge, then $A \square_\Gamma B = \ast_B A \rtimes B$ is isomorphic to $A \ast B$. 

\medskip \noindent
Observe that the natural projection
$$\dag : \left\{ \begin{array}{ccc} \Gamma A & \twoheadrightarrow & \bigoplus\limits_B A \\ g & \mapsto & g^\dag \end{array} \right.$$
induces a surjective morphism
$$\pi_\Gamma : \left\{ \begin{array}{lcl} A \square_\Gamma B & \twoheadrightarrow & A \wr B \\ (g,b) & \mapsto & (g^\dag, b) \end{array} \right..$$
Alternatively, one can think of $A \square_\Gamma B$ as a group obtained by truncating a presentation of $A \wr B$, and then $\pi_\Gamma$ is the natural corresponding projection. More precisely, the wreath product $A \wr B$ admits
$$\langle A_b \ (b \in B), B \mid [A_b,A_c]=1 \text{ and } bA_cb^{-1}=A_{bc} \text{ for all } b,c \in B \rangle$$
as a relative presentation, where $A_b$ are copies of $A$ and where $[H,K]=1$ is a shorthand for ``$[h,k]=1$ for all $h \in H$, $k \in K$''. Because $B$ acts freely-transitively on the vertices of $\Gamma$, we know that $\Gamma$ is a Cayley graph of $B$. Let $S$ be the corresponding generating set (which may not be finite). Then
$$\langle A_b \ (b \in B), B \mid [A_b,A_{bs}]=1 \text{ and } bA_cb^{-1}=A_{bc} \text{ for all } b,c \in B, s \in S \rangle$$
is a relative presentation of $A \square_\Gamma B$ obtained by truncating the previous presentation of $A \wr B$. Loosely speaking, we removed commutation relations. Notice that, when one adds to $S$ successively all the elements of $B$, then one adds to the presentation of $A \square_\Gamma B$ successively the missing commutation relation. 

\medskip \noindent
The connection between partial wreath products and our Embedding Theorem is made explicit by the following well-known fact:

\begin{lemma}
Let $G$ be a group, $S \subset G$ a finite generating set, and $\langle S \mid r_1,r_2, \ldots \rangle$ an infinite presentation of $G$. For every $n \geq 1$, let $G_n$ denote the group given by the presentation $\langle S \mid r_1, \ldots, r_n \rangle$. For every coarsely simply connected graph $Z$ and every coarse embedding $\rho : Z \to G$, there exist an $n \geq 1$ and a coarse embedding $\eta : Z \to G_n$ such that $\rho= \pi \circ \eta$ where $\pi$ is the canonical projection $G_n \to G$. Moreover, the parameters of $\eta$ only depends on $G$, $S$, $Z$, and the parameters of $\rho$. \qed
\end{lemma}

\noindent
As an immediate consequence of the previous discussion, we deduce that:

\begin{cor}\label{cor:ApproxWreath}
Let $A,B$ be two finitely generated groups, $Z$ a coarsely simply connected graph, and $\rho : Z \to A \wr B$ a coarse embedding. Then there exist a locally finite Cayley graph $\Gamma$ of $B$ and a coarse embedding $\eta : Z \to A \square_\Gamma B$ such that $\pi_\Gamma \circ \eta = \rho$. Moreover, the parameters of $\eta$ only depend on $A$, $B$, $Z$, and the parameters of $\rho$. 
\end{cor}

\noindent
An element $(g,b)$ of $A \square_\Gamma B$ is the data of an element $g \in \Gamma A$ and an element $b \in B$. The element $g$ can be thought of as a vertex in the quasi-median graph $\mathrm{QM}(\Gamma,A)$. According to Lemma~\ref{lem:QMcliques}, the cliques of $\mathrm{QM}(\Gamma,A)$ containing $g$ are the cosets $g \langle u \rangle$, where $u \in V(\Gamma)$. But $V(\Gamma)=B$, so we can use the element $b$ to choose the clique $g \langle b \rangle$. Thus, to each element $(g,b)$ of $A \square_\Gamma B$ can be associated a pointed clique $(g \langle b \rangle, g)$ of $\mathrm{QM}(\Gamma,A)$. 

\medskip \noindent
Let us describe how the pointed clique $(g \langle b \rangle, g)$ evolves when $(g,b)$ moves in the Cayley graph of $A \square_\Gamma B$. Here, $A \square_\Gamma B$ is endowed with a generating set obtained as the union of a finite generating set of $A$ (identified with the copy $A_1$ in $A \square_\Gamma B$) with the generating set of $B$ given by $\Gamma$ (which we assume to be finite). 
\begin{itemize}
	\item Given a generator $s$ of $A$, we have $(g,b) \cdot s = (g s^b, b)$ where $s^b$ corresponds to the element $s$ in the copy $A_b$ of $A$. The new pointed clique is then $(g s^b \langle b \rangle, gs^b)= (g \langle b \rangle, gs^b)$. Thus, the clique remains the same but the vertex is modified. When identifying the clique $g \langle b \rangle$ with $B$, this amounts to replacing a vertex with a neighbour in the Cayley graph.
	\item Given a generator $s$ of $B$, we have $(g,b) \cdot s = (g,bs)$. The new pointed clique is then $(g \langle bs \rangle, g)$. Thus, the vertex remains the same but he clique is modified. Notice that, because $b$ and $bs$ are adjacent in the Cayley graph $\Gamma$ of $B$, the cliques $g \langle b \rangle$ and $g \langle bs \rangle$ span a prism (and, conversely, the fact that $g \langle b\rangle$ and $g \langle bs \rangle$ span a prism implies that $b$ and $bs$ must be adjacent in $\Gamma$).
\end{itemize}
It follows from this description that the Cayley graph of $A \square_\Gamma B$ can be described as the graph of pointed cliques of the quasi-median graph $\mathrm{QM}(\Gamma,A)$ endowed with the coherent system of metrics given by Lemma~\ref{lem:GPsystem}.

\subsubsection{Proof of the Embedding Theorem}

\noindent
We are finally ready to prove our main theorem regarding the large-scale geometry of wreath products.

\begin{proof}[Proof of Theorem \ref{thm:FullEmbeddingThm}.]
According to Corollary \ref{cor:ApproxWreath}, there exist a locally finite Cayley graph $\Gamma$ of $B$ and a coarse embedding $\eta : Z \to A \square_\Gamma B$ such that $\rho = \pi_\Gamma \circ \eta$. 

\centerline{
\xymatrix{
Z \ar[rrd]^\rho \ar[d]_\eta & & \\ A \square_\Gamma B \ar[rr]^{\pi_\Gamma} & & A \wr B 
}}

\noindent
In the following, when thought of geometrically, $A \square_\Gamma B$ will be identified with the graph of pointed cliques $\mathrm{PC}:= \mathrm{PC}(\mathrm{QM}(\Gamma,A), \mathcal{M})$ of the quasi-median graph $\mathrm{QM}:=\mathrm{QM}(\Gamma,A)$ endowed with a coherent system of graph metrics $\mathcal{M}$ as described in Section~\ref{section:Approximations}. 

\medskip \noindent
Our first observation is that, given a fibre $F$ in $\mathrm{QM}$, the pointed cliques in $\eta(Z) \subset \mathrm{PC}$ lie mostly in a single piece delimited by $F$. More precisely:

\begin{claim}\label{claim:Orientation}
Let $F$ be a fibre in $\mathrm{QM}$ and let $B_F$ denote the corresponding bulkhead in $\mathrm{PC}$. There exist a zone $P$ delimited by $F$ and a constant $D \geq 0$ that depends only on $A$, $B$, $Z$, $\Gamma$, and the parameters of $\rho$ such that $\eta(Z) \subset \mathrm{PC}$ contains pointed cliques in $P$ arbitrarily far away from $B_F$ but every pointed clique of $\eta(Z)$ in $(P \cup F)^c$ lies at distance $\leq D$ from $B_F$. 
\end{claim}

\noindent
Let $J$ denote the hyperplane having $F$ as a fibre. By definition, we have
$$B_F= \{ (C,p) \text{ pointed clique} \mid C \subset J, p \in C \cap F\}.$$
Notice that, writing $J=gJ_u$ for some $u \in V(\Gamma)$ and $g \in \Gamma A$ satisfying $u \notin \mathrm{tail}(g)$ and writing $F=gh \langle \mathrm{link}(u) \rangle$ for some $h \in G_u$, we have
$$\pi_\Gamma(B_F) = \left\{ (p^\dag, u) \mid p \in gh \langle \mathrm{link}(u) \rangle \right\}=\left\{ (g^\dag h^\dag  c,u) \mid \mathrm{supp}(c) \subset \mathrm{link}(u) \right\}.$$
Thus, $\pi_\Gamma(B_F)$ is contained in $\mathcal{B} \left(\mathrm{link}(u),k \right)$ where $k:=g^\dag h^\dag$. Actually, it is quasi-dense in $\mathcal{B} \left(\mathrm{link}(u),k \right)$. Indeed, given a point in $\mathcal{B} \left(\mathrm{link}(u),k \right)$, say $(kc,v)$ where $c \in A^{\mathrm{link}(u)}$ and $v \in \mathrm{link}(u)$, the point $(kc,u)$ belongs to $\pi_\Gamma(B_F)$ and lies at controlled distance from $(kc,v)$. (More precisely, when $A\wr B$ is endowed with a finite generating set obtained from the union of a generating set of $A$ and the generating set of $B$ associated to $\Gamma$, then the distance between $(kc,u)$ and $(kc,v)$ in $A \wr B$ is $\leq 1$.) Therefore, $\pi_\Gamma(B_F)$ and $\mathcal{B}(\mathrm{link}(u),k)$ coarsely coincide.

\medskip \noindent
But we know that $\pi_\Gamma$ restricts to a coarse embedding on $\eta(Z)$ and that $\mathcal{B}_{\mathrm{deg}(\Gamma)}$ does not coarsely separate $\rho(Z)= \pi_\Gamma(\eta(Z))$. We deduce from Lemmas~\ref{lem:BulkSep} and~\ref{lem:SepButNotCoarse} that $B_F$ contains one of its zones in some neighbourhood (whose size does not depend on $F$), while the other zone contains points arbitrarily far away from $B_F$. This concludes the proof Claim~\ref{claim:Orientation}. 

\medskip \noindent
For every fibre $F$ of $\mathrm{QM}$, we denote by $F^+$ the zone delimited by $F$ as given by Claim~\ref{claim:Orientation}. 

\begin{claim}\label{claim:Vertex}
The intersection $\bigcap\limits_{\text{$F$ fibre}} F^+$ is non-empty and reduced to a single vertex.
\end{claim}

\noindent
Our goal is to apply Lemma~\ref{lem:InterSectors}.

\medskip \noindent
Let $F_1,F_2$ be two fibres of two hyperplanes $J_1$ and $J_2$. If $J_1$ and $J_2$ are transverse, then clearly $F_1^+$ and $F_2^+$ intersect. Next, if $J_1$ and $J_2$ are not transverse and if $F_1^+$ contains $F_2$, then clearly $F_1^+$ and $F_2^+$ intersect. Finally, if $J_1$ and $J_2$ are not transverse and if $F_1^+$ does not contain $F_2$, then $F_2^+$ must be the piece delimited by $F_2$ that contains $J_1$ as a direct consequence of our choice. It follows that $F_1^+$ and $F_2^+$ intersect. Thus, the first item of Lemma~\ref{lem:InterSectors} is satisfied.

\medskip \noindent
Next, assume for contradiction that $F_1^+ \supsetneq F_2^+ \supsetneq \cdots$ for some collection of fibres $F_1,F_2, \ldots$ of hyperplanes $J_1,J_2, \ldots$ in $\mathrm{QM}$. Fix a pointed clique $(C,x) \in \eta(Z)$ such that $C \subset F_1^+$. Because two vertices in $\mathrm{QM}$ are always separated by only finitely many hyperplanes, there must exist some $r \geq 1$ such that $C$ is disjoint from $F_r^+$. Given an $s >r$, we know from Claim~\ref{claim:Orientation} that $(C,x)$ must lie at distance $\leq D$ from the bulkhead $B_s$ associated to $F_{s}$. But
$$d((C,x),B_s)) \geq d(x,J_s) \geq s-r$$
because $J_r, \ldots, J_{s-1}$ separates $x$ from $J_s$. Therefore, we get a contradiction if $s$ is taken sufficiently large, namely if $s>r+D$. Thus, the second item of Lemma~\ref{lem:InterSectors} is satisfied. 

\medskip \noindent
Finally, let $J$ be a hyperplane. Assume for contradiction that, for every fibre $F$ of $J$, $F^+$ is the cosector delimited by $F$. This implies that, for every sector $S$ delimited by $J$, every pointed clique of $\eta(Z)$ contained in $S$ lies at distance $\leq D$ from the bulkhead of the fibre of $J$ in $S$. Consequently, $\eta(Z)$ is contained in the $D$-neighbourhood of 
$$B:=\{ (C,p) \text{ pointed clique} \mid C \subset J, p \in C \}.$$
A fortiori, $\rho(Z)= \pi_\Gamma(\eta(Z))$ must be contained in a neighbourhood of $\pi_\Gamma(B)$. But, writing $J=gJ_u$ for some $u \in V(\Gamma)$ and $g \in \Gamma A$, we have
$$\pi_\Gamma(B)= \left\{ (p^\dag, u) \mid p \in g \langle \mathrm{star}(u) \rangle \right\} = \left\{ (g^\dag c, u) \mid \mathrm{supp}(c) \subset \mathrm{star}(u) \right\}.$$
As before, we deduce that $\pi_\Gamma(B)$ coarsely coincides with $\mathcal{B}(\mathrm{star}(u),g)$. It follows that $\rho(Z)$ is contained in a neighbourhood of $\mathcal{B}(\mathrm{star}(u),g)$, which is impossible by assumption. Thus, the third item of Lemma~\ref{lem:InterSectors} is satisfied. 

\medskip \noindent
We conclude that Lemma~\ref{lem:InterSectors} applies, proving Claim~\ref{claim:Vertex}.

\medskip \noindent
Let $x$ denote the vertex of $\mathrm{QM}$ given by Claim \ref{claim:Vertex} and set 
$$\mathcal{H}:=\{(C,x) \mid \text{$C$ clique} \} \subset \mathrm{PC}.$$
Notice that $\pi_\Gamma(\mathcal{H})$ coincides with the coset $x^\dag B$ in $A \wr B$. Therefore, in order to conclude the proof of our theorem, it suffices to show that every vertex in $\eta(Z)$ is at (uniformly) bounded distance from $\mathcal{H}$. 

\medskip \noindent
So let $(C,y)$ be an arbitrary pointed clique in $\eta(Z)$. Fix a geodesic in $\mathrm{QM}$ from $y$ to $x$, and let $Q$ denote the clique containing the last edge of this geodesic. Also, let $J$ denote the hyperplane containing $Q$ and $F$ its fibre containing $x$. By definition of $x$, $F^+$ is the sector delimited by $J$ containing $x$; and, by construction, $(C,y)$ is contained in the union of $F$ with the cosector delimited by $F$. Therefore, $(C,y)$ must lie at distance $\leq D$ from the bulkhead associated to $F$. In other words, there exists a pointed clique $(K,z)$ with $K \subset J$ and $z \in F$ such that the distance between $(K,z)$ and $(C,y)$ in $\mathrm{PC}$ is $\leq D$. 

\begin{center}
\includegraphics[width=0.4\linewidth]{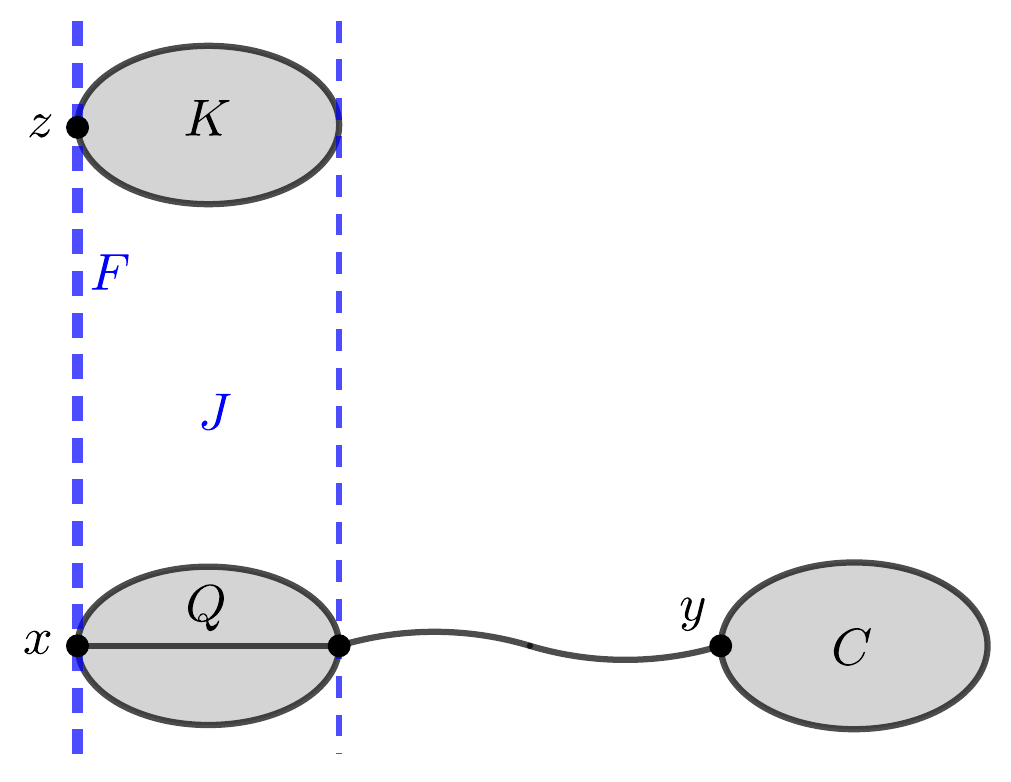}
\end{center}

\noindent
According to Lemma~\ref{lem:PointedHyp}, the distance between $(K,z)$ and $(Q,x)$ is bounded above by $3 \delta(x,z)+1$. Therefore,
$$\begin{array}{lcl} d_\mathrm{PC}((C,y),\mathcal{H}) & \leq & d_\mathrm{PC}((C,y),(Q,x)) \leq d_\mathrm{PC}((C,y),(K,z)) + d_\mathrm{PC}((K,z),(Q,x)) \\ \\ & \leq & D+ 3 \delta (x,z)+1. \end{array}$$
It remains to show that $\delta(x,z)$ is controlled. By triangle inequality, we have $\delta(x,z) \leq \delta(x,y)+\delta(y,z)$. On the one hand, we know from Lemma~\ref{lem:PCdistLow} that
$$\delta(y,z) \leq d_\mathrm{PC} ((C,y), (K,z)) \leq D.$$
On the other hand, it follows from Facts~\ref{fact:FewHyp} and~\ref{fact:ThinHyp} below that  $\delta(x,y) \leq \mathrm{clique}(\Gamma) \cdot D^2$.  We finally deduce that
$$d_\mathrm{PC}((C,y),\mathcal{H}) \leq 3D^2 \mathrm{clique}(\Gamma) + 4D+1.$$
Thus, we have proved that $\eta(Z)$ lies in a controlled neighbourhood of $\mathcal{H}$, as desired.

\begin{fact}\label{fact:FewHyp}
There exist at most $\mathrm{clique}(\Gamma) \cdot D$ hyperplanes separating $x$ and $y$.
\end{fact}

\noindent
Let $J_1, \ldots, J_n$ be a maximal collection of pairwise non-transverse hyperplanes separating $x$ and $y$. Assume that $J_i$ separates $J_{i-1}$ and $J_{i+1}$ for every $2 \leq i \leq n-1$ and that $J_1$ separates $x$ and $J_2$. Notice that $x$ belongs to the carrier of $J_1$, since otherwise there would exist a hyperplane separating $x$ from $J_1$, contradicting the maximality of our collection of hyperplanes. Let $F_1$ denote the fibre of $J_1$ containing $x$. Because $J_1, \ldots, J_n$ separate $y$ from $F_1$, necessarily $(C,y)$ lies at distance $\geq n$ from the bulkhead associated to $F_1$. However, since $F_1^+$ is necessarily the sector delimited by $J_1$ containing $x$, the latter distance must be $\leq D$. Hence $n \leq D$. We deduce from Lemma~\ref{lem:CubDist} (and Theorem~\ref{thm:GPandQM}) that there are at most $\mathrm{clique}(\Gamma) \cdot D$ hyperplanes separating $x$ and $y$. 

\begin{fact}\label{fact:ThinHyp}
For every hyperplane $H$ separating $x$ and $y$, $\delta_H(x,y) \leq D$. 
\end{fact}

\noindent
Let $F$ denote the fibre of $H$ given by the sector delimited by $H$ containing $x$. Because $H$ separates $x$ and $y$, necessarily $F^+$ coincides with the sector delimited by $H$ containing $x$. Fix a pointed clique $(P,a)$ in the bulkhead of $F$ at distance $\leq D$ from $(C,y)$. Then
$$\delta_H(x,y)= \delta_H(a,y) \leq d_\mathrm{PC}((P,a),(C,y)) \leq D,$$
where the first equality is justified by the fact that $x$ and $a$ belong to the same sector delimited by $H$. 
\end{proof}

\subsection{Applications to coarse geometry}

\subsubsection{Aptolic embeddings}\label{section:Aptolic}

\noindent
In \cite{GTbig} is introduced a notion of maps between wreath products that, roughly speaking, preserve the lamplighter structure. More precisely:

\begin{definition}
Let $A_1,A_2,B_1,B_2$ be four finitely generated groups. A coarse embedding (resp.\ a quasi-isometry) $\varphi : A_1 \wr B_1 \to A_2 \wr B_2$ is \emph{aptolic} if there exist two maps $\alpha : \bigoplus_{B_1} A_1 \to \bigoplus_{B_2} A_2$ and $\beta : B_1 \to B_2$ such that 
$$\varphi : (a,b) \mapsto (\alpha(a),\beta(b)), \ (a,b) \in \left( \bigoplus_{B_1} A_1 \right) \rtimes B_1$$
where $\alpha$ is injective (resp.\ bijective) and where $\beta$ is a coarse embedding (resp.\ a quasi-isometry).
\end{definition}

\noindent
A fundamental result is that quasi-isometries between wreath products are often aptolic. For instance:

\begin{thm}\label{thm:QIaptolic}
Let $A_1,A_2,B_1,B_2$ be four finitely generated groups, with $B_1,B_2$ infinite, and let $\varphi : A_1 \wr B_1 \to A_2 \wr B_2$ be a quasi-isometry. Assume that there is a constant $D \geq 0$ such that $\varphi$ (resp.\ its quasi-inverse) sends every $B_1$-coset (resp. $B_2$-coset) in the $D$-neighbourhood of some $B_2$-coset (resp.\ $B_1$-coset). Then $\varphi$ lies at finite distance from an aptolic quasi-isometry $\psi$.
\end{thm}

\noindent
The statement is proved for wreath products with finite lamp-groups in \cite[Theorem~4.3]{GTbig} and follows from Theorem~\ref{thm:FullEmbeddingThm} combined with \cite[Corollary~6.11 and Lemma~6.12]{GTlike} for infinite lamp-groups.

\medskip \noindent
Coarse embeddings between wreath products, however, may not be aptolic. Obvious counterexamples are given by embeddings of the form $B \wr C \hookrightarrow A \wr (B \wr C)$. Nevertheless, we have the following statement, which shows that there are essentially no other counterexamples.

\begin{thm}\label{thm:CoarseAptolic}
Let $A_1,A_2,B_1,B_2$ be four finitely generated groups and $\rho : A_1 \wr B_1 \to A_2 \wr B_2$ a coarse (resp.\ quasi-isometric) embedding. Assume that there is a constant $D \geq 0$ such that $\rho$ sends every $B_1$-coset in the $D$-neighbourhood of some $B_2$-coset. Either $\rho$ is aptolic or it induces a coarse (resp.\ quasi-isometric) embedding $A_1 \wr B \to B_2$ for some graph of bounded degree $B$ quasi-isometric to $B_1$.
\end{thm}

\noindent
Here, given a pointed graph $(X,o)$ and a graph $Y$, the \emph{wreath product $(X,o) \wr Y$} refers to the graph 
\begin{itemize}
	\item whose vertices are the pairs $(\varphi : Y \to X, p \in Y)$ where $\varphi(y)=o$ for all but finitely many $y \in Y$;
	\item whose edges connect $(\varphi_1,p_1)$ and $(\varphi_2,p_2)$ whenever either $\varphi_1=\varphi_2$ and $p_1,p_2$ are adjacent in $Y$ or $\varphi_1,\varphi_2$ differ only at $p_1=p_2$ where they take adjacent values in~$X$. 
\end{itemize}
By abuse of notation, in the statement above, we wrote $A_1 \wr B$ where $A_1$ is a finitely generated group and $B$ a graph. Here, $A_1$ is identified with one of its Cayley graphs (given by a finite generating set) and the neutral element is chosen as the basepoint. 

\medskip \noindent
Theorem~\ref{thm:CoarseAptolic} follows from \cite[Theorem~6.20 and Lemma~6.12]{GTlike}.

\subsubsection{Applications to coarse embeddings}\label{section:AppCoarseEmb}

\noindent
As an application of the aptolicity provided by Theorem~\ref{thm:CoarseAptolic}, let us prove that:

\begin{thm}\label{thm:AppliCoarse}
Let $A_1,A_2$ be two finitely generated groups with subexponential growth and let $B_1,B_2$ be two finitely presented groups from $\mathfrak{M}_{exp}$. Then every coarse (resp.\ quasi-isometric) embedding $\rho : A_1\wr B_1 \to A_2 \wr B_2$ induces a coarse (resp.\ quasi-isometric) embedding $B_1 \hookrightarrow B_2$. Moreover, if $B_2$ does not contain a coarsely (resp.\ quasi-isometrically) embedded copy of $A_1 \wr B$ for some graph of bounded degree $B$ quasi-isometric to $B_1$, then $\rho$ induces a coarse embedding $A_1 \hookrightarrow A_2^k$ for some $k \geq 1$. 
\end{thm}

\noindent
In addition to Theorems~\ref{thm:FullEmbeddingThm} and~\ref{thm:CoarseAptolic}, we need the following observation:

\begin{lemma}\label{lem:AptolicEmbPowerLamp}
Let $A_1,A_2,B_1,B_2$ be four finitely generated groups. If $\rho : A_1 \wr B_1 \to A_2 \wr B_2$ is a coarse (resp.\ quasi-isometric) embedding, then it induces a coarse (resp.\ quasi-isometric) embedding $A_1 \to A_2^k$ for some $k \geq 1$.
\end{lemma}

\begin{proof}
There exist an injective map $\alpha : \bigoplus_{B_1} A_1 \to \bigoplus_{B_2}A_2$ and a coarse embedding $\beta : B_1 \to B_2$ such that 
$$\rho : (c,b) \mapsto (\alpha(c), \beta(b)), \ (c,b) \in A_1 \wr B_1.$$
In order to conclude that $\rho$ induces a coarse embedding $A_1 \hookrightarrow A_2^k$ for some $k \geq 1$, it suffices to show that there exists some $R \geq 1$ such that
$$\rho \left( \{ (c,1) \mid \mathrm{supp}(c)= \{1\} \} \right) \subset \left\{ (c,p) \mid \mathrm{supp}(c) \cup \{p\} \subset B(\beta(1),R) \right\}.$$
Indeed, the left-hand side is the image under $\rho$ of the canonical copy of $A_1$ in $A_1 \wr B_1$ and the right-hand side is quasi-isometric to $A_2^N$ where $N$ is the cardinality of $B(\beta(1),R)$. Assume for contradiction that there is no such inclusion, i.e.\ for every $R \geq 1$ one can find some $(c,1) \in A_1$ such that $\rho(c,1)=(\alpha(c),\beta(1))$ satisfies $\mathrm{supp}(\alpha(c)) \nsubseteq B(\beta(1),R)$. But $\rho(A_1)$ is coarsely connected, and a coarse path connecting $(\alpha(1),\beta(1))$ to $(\alpha(c),\beta(1))$ where $\alpha(c)$ has a support not contained in a large ball $B(\beta(1),R)$ (compared to the size of the support of $\alpha(1)$) cannot have a second coordinate constant to $\beta(1)$.
\end{proof}

\begin{proof}[Proof of Theorem~\ref{thm:AppliCoarse}.]
Our assumptions on $A_1,A_2$ and $B_1,B_2$ imply that Theorem~\ref{thm:FullEmbeddingThm} applies and shows that, for some constant $D \geq 0$, $\rho$ always sends a $B_1$-coset in the $D$-neighbourhood of some $B_2$-coset. In particular, $\rho$ induces a coarse embedding $B_1 \hookrightarrow B_2$. If moreover $B_2$ does not contain a coarsely (resp.\ quasi-isometrically) embedded copy of $A_1 \wr B$ for some graph of bounded degree $B$ quasi-isometric to $B_1$, then Theorem~\ref{thm:CoarseAptolic} implies that $\rho$ is (up to finite distance) aptolic. Then the desired conclusion follows from Lemma~\ref{lem:AptolicEmbPowerLamp}. 
\end{proof}

\begin{ex}
Let $N_1,N_2$ be two finitely generated groups of subexponential growth and $H_1,H_2$ two uniform lattices in $\mathbb{H}_\mathbb{R}^{n_1}, \mathbb{H}_\mathbb{R}^{n_2}$ respectively with $n_1,n_2 \geq 3$. If there exists a quasi-isometric embedding $N_1 \wr H_1 \to N_2 \wr H_2$, then $n_2 \geq n_1$. Moreover, if $N_2$ is virtually nilpotent, then so is $N_1$. Indeed, there is no quasi-isometric embedding from a wreath product $N_1 \wr H$ to $H_2$ for some graph of bounded degree $H$ quasi-isometric to $H_1$ because the asymptotic cones of $H_2$ have finite topological dimension whereas those of such a $N_1 \wr H$ have infinite topological dimension (see for instance \cite[Proposition~6.25]{GTlike}). Thus, it follows from Theorem~\ref{thm:AppliCoarse} that $N_1$ quasi-isometrically embeds into a power of $N_2$, so it must have polynomial growth. 
\end{ex}

\subsubsection{Applications to quasi-isometric classifications}\label{section:AppQI}

\noindent
As an application of the aptolicity provided by Theorem~\ref{thm:QIaptolic}, let us prove that:

\begin{thm}\label{thm:AppliQI}
Let $A_1,A_2$ be two finitely generated groups with subexponential growth and let $B_1,B_2$ be two finitely presented groups from $\mathfrak{M}_{exp}$.
Then every quasi-isometry $A_1\wr B_1 \to A_2 \wr B_2$ induces a quasi-isometry $B_1 \to B_2$, and there are quasi-isometric embeddings $A_1 \hookrightarrow A_2^p$ and $A_2 \hookrightarrow A_1^q$ for some $p,q \geq 1$. 
\end{thm}

\begin{proof}
Our assumptions on $A_1,A_2$ and $B_1,B_2$ imply that Theorem~\ref{thm:FullEmbeddingThm} can be applied to quasi-isometric embeddings $B_1 \to A_2 \wr B_2$ and $B_2 \to A_1 \wr B_1$, so it follows from Theorem~\ref{thm:QIaptolic} that every quasi-isometry $\varphi : A_1\wr B_1 \to A_2 \wr B_2$ is (up to finite distance) aptolic. The fact that $\varphi$ induces a quasi-isometric embedding $A_1 \hookrightarrow A_2^p$ for some $p \geq 1$ and that a quasi-inverse of $\varphi$ induces a quasi-isometric embedding $A_2 \hookrightarrow A_1^q$ for some $q \geq 1$ follows from Lemma~\ref{lem:AptolicEmbPowerLamp}. 
\end{proof}

\begin{remark}
It is worth noticing that, in Theorem~\ref{thm:AppliQI}, one cannot replace $A_1 \overset{\mathrm{QI}}{\hookrightarrow} A_2^p$ and $A_2 \overset{\mathrm{QI}}{\hookrightarrow} A_1^q$ for some $p,q \geq 1$ with $A_1^p \overset{\mathrm{QI}}{\sim} A_2^q$ for some $p,q \geq 1$. Indeed, as a consequence of Fact~\ref{fact:ExQI} below, if $H$ is a non-amenable finitely generated group, then $(\mathrm{Heis}\times \mathbb{Z}) \wr H$ and $(\mathrm{Heis} \times \mathbb{Z}^2) \wr H$ are quasi-isometric but no powers of $\mathrm{Heis}\times \mathbb{Z}$ and $\mathrm{Heis} \times \mathbb{Z}^2$ are quasi-isometric (since they never have the same asymptotic dimension).

\begin{fact}\label{fact:ExQI}
Let $A,B,L$ be three finitely generated groups with $B$ non-amenable. The groups $(A \times L) \wr B$ and $(A \times L^2) \wr B$ are quasi-isometric. 
\end{fact}

\noindent
This fact essentially follows from the proof of \cite[Proposition~3.12]{GTbig}, so we only sketch the argument and refers to \cite{GTbig} for more details.

\medskip \noindent
First, because $B$ is non-amenable, we know from \cite{MR1700742} that there exists a $2$-to-$1$ map $f : B \to B$ at finite distance from the identity. Our goal is to use $f$ in order to associate a colouring $\bar{c} : B \to A \times L^2$ to every finitely supported colouring $A\times L$. We fix an enumeration of $B$, and, for convenience, we denote by $\pi_1$ and $\pi_2$ the projections on the first and second coordinates in both $A \times L$ and $A \times L^2$. Given a $b \in B$,
\begin{itemize}
	\item set $\pi_1(\bar{c}(b)):= \pi_1(c(b))$;
	\item enumerate $f^{-1}(b)$ as $\{b_1, b_2\}$ by following the order induced by our enumeration of $X$, and set $\pi_2( \overline{c}(b)) = (\pi_2(c(b_1)), \pi_2(c(b_2)))$.
\end{itemize}
Then $(c,p) \mapsto (\bar{c},p)$ defines a quasi-isometry $(A \times L) \wr B \to (A \times L^2) \wr B$. 
\end{remark}

\noindent
By a famous result of Pansu, if there exists a quasi-isometric embedding $N_1\to N_2$ between two torsion-free nilpotent groups, then, if $N_2$ is $l$-nilpotent, then so is $N_1$. Indeed, such a quasi-isometry induces a biLipschitz embedding between the asymptotic cones, which by the main result of \cite{Pansu} identify with the associated Carnot graded groups $N_1^c$ and $N_2^c$, equipped with left-invariant subfinsler metrics. Second, by \cite[Theorem 2]{MR0979599}, this implies that $N_1^c$ admits an injective continuous morphism to $N_2^c$.
Since the nilpotent step of a torsion-free nilpotent group equals that of its associated Carnot, our statement follows. This yields the following rigidity statement for quasi-isometries between lamplighters with lamps of polynomial growth.

\begin{cor}
Let $N_1,N_2$ be two finitely generated nilpotent groups and $H_1,H_2$ two finitely presented groups from $\mathfrak{M}_{exp}$. If $N_1 \wr H_1$ and $N_2 \wr H_2$ are quasi-isometric, then $N_1,N_2$ have the same nilpotent step and $H_1,H_2$ are quasi-isometric. 
\end{cor}

\begin{proof}
It follows from Theorem~\ref{thm:AppliQI} that $H_1$ and $H_2$ are quasi-isometric and that there are quasi-isometric embeddings $N_1 \to N_2^p$ and $N_2 \to N_1^q$ for some $p,q \geq 1$. Because the nilpotent step is monotone with respect to quasi-isometric embeddings
and that it remains the same by taking powers, we conclude that $N_1$ and $N_2$ must have the same nilpotent step. 
\end{proof}

\section{Open questions}\label{section:questions}
We end this paper with a list of open questions related to the notion of coarse separation.

\subsection{Stability properties of $\mathfrak{M}_{exp}$}
\begin{question}
Given a space  $X\in \mathfrak{M}_{exp}$ and $Y$ an arbitrary metric space, do we have $X\times Y\in \mathfrak{M}_{exp}$? 
\end{question}

\noindent
Of course the same question is relevant for similar conditions, such as belonging to $\mathfrak{M}_{d}$ for some $d$. Note that this is the case of all the examples treated in this article. Among groups, one may ask for an even stronger stability:
\begin{question}
Let $G=H\rtimes K$ be a finitely generated group which splits as a semi-direct product. If $K\in  \mathfrak{M}_{exp}$ (resp.\  $\mathfrak{M}_{d}$ for some $d$), does the same hold for $G$?
\end{question}

\medskip \noindent

\subsection{$\mathfrak{M}_{exp}$ and cuts of balls}

The notion of coarse separation is intimately related to that of cut of finite subsets, as illustrated by the proofs of Theorems  \ref{thmIntro:Symmetric/buildings}, \ref{thmIntro:horocyclic products} and \ref{thmIntro:nilpot}.
In particular, one can define a ``finitary'' version of belonging to $\mathfrak{M}_{exp}$ for a bounded degree graph $X$:  there exists $c>0$ such that 
for all $r\geq 1$, the cut of balls of radius $r$ is at least $ce^{cr}$. Since balls are persistent families (see Definition \ref{persistent}), this readily implies that $X$ belongs to $\mathfrak{M}_{exp}$. We may ask for a converse, at least for vertex-transitive graphs.

\begin{question}
Let  $X\in \mathfrak{M}_{exp}$ be a vertex-transitive graph. Is it true that there exists $c>0$ and $\delta>0$ such that $\cut(B(x,r)\geq ce^{cr}$ for infinitely many $r\in \mathbb N$? 
\end{question}

\noindent
Of course, a similar question makes sense for $\mathfrak{M}_{d}$ (we leave the details to the reader).
This draws a possible connection between coarse separation and separation profile. In particular, a positive answer to the previous question would imply a positive answer to the following one. 
\begin{question}
Let  $X\in \mathfrak{M}_{exp}$ be a vertex-transitive graph. Is there some $\alpha>0$ such that $\textnormal{sep}_X(n)\geq n^{\alpha}$ for infinitely many $n\in \mathbb N$?
\end{question}

\subsection{On groups belonging to $\mathfrak{M}_{exp}$}

In this article, we show that a large class of hyperbolic groups belong to  $\mathfrak{M}_{exp}$. It is tempting to conjecture the following generalisation.

\begin{conj}
A non-elementary hyperbolic group is coarsely separable by subspaces of subexponential growth if and only if it splits over a two-ended group.
\end{conj}

\noindent
We tend to believe that the class of groups or spaces belonging to $\mathfrak{M}_{exp}$ is huge. In particular, among groups with exponential growth, the only examples we know that do not belong to $\mathfrak{M}_{exp}$ virtually  split over a subgroup with subexponential relative growth. 
It would be interesting to determine to which extend this observation holds in full generality. In this direction, one can ask:  

\begin{question}
Are all finitely generated amenable groups with exponential growth in~$\mathfrak{M}_{exp}$? 
\end{question}

\noindent
We conjecture that this holds for elementary amenable groups.  
A tractable case would be polycyclic groups. Recall that these are virtually uniform lattices in connected solvable Lie groups. Hence the right setting would be to prove this for all unimodular connected Lie groups. More generally, we address the following question.

\begin{question}
Which connected Lie group belong to $\mathfrak{M}_{exp}$? 
\end{question}

\subsection{Coarse separation and median spaces}

\noindent
Here is a question on coarse embeddings to groups acting on median spaces (or equivalently, on CAT(0) cube complexes). 

\begin{question}
Let $G$ be a finitely generated group acting on a median graph $X$. Let $Z$ be a geodesic metric space and $\rho : Z \to G$ a coarse embedding. If $\rho(Z)$ is not coarsely separated by the cosets of hyperplane-stabilisers and is not contained in the neighbourhood of a coset of a hyperplane-stabiliser, is it true that $\rho(Z)$ is contained in a neighbourhood of a coset of a vertex-stabiliser?
\end{question}

\noindent
The question is motivated by the case of trees treated by Theorem~\ref{sep in graphs of groups general} and by the connection between multi-ended pairs of groups (i.e.\ coarsely separating subgroups) and actions on median graphs exhibited in \cite{MR1347406}.

\subsection{The case of right-angled Artin groups}

We now turn to coarse separation properties of right-angled Artin groups. In general, characterizing which ones belong to $\mathfrak{M}_{exp}$ is an interesting and challenging question. Here is a more specific question in this direction:

\begin{question}
Let $\Gamma$ be a finite graph. Is it true that the right-angled Artin group $A(\Gamma)$ is coarsely separable by a subspace of subexponential growth if and only if $\Gamma$ contains a separating clique?
\end{question}

\noindent
Recall that we proved that $\mathbb{H}^3$ does not coarsely embed into a right-angled Artin group over a finite chordal graph. This motivates the following question.

\begin{question}
Let $\Gamma$ be a finite chordal graph. Does there exist a coarse embedding from $\mathbb{H}^2$ to the right-angled Artin group $A(\Gamma)$?
\end{question}

\noindent
It is worth mentioning that \cite{MR2422070} already proves that such right-angled Artin groups do not contain a subgroup isomorphic to the fundamental group of a closed surface of genus $\geq 2$, motivating a positive answer to our question. 

\addcontentsline{toc}{section}{References}

\bibliographystyle{alpha}
{\footnotesize\bibliography{CoarseSepWreath}}

\Address

%

\end{document}